\documentclass[preprint,12pt]{elsarticle}

\usepackage{amsmath}
\usepackage{amssymb}
\usepackage{amsthm}
\usepackage{graphicx}
\usepackage{array}
\usepackage{booktabs} 
\usepackage{subcaption}
\usepackage{caption}  
\usepackage{cases}
\usepackage{xcolor}
\usepackage{tabularx}
\usepackage{pgffor}
\usepackage[colorlinks=true,linkcolor=blue,citecolor=blue,urlcolor=blue]{hyperref}

\journal{}

\newtheorem{theorem}{Theorem}[section] 
\newtheorem{proposition}{Proposition}[section]
 
\newtheorem{remark}{Remark}[section] 
\newcommand{\la}{\langle}
\newcommand{\ra}{\rangle}
\newcommand{\ve}{\vert}

\newcommand{\sech}{\text{sech}}
\newcommand{\curl}{\text{curl\,}}

\begin{document}

\begin{frontmatter}



\title{A Structure-Preserving Scheme for the Time-Dependent Ginzburg--Landau Model with BCS Gap Coupling}

\author[label1]{Boyi Wang} 
\author[label2]{Saurav Shenoy}
\author[label2]{Daniel Fortino}
\author[label1,label2]{Long-Qing Chen}
\author[label1]{Wenrui Hao}

\affiliation[label1]{organization={Department of Mathematics},
            addressline={The Pennsylvania State University, 54 McAllister St}, 
            city={State College},
            postcode={16801}, 
            state={PA},
            country={USA}}

\affiliation[label2]{organization={Department of Materials Science and Engineering},
            addressline={The Pennsylvania State University}, 
            city={State College},
            postcode={16801}, 
            state={PA},
            country={USA}}

\begin{abstract}
We propose a structure-preserving scheme for a hybrid model that couples the time-dependent Ginzburg–Landau (TDGL) equation of superconducting vortex dynamics and the nonlinear Bardeen–Cooper–Schrieffer (BCS) gap equation. This formulation is consistent with the classical TDGL equation in the near-critical temperature, while extending the applicability of the existing TDGL model to regimes beyond the critical temperature. The resulting system poses significant computational challenges due to its nonlinear and coupled structure.

To achieve stable and reliable simulations of the vortex dynamics and accompanying morphological transitions, we develop a maximum bound preserving, energy-stable implicit–explicit (IMEX) scheme. The structure preserving properties of the scheme are rigorously established, ensuring long-time stability and physical consistency. Through two- and three-dimensional simulations, the hybrid model successfully captures the temporal and spatial formation and alignment of vortices and the suppression of superconductivity under increasing magnetic fields, demonstrating both the accuracy and robustness of the proposed computational approach.


\end{abstract}




\end{frontmatter}

\section{Introduction}
Superconductors play a central role in a wide range of technological applications, from magnetic resonance imaging (MRI) and microwave filters to superconducting radio-frequency cavities in particle accelerators \cite{Padamsee08, Seidel15}. The remarkable property of zero electrical resistance makes them ideal for energy-efficient devices, yet their performance is often limited by dissipation mechanisms associated with magnetic vortices. Understanding and controlling these phenomena is therefore critical for optimizing superconducting materials and devices.

In type-II superconductors, magnetic flux penetrates the material in the form of quantized vortices, leading to the so-called mixed state. The dynamics of these vortices—including nucleation, motion, and annihilation—govern the resistive behavior of the system and are a primary source of energy dissipation \cite{SHD09, Tinkham04}. Consequently, accurately modeling vortex behavior is essential for predicting and improving the performance of superconducting devices under applied currents and magnetic fields.

A variety of theoretical models have been developed to describe superconducting phenomena. On the macroscopic level, the Ginzburg--Landau (GL) theory provides a phenomenological description of superconductivity near the critical temperature $T_c$ \cite{Ginz50}, while the time-dependent Ginzburg--Landau (TDGL) equations extend this framework to capture the temporal evolution of the superconducting order parameter \cite{Chapman92, Tang95}. Microscopically, the Bardeen--Cooper--Schrieffer (BCS) theory \cite{Bardeen57} and Gor’kov’s derivation \cite{Gor68} connect the GL order parameter to the underlying Cooper pair condensate, providing a rigorous foundation for the macroscopic models.

The phenomenological TDGL equation originates from the Ginzburg--Landau theory of superconductivity \cite{Ginz50} and its microscopic justification via Gor'kov’s derivation from the BCS theory \cite{Gor68}. Near the superconducting critical temperature $T_c$, the TDGL equation effectively models the dynamics of $\psi$ coupled with the electromagnetic potentials, thereby capturing both the phase transition process and the temporal evolution of the superconducting state \cite{Chapman92, Tang95}. When coupled with the magnetic field, the complex order parameter $\psi$ evolves according to the dynamic law \cite{Oripov20, Tang95}, which describes the relaxation from a perturbed equilibrium state \cite{Oripov20}:
\begin{equation}
\frac{\hbar^2}{2m^*D}\left(\partial_t+i\frac{e^*}{\hbar}\phi\right)\psi = -\frac{\delta E}{\delta \psi^*},\quad\tilde\sigma\left(A_t +\nabla\phi\right)=-\frac{\delta E}{\delta A}\label{eq-evo}.
\end{equation}
with $\phi$ the electric potential and $\delta E/\delta \psi^*,\,\delta E/\delta A$ the variational derivatives of the free energy $E$.
\begin{equation}
E = \int_\Omega\left[ \frac1{2\eta_1 m^*}\left\ve \hbar\nabla\psi-ie^*A\psi\right\ve ^2 +F(\ve \psi\ve ^2)+ \frac{\ve \curl A-H\ve ^2}{2\eta_2\mu_0} \right]dx.\label{tot-ene}
\end{equation}
The second equation of \eqref{eq-evo} is the modified Ampere law from the Maxwell equation and  $\tilde{\sigma}$ represents the relaxation time. Here, the function $\psi^*=\overline{\psi}$ is the conjugate of the complex function $\psi$, and $A$ is the magnetic potential field, and the precise definitions of the parameters and physical constants are given in the next section.  When the diffusion coefficient $\eta_1, \eta_2 = 1$, the free energy \eqref{tot-ene} recovers the classical definition (see e.g. \cite{Du94b, Kop01, Oripov20, Tang95}). This framework naturally accounts for vortex nucleation, motion, and annihilation in type-II superconductors, as vortices appear as topological defects where $\vert\psi\vert=0$, and the total phase equals multiples of $2\pi$. Such vortex dynamics play a central role in determining the resistive behavior and dissipation mechanisms of superconductors under applied currents and fields.  

Comprehensive expositions of these ideas and their experimental implications can be found in the monographs \cite{Arm20, Tinkham04}. Recent work has shown that the superconducting–insulating transition observed in 2D disordered samples cannot be explained within the conventional theoretical framework \cite{Yang24}. To address this, a thermodynamic model:
\begin{equation}\label{gapeq-1}
\eta \nabla^2|\Delta(x)|=\frac{|\Delta(x)|}U+\sum_{k} \frac{|\Delta(x)|}{2\sqrt{\xi_k^2+|\Delta(x)|^2}}\sum_{\pm}\tanh(\frac{1}{2k_BT} E_{k}^{\pm}).
\end{equation} 
was formulated by adopting an energy-based approach for the complex superconducting gap $\Delta$ \cite{Yang24}, where $E^\pm_k=\sqrt{\xi_k^2+|\Delta(x)|^2}\pm p_k(x)$ are the Bogoliubov quasi-particle spectra, $\eta$ is the diffusion coefficient, and $U$ is a coupling constant. The nonlinear term in \eqref{gapeq-1} resembles the classical BCS gap equation \cite{Bardeen57,Watanabe25}:  

\begin{equation}\label{gapeq}
\Delta(k)=\sum_{k'} \frac{V_{k,k'}\Delta(k')}{E_{k'}}\tanh(\frac1{2k_BT}E_{k'}),
\end{equation}
with $E_k=\sqrt{\xi_k^2+|\Delta(k)|^2}$ and $V_{k,k'}$ denoting the effective interaction between Cooper pairs. Importantly, the microscopic formulation in \cite{Yang24}, derived within the quantum field theory framework, remains valid even as the temperature increases away from zero. In the homogeneous case without phase fluctuation ($p_k=0$), the gap equation \eqref{gapeq-1} reduces exactly to the BCS equation \eqref{gapeq}. Furthermore, Gorkov \cite{Gor68} established a connection between the GL and BCS theories by showing that the order parameter $\phi$ is proportional to the pair potential $\Delta$ in certain limiting cases. Motivated by this connection, we extend the framework to incorporate time dependence and magnetic fields.

\section{Model problem}
In this paper, we investigate a dimensionless TDGL model coupled with the magnetic field equation, whose nonlinear term is derived in analogy to the BCS gap equation. The model originates from the dimensional system \eqref{eq-evo} - \eqref{tot-ene}. After nondimensionalization, we work with the dimensionless parameters: the normalized Debye frequency \( \tilde\omega \),  the temperature \( \beta \), and the inhomogeneity \( \delta(x,t) \) (see Table.~ \ref{tab:parameters}), which capture the material properties and external effects. The resulting system is given by

\begin{subequations}\label{eq-m}
\begin{numcases}{}
\psi_t + i\phi\psi + \left(\frac{i}{\kappa}\nabla + A\right)^2\psi + f(|\psi|^2)\psi = 0,\quad x\in\Omega,\, t>0,\label{eq1a}\\
\sigma(A_t + \nabla\phi) + \curl\curl A + \frac{1}{\kappa}\text{Im}(\psi\nabla\psi^*) + A|\psi|^2 \notag\\
\phantom{xx}= \curl H,\quad x\in\Omega,\, t>0,\label{eq1b}
\end{numcases}
\end{subequations}

subject to the boundary conditions
\begin{equation}
\left(\frac{i}{\kappa}\nabla\psi + A\psi\right)\cdot n = 0,\quad n\times \curl A = n\times H,\quad x\in\partial\Omega,\, t>0,\label{bv-1}
\end{equation}
and initial conditions
\begin{equation}
\psi(x,0) = \phi_0(x),\quad A(x,0) = A_0(x),\quad x\in{\Omega}.\label{iv-1}
\end{equation}

The novel nonlinearity \( f(|\psi|^2) \) takes the form
\begin{equation}\label{pot}
f(|\psi|^2) = 1 - \int_0^{\tilde\omega} \frac{\nu_0\,\mathrm{d}\xi}{4\sqrt{\xi^2 + |\psi|^2}} \sum_{\pm}\tanh\left(\beta(\sqrt{\xi^2 + |\psi|^2} \pm \delta)\right),
\end{equation}
which corresponds to the derivative of a free-energy-type potential function \( F(|\psi|^2) \):
\begin{equation}\label{frE2}
F(|\psi|^2) = |\psi|^2 - \frac{\nu_0}{2\beta} \int_0^{\tilde\omega} \sum_{\pm} \ln \cosh\left(\beta(\sqrt{\xi^2 + |\psi|^2} \pm \delta)\right)\,d\xi.
\end{equation}

\begin{table}[htbp]
\centering
\caption{List of dimensionless and dimensional parameters appearing in the TDGL simulation.}
\renewcommand{\arraystretch}{1.3}
\small
\begin{tabular}{m{0.35\textwidth} m{0.55\textwidth}}
\toprule
\centering \textbf{Dimensionless parameters} & 
\centering \textbf{Dimensional parameters} \tabularnewline
\midrule
\centering
Ginzburg--Landau parameter: $\;\kappa \ge0.707$ (type-II) \par
Normalized Debye frequency: $\;\tilde{\omega} = \dfrac{\omega_D}{\psi_{00}} \sim 20\text{--}50$ \par
Dimensionless temperature: $\;\beta = \dfrac{\psi_{00}}{2k_B T}$ (typically $\mathcal{O}(0\text{--}10)$ near $T_c$) \par
Artificial inhomogeneity: $\;\delta(x,t)$ 
&
\centering
Effective mass of Cooper pair $m^*$ \par
Planck constant $\hbar$ \par
Charge of Cooper pair $e^* = 2e/c$ \par
Magnetic constant $\mu_0$ \par
Speed of light $c$ \par
Diffusion parameter $D$ \par
Debye frequency $\omega_D$, $\psi_{00} =\Delta(T=0)$ \par
Boltzmann constant $k_B$ \par
Energy constant $\eta$ \par
Critical/Temperature $T_c$, $T$ \tabularnewline
\bottomrule
\end{tabular}
\label{tab:parameters}
\end{table}

The derivation of this model will be presented in the following section. Here, $\Omega$ is a bounded domain in $\mathrm{R}^d(d=2,3)$, $n$ is the unit outer normal vector, $i=\sqrt{-1}$, the electric potential $\phi(x,t)$ is a real unknown scalar-valued function, the unknown real vector-valued function $A(x,t)$ is the magnetic vector potential, the unknown complex scalar-valued function $\psi(x,t)$ is the complex order parameter, the function $\psi^*=\overline{\psi}$ is the conjugate of $\psi$, the Ginzburg-Landau parameter $\kappa$ is an important positive material constant representing the ratio of penetration length to the coherence length, and the real vector-valued function function $H$ is a given external magnetic field.

This formulation ensures that the system remains dissipative, governed by the generalized energy functional:
\begin{equation}
E(t) = \frac{1}{2}\int_\Omega \left[ \left| \left( \frac{i}{\kappa}\nabla + A \right)\psi \right|^2 + F(|\psi|^2) + |\curl A - H|^2\right]\,dx. \label{tot-ene-1}
\end{equation}
The constant $\nu_0$ is determined by
\begin{equation}\label{nu-0}\nu_0\int_0^{\tilde\omega}\frac{\tanh(\beta_0\xi)}{2\xi}d\xi=1\end{equation}
with $\beta_0$ satisfying $\beta\to \beta_0$ when the temperature $T\to T_c$, where $T_c$ is the critical temperature of the superconductor.

Moreover, $x_0\rightarrow 1$ when $\beta\rightarrow \infty$, a property can be numerically verified for $f(x_0)=0$. When one takes $f(\ve \psi\ve ^2) =\frac12(\ve \psi\ve ^2-1)$ in the system \eqref{eq1a}-\eqref{eq1b}, we recover the classical TDGL equation (see e.g. \cite{Du92, Du94, Du94b, Du98, Oripov20}), whose solution is unique and regular under a gauge transformation such as the temporal gauge $\phi=0$, the Coulomb gauge $div A=0$ or the Lorentz gauge $\phi=-div A$ (e.g. \cite{Chen97, Du94, Ma23}) under an additional boundary condition. 

The present model extends the recent work \cite{Yang24} on disordered 2D superconductors, grounded in the BCS theory \cite{Bardeen57} and the nonlinear gap equation \eqref{gapeq}, which serves as the basis for the potential \eqref{pot} in our formulation. In contrast to previous studies, our hybrid TDGL model inherits the thermodynamic consistency of \cite{Yang24} while incorporating time dependence and magnetic fields. To our knowledge, this is the first time-dependent PDE framework capable of describing superconducting behavior beyond the near-critical regime and up to finite temperatures under an external magnetic field.  

On the theoretical side, we establish consistency with the classical Ginzburg--Landau model through a formal asymptotic expansion as $T \to T_c$, thereby connecting the hybrid model with the GL framework. On the numerical side, while many schemes have been developed for phase-field and classical TDGL equations (see, e.g., \cite{Fortino24, Gao14, Gao16, Gao17, Gao19,  Hong22,  Li17b, Li20, Mu97, Mu98, Qiao11, SXY18, SXY19, Tang19}), our approach distinguishes itself by proposing a simple and stable structure-preserving IMEX scheme. Compared with existing methods, the scheme is easy to implement while still ensuring both the maximum bound principle and energy stability under the zero-electric gauge $\phi=0$. Simulations in 2D and 3D with inhomogeneity and external fields confirm the model’s effectiveness in capturing vortex dynamics and thermodynamic behavior.

This paper is organized as follows. The remainder of this section supplements the motivation with a formal derivation of the model and its asymptotic reduction to the classical TDGL equation. Section~\ref{sec-scheme} presents the numerical scheme along with proofs of its structure-preserving properties. Section~\ref{sec-Num} reports simulation results illustrating vortex phenomena. Conclusions are drawn in Section~\ref{sec-con}.

\subsection{Non-dimensionalization of the Model}
This section presents the derivation of the evolution equations \eqref{eq1a}--\eqref{iv-1}, based on the time-dependent Ginzburg--Landau theory and the thermodynamic framework for disordered 2D superconductors developed in \cite{Yang24}. To demonstrate consistency with the classical GL theory, the asymptotic behavior as the temperature $T \to T_c$ is analyzed, and the classical TDGL equation with cubic nonlinearity is formally recovered via a small-parameter expansion under suitable scaling assumptions. 

As a first step, we derive the gap potential function \eqref{frE2} by embedding the gap equation into the TDGL framework, which leads to the following hybrid model.
\begin{equation}
    \psi=\sum_{k} \frac{\nu_{0}\psi}{4E_k}\sum_\pm\tanh\left(\frac1{2k_BT}(E_k\pm\delta_k )\right).\notag
\end{equation}
Here, the spectrum splitting effect in the BCS gap equation of the isotropic case is approximated using the form $E_k\pm\delta_k(x)$, and $\nu_0$ is computed at temperature $T=0$. Accordingly, the function $F$ in the free energy \eqref{tot-ene} is given by
\begin{equation}\label{frE-200}
    F(\ve\psi\ve^2) = |\psi|^2-k_BT\nu_{0}\sum_{k}\sum_\pm\ln \cosh\left(\frac1{2k_BT}(E_k\pm\delta_k)\right).
\end{equation}
By approximating the summation by using a integral and taking $\delta_k=\delta(x)$, the non-dimensionalized nonlinear function $f(|\psi|^2)=F'(|\psi|^2)$ in \eqref{eq1a} and the corresponding potential function $F(|\psi|^2)$ are derived respectively as \eqref{pot} and \eqref{frE2} from \eqref{frE-200}. Lastly, the dynamical equation is established from the time evolution from a perturbed state, as described in \eqref{eq-evo}, and by using the free energy  \eqref{tot-ene} in which the nonlinear function $F(|\psi|^2)$ is specified as \eqref{frE2}.

The TDGL equation can be derived under the energy framework, by evolving the complex order parameter $\psi$ according to \eqref{eq-evo}. The parameters in \eqref{eq-evo} - \eqref{tot-ene} are defined as in Table~\ref{tab:parameters}:


Through the direct calculation, the following functional derivatives
\begin{subequations}
\begin{numcases}{}
\frac{\delta E}{\delta \psi^*}=\frac{\hbar^2}{2\eta_1 m^*}(\nabla-i\frac{e^*}{\hbar}A)^2\psi+\psi\left(1-\sum_k\frac{\nu_{0}}{4E_k}\sum_\pm \tanh\left(\frac{E_k^\pm}{2k_BT}\right)\right)\label{var1}\\
\frac{\delta E}{\delta A}=\frac{1}{\eta_2\mu_0}\curl(\curl A-H)+ \frac{\hbar e^*}{2\eta_1 m^*i}(\psi\nabla\psi^*-\psi^*\nabla\psi)+\frac {e^{*2}}{\eta_1 m^*}\ve\psi\ve^2A\label{var2}
\end{numcases}
\end{subequations} are obtained from \eqref{tot-ene}, which found the foundation for the subsequent derivation of the TDGL equations.

Now, substituting the functional derivatives \eqref{var1}-\eqref{var2} into \eqref{eq-evo}, we obtain the following hybrid TDGL equation 
\begin{subequations}
    \begin{numcases}{}
    \frac{\hbar^2}{2m^*D}\left(\partial_t+i\frac{e^*}{\hbar}\phi\right)\psi =\frac{\hbar^2}{2\eta_1 m^*}(\nabla-i\frac{e^*}{\hbar}A)^2\psi\notag\\
    \phantom{\frac{\hbar^2}{2m^*D}\left(\partial_t+i\frac{e^*}{\hbar}\phi\right)\psi=}+\psi\left(1-\sum_k\frac{\nu_{0}}{4E_k}\sum_\pm \tanh\left(\frac{E_k^\pm}{2k_BT}\right) \right),\label{d-eq-1}\\
    \tilde\sigma\left(A_t +\nabla\phi\right)+\frac{1}{\mu_0\eta_2}\curl(\curl A-H)= \frac{\hbar e^*}{\eta_1 m^*}\text{Im}(\psi^*\nabla\psi)-\frac {e^{*2}}{\eta_1 m^*}\ve\psi\ve^2A.\label{d-eq-2}
    \end{numcases}
\end{subequations}
Here, $E_k^\pm=E_k\pm\delta_k$.

{Furthermore, to facilitate the analysis and numerical computation, a nondimensionalized model will be considered. This transformation is achieved by introducing suitable characteristic scales that normalize the physical quantities involved. Specifically, the following dimensionless parameter will be defined: $\xi^2 = \frac{\hbar^2}{2m^*\eta_1}$ and $\psi_{00} = \ve\Delta\ve\ve_{T=0}$. The dimensionless time variable is $t = \frac{\xi^2\eta_1}{D} \hat{t}$. Similarly, the spatial coordinate is rescaled as $x = \lambda \hat{x}$, and $\lambda = \sqrt{\frac{m^*\eta_1}{\mu_0\eta_2 e^{*2} \psi_{00}^2}}$ denotes the London penetration depth. Furthermore, the GL parameter is defined by $\lambda/\xi$. The flux quantum, a fundamental unit of magnetic flux in superconductivity, is given by $\phi_0 = \frac{2\pi\hbar}{e^*}$. The vector potential is then rescaled as $A = \frac{\phi_0}{2\pi\xi} \hat{A}$, while the magnetic field is normalized to $H = \frac{\phi_0}{2\pi\lambda\xi} \hat{H}$. Lastly, the scalar potential undergoes a transformation as $\phi = \frac{\lambda D \phi_0}{2\pi\eta_1\xi^3} \hat{\phi}$, and in \eqref{eq1b} the evolution parameter $\sigma = \frac{\tilde\sigma\mu_0D\eta_2}{\eta_1}\kappa^2$.}

{In addition, introduce the dimensionless parameters $\hat\beta = \frac{\psi_{00}}{2k_B T}$, $\tilde \omega = \frac{\omega_D}{\psi_{00}}$, and the dimensionless inhomogeneity $\hat \delta = \frac\delta{\psi_{00}}$. $\psi_{00}$ is approximated by $Ck_BT_c$, following the BCS approximation of the zero temperature energy gap in the pure state. After substituting these dimensionless variables into the governing equations \eqref{d-eq-1}-\eqref{d-eq-2}, the resulting expressions no longer depend on the original physical units. To simplify notation, the hat notation is subsequently removed. The desired dimensionless hybrid TDGL equations \eqref{eq1a}-\eqref{eq1b} with \eqref{pot} are derived.}

\subsection{Asymptotic behavior when $T\rightarrow T_c$}
One effective approach to understanding the behavior of the system in the weak fluctuation regime is to perform an asymptotic expansion in terms of a small parameter. The relation between BCS and the TDGL equation has been shown in \cite{Gor68}. Here, we apply a formal asymptotic expansion to analyze the limiting behavior of \eqref{eq1a}-\eqref{eq1b} when the temperature approaches the critical temperature $T \to T_c$, which determines the order of the small parameter $\varepsilon$ and yields a TDGL-type equation with polynomial form at leading order. We consider the limiting regimes $\beta=\frac{\beta_0}{1-\varepsilon^l},\, l=1,2$, and $\varepsilon$ is the dimensionless temperature parameter.

\begin{proposition}[Formal asymptotic reduction as $T \to T_c$]
Consider the hybrid TDGL system \eqref{eq-m}--\eqref{iv-1} with the nonlinear term $f(|\psi|^2)$ defined in \eqref{pot}. 
Let $\varepsilon>0$ be the dimensionless parameter associated with the scaling $\beta = \beta_0/(1-\varepsilon^2)$ and assume the expansions
\[
\psi = \varepsilon \psi_0,\quad A=\varepsilon A_0,\quad \phi=\varepsilon^2 \phi_0,\quad H=\varepsilon^2 H_0,\quad \delta\simeq \varepsilon.
\]
Then, as $\varepsilon\to 0$, the leading-order approximation of the system formally reduces to the classical TDGL-type system
\begin{subequations}\label{eq-GL-limit}
\begin{numcases}{}
    \psi_{0\tau} + i\phi_0\psi_{0} + \Big(\tfrac{i}{\kappa}\nabla + A_0\Big)^2 \psi_0
    + \tfrac{\nu_0}{2}\Big(\hat\gamma_{21}|\psi_0|^2 -1 + \hat\gamma_{23}\delta_0^2\Big)\psi_0 = 0, \\[6pt]
    \sigma(A_{0\tau}+\nabla\phi_0) + \curl\curl A_0
    + \tfrac{1}{\kappa}\text{Im}(\psi_0\nabla\psi_0^*) + A_0|\psi_0|^2 = \curl H_0,
\end{numcases}
\end{subequations}
where $\tau=\varepsilon^2 t$, $\hat x=\varepsilon x$, and the coefficients $\hat\gamma_{21},\hat\gamma_{23}$ are determined by the integrals in \eqref{coef1}.
\end{proposition}

\begin{proof}[Formal derivation]

We consider $\delta\simeq\varepsilon^\sigma$ for $\sigma=1$ in \eqref{pot}, while the results can also been extended to other values such as $\sigma=2,3...$. Specifically, we calculate the formal asymptotic expansion in the following scale $\psi = \varepsilon\psi_0$, $A = \varepsilon A_0$, $\phi = \varepsilon^2\phi_0$, and $H = \varepsilon^2H_0$. Here, $\psi_0$ is the leading-order solution obtained from the modified equation above, which captures the dominant behavior of the order parameter in the small $\varepsilon$ limit.

To derive the leading-order asymptotic expansion in the limit $\beta=\frac{\beta_0}{1-\varepsilon}$, the rescaled order parameter \(\psi = \varepsilon \psi_0\) is substituted into the equations \eqref{eq1a}-\eqref{eq1b} and the Taylor expansion is applied. Taking the leading order terms, we have
\begin{equation*}
    E_\pm =\sqrt{\xi^2 + \lvert\varepsilon \psi_0\rvert^2}+\varepsilon\delta_0\approx \xi\pm\varepsilon\delta_0+\varepsilon^2\frac{\vert\psi_0\vert^2}{2\xi},
\end{equation*}
and the expansion of $\beta E_\pm$:
\begin{align}\label{tl-e}
&\beta E_\pm = \frac{\beta_0}{1-\varepsilon}(\sqrt{\xi^2 + \lvert\varepsilon \psi_0\rvert^2} \pm \varepsilon\delta_0) \notag\\
&\approx\beta_0\left[ \xi+\varepsilon(\xi\pm\delta_0) +\varepsilon^2\left(\xi\pm\delta_0+\frac{\vert\psi_0\vert^2}{2\xi}\right)\right].
\end{align}
In addition, $\tanh(\theta+\Delta\theta) = \tanh(\theta)+\Delta\theta\sech^2(\theta)-(\Delta\theta)^2\tanh(\theta)\sech^2(\theta)+\cdots$. Substituting  \eqref{tl-e} in the taylor expansion of $\tanh(x)$, taking $\theta = \beta_0\xi$, we have
\begin{align}
    \tanh(\beta E_+)+\tanh(\beta E_-)\approx 2\tanh(\beta_0\xi)+2\varepsilon\beta_0\xi\sech^2(\beta_0\xi)+\varepsilon^2R_0(\xi),\label{tyep}
\end{align}
where
\begin{align*}
    &R_{0}(\xi)=2\beta_0\sech^2(\beta_0\xi)\left(\xi+\frac{\vert\psi_0\vert^2}{2\xi}\right)-2\beta_0^2(\xi^2+\delta_0^2)\tanh(\beta_0\xi)\sech^2(\beta_0\xi).
\end{align*}
Using the expansion $\frac1{\sqrt{\xi^2 + \lvert\varepsilon \psi_0\rvert^2}} \approx \frac1\xi-\frac{\varepsilon^2\vert\psi_0\vert^2}{2\xi^3}$, we have
\begin{align*}
    &\frac{\tanh(\beta E_+)+\tanh(\beta E_-)}{2(E_++E_-)}\\
    &\approx \frac14(\tanh(\beta E_+)+\tanh(\beta E_-))\left(\frac1\xi-\frac{\varepsilon^2\vert\psi_0\vert^2}{2\xi^3}\right)\\
    &\approx\frac{\tanh(\beta_0\xi)}{2\xi}+\varepsilon\frac{\beta_0}2\sech^2(\beta_0\xi)+\varepsilon^2\left(\frac1{4\xi}R_{0}(\xi)-\frac{\vert\psi_0\vert^2}{4\xi^3}\tanh(\beta_0\xi)\right).
\end{align*}
The above expansion corresponds to the integrand of the nonlinear term $f$. Furthermore, the following coefficients $1-\hat \gamma_0$, $\hat \gamma_1$, $1-\hat \gamma_2$ can be explicitly determined from the expansion respectively. Denoting
\begin{align}\label{coef1}
    &f \approx 1-\nu_0(\hat \gamma_0+\varepsilon\hat \gamma_1+\varepsilon^2\hat \gamma_2)=\frac{\nu_0}2\left[-\hat\gamma_1\varepsilon+\varepsilon^2\left(\hat\gamma_{21}\vert\psi\vert_0^2-\hat\gamma_{22}+\hat\gamma_{23}\delta_0^2\right)\right],
\end{align}
we have
\begin{align*}
    &\hat\gamma_0=\int_0^{\tilde\omega}\frac{\tanh(\beta_0\xi)}{2\xi}\text{d}\xi,\,\hat\gamma_1 = \int_0^{\tilde\omega}\beta_0\sech^2(\beta_0\xi)\text{d}\xi\approx 1,\\
    &2\hat\gamma_2 = -\hat\gamma_{21}|\psi_0|^2+\hat\gamma_{22}-\hat\gamma_{23}\delta_0^2,\\
    &\hat\gamma_{21}=\int_0^{\tilde\omega}\left(\frac{\tanh(\beta_0\xi)}{2\xi^3}-\frac{\beta_0\sech^2(\beta_0\xi)}{2\xi^2}\right)\text{d}\xi\geq0,\\
    &\hat\gamma_{22}=\int_0^{\tilde\omega}\left(\beta_0\sech^2(\beta_0\xi)-\beta_0^2\xi\tanh(\beta_0\xi)\sech^2(\beta_0\xi)\right)\text{d}\xi\approx\frac{1}2,\notag\\
    &\hat\gamma_{23}=\int_0^{\tilde\omega}\frac{\beta_0^2\tanh(\beta_0\xi)\sech^2(\beta_0\xi)}{\xi}\text{d}\xi.
\end{align*}

The above estimated values are calculated at $\beta_0 = 0.882$. Since $1 = \nu_0\int_0^{\tilde\omega}\frac{\tanh(\beta_0 \xi)}{2\xi}\text{d}\xi$, the constant term $1-\nu_0\hat\gamma_0$ in $f$ vanishes when $\varepsilon\rightarrow 0$.
Consequently, the leading order equation of $\psi_0$ can be obtained. The leading order limit of the asymptotic approximation of the equations \eqref{eq-m} in the scale $\varepsilon^2 t = \tau$ and $\varepsilon x = \hat x$ are
\begin{equation*}
    \psi_0 =0,\,\sigma(A_{0{\tau}} +\nabla\phi_0)+{\curl}{\curl} A_0 ={\curl} H_0\notag.
\end{equation*}

The $O(\varepsilon)$ term in the expansion give rise to a rapidly varying contribution term $\frac{\psi_0}{\varepsilon}$, which leads to the above equation and deviates from the standard TDGL equation. To address this inconsistency, we consider an alternative limitting regimeby rescaling the parameter as $\beta=\frac{\beta_0}{1-\varepsilon^2}$. Using this new scaling, we similarly obtain
\begin{align}
    &\beta E_\pm = \frac{\beta_0}{1-\varepsilon^2}(\sqrt{\xi^2 + \lvert\varepsilon \psi_0\rvert^2} \pm \delta)\approx\beta_0\left[ \xi\pm\varepsilon\delta_0 +\varepsilon^2\left(\xi+\frac{\vert\psi_0\vert^2}{2\xi}\right)\right].
\end{align}
and
\begin{align*}
    &\frac{\tanh(\beta E_+)+\tanh(\beta E_-)}{2(E_++E_-)}\\
    &\approx \frac14(\tanh(\beta E_+)+\tanh(\beta E_-))\left(\frac1\xi-\frac{\varepsilon^2\vert\psi_0\vert^2}{2\xi^3}\right)\\
    &\approx\frac{\tanh(\beta_0\xi)}{2\xi}+\varepsilon^2\left(\frac1{4\xi}R^*(\xi)-\frac{\vert\psi_0\vert^2}{4\xi^3}\tanh(\beta_0\xi)\right).
\end{align*}
Here,
\begin{align*}
    &R^*(\xi)=2\beta_0\sech^2(\beta_0\xi)\left(\xi+\frac{\vert\psi_0\vert^2}{2\xi}\right)-2\beta_0^2\delta_0^2\tanh(\beta_0\xi)\sech^2(\beta_0\xi).
\end{align*}
Before proposing the equation of this scale, we derive the coefficients as in \eqref{coef1}. Note that in this scale, the $O(\varepsilon)$ term vanishes. Denote
\begin{align}
    &f^* \approx 1-\nu_0(\hat \gamma_0+\varepsilon^2\hat \gamma_2)=\frac{\varepsilon^2\nu_0}2\left(\hat\gamma_{21}\vert\psi\vert_0^2-\hat\gamma_1+\hat\gamma_{23}\delta_0^2\right),\label{rem}
\end{align}
where the constants are defined in \eqref{coef1}.

Therefore, we recover the TDGL equation in this limiting regime. Consider again the scaled variables $\varepsilon^2 t = \tau$ and $\varepsilon x = \hat x$. In these scaled variables, the equation becomes
\begin{subequations}
    \begin{numcases}{}
    \psi_{0 \tau} +i\phi_0\psi_{0}+(\frac{i}{\kappa}\nabla+A_0)^2\psi_0+\frac{\nu_0}2\left(\hat\gamma_{21}\ve\psi_0\ve^2-1+\hat\gamma_{23}\delta_0^2\right)\psi_0=0,\notag\\
    \sigma(A_{0{\tau}} +\nabla\phi_0)+{\curl}{\curl} A_0  +\frac{1}{\kappa}\text{Im}(\psi_0\nabla\psi_0^*)+A_0\ve \psi_0\ve ^2 ={\curl} H_0\notag.
\end{numcases}
\end{subequations}


\end{proof}

\begin{theorem}[Convergence rate as $T \to T_c$]
Under the same conditions, consider the hybrid TDGL system \eqref{eq-m}--\eqref{iv-1} with the scaling $\beta = \beta_0/(1-\varepsilon^2)$ 
Then, as $\varepsilon\to 0$, we have
\begin{equation}
\Vert \psi(t,x) - \varepsilon \psi_0(\varepsilon^2 t,\varepsilon x)\Vert+\Vert A (t,x)- \varepsilon A_0(\varepsilon^2 t,\varepsilon x)\Vert\leq C\varepsilon^2
\end{equation}
\begin{proof}
The argument resembles the a priori estimates in \cite{Mu98}, and is thus simplified. Denote $e_\psi = \psi(t,x) - \varepsilon\psi_0(\varepsilon^2 t,\varepsilon x)$ and $e_A = A(t,x) - \varepsilon A_0(\varepsilon^2 t,\varepsilon x)$. 
By multiplying the equation \eqref{eq-GL-limit} by $\varepsilon^3$, and then subtracting the resulting equation from \eqref{eq-m}, we obtain the error equation. Passing to the weak formulation, this reads

\begin{subequations}
\begin{numcases}{}
\langle{e_\psi}_t, \eta\rangle 
+\langle D_A^2 e_\psi ,\eta\rangle - \langle \varepsilon f^*\psi_0 - f(|\psi|^2)\psi,\eta \rangle +C\langle R_\varepsilon,\eta\rangle\\
\phantom{xx} =\frac 1{\kappa}  Im\langle \eta,A\varepsilon\nabla \psi_0\rangle+ \langle \varepsilon(A+\varepsilon A_0)e_A\psi_0,\eta\rangle,\notag
\\
\sigma({e_A}_t,\xi) + (\curl\curl e_A,\xi) + \frac{1}{\kappa}(\text{Im}(\psi\nabla\psi^*-\varepsilon^2\psi_0\nabla\psi_0^*) ,\xi)) \\
\phantom{xx}= (\curl e_H,\xi) - (|\psi|^2e_A,\xi) - (\varepsilon(|\psi|^2-\varepsilon^2|\psi_0|^2)A_0,\xi).  \notag
\end{numcases}
\end{subequations}
Taking $\eta=e_\psi$, $\xi=e_A$ in the error equations, we obtain
\begin{equation}
\langle (e_\psi)_t, e_\psi\rangle + \sigma \big((e_A)_t,e_A\big)
+ \langle D_A^2 e_\psi, e_\psi\rangle + (\curl\curl e_A,e_A)
= \mathcal{N} + \mathcal{C},\notag
\end{equation}
where $\mathcal{N}$ collects the nonlinear difference terms and 
$\mathcal{C}$ collects the remaining cross terms. The nonlinear term reads
\begin{equation}
|\langle \varepsilon f(\varepsilon^2|\psi_0|^2)\psi_0 - f(|\psi|^2)\psi,e_\psi \rangle| + |\langle\varepsilon f(\varepsilon^2|\psi_0|^2)\psi_0 - \varepsilon f^*\psi_0,e_\psi \rangle|.\notag
\end{equation}
The first part is controlled by Lipschitz continuity. Using \eqref{rem}, we have $R_\varepsilon = O(\varepsilon^3)$, and the $R_\varepsilon $ yields $O(\varepsilon^6)$ after pairing with $e_\psi$.  
For the cross terms $\mathcal{C}$, using the estimates in \cite{Mu98} yields
\begin{equation}
\frac{d}{dt}\Big(\|e_\psi\|^2+\sigma\|e_A\|^2\Big)
\;\le\; C\big(\|e_\psi\|^2+\|e_A\|^2\big) + C\varepsilon^6.\notag
\end{equation}

Applying Gronwall’s inequality, and assuming 
$\|e_\psi(0)\|^2+\sigma\|e_A(0)\|^2=O(\varepsilon^2)$, we conclude
\begin{equation}
\|e_\psi\|^2+\sigma\|e_A\|^2 \;\le\; C\varepsilon^2, \qquad t\in[0,T].\notag
\end{equation}
\end{proof}
\end{theorem}


\section{Energy stable semi-discrete scheme}
\label{sec-scheme}
In this section, we investigate a semi-discrete scheme derived from a stabilized numerical scheme for the dimensionless equation for numerical simulation. When $\phi = 0$, the discrete maximum modulus bound and energy stability can be established for this implicit-explicit (IMEX) scheme with a stabilizer, similar to the property in \cite{Ma23}. In this work, a global Lipschitz condition $\vert f'(|\psi|)\vert \le L $ is assumed, which can be enforced via a truncation technique. However, this condition is not strictly necessary and may be relaxed, provided that a uniform bound on \( |\psi| \) is maintained. Based on this observation, the following linear first-order scheme is proposed to enable fast and stable numerical simulation.
 The time interval is discretized into $N$ intervals and $t^n = n\tau$ for $n= 1,2,3\cdots$. Furthermore, we introduce the short notation $\delta_t\psi^{n+1} =\psi^{n+1}-\psi^n$, and impose the homogenous Neumann boundary condition throughout this section.


\begin{subequations}
\begin{numcases}{}
\frac{\delta_t\psi^{n}}\tau +(\frac{i}{\kappa}\nabla+ A^{n-1})^2\psi^{n} +f(|\psi^{n-1}|^2)\psi^{n}+S\delta_t\psi^{n}=0,\label{sh2a}\\
\sigma\frac{\delta_tA^{n}}\tau +\curl\curl A^{n} +{\frac1\kappa}\text{Im}\{\psi^{n}\nabla{{\psi^{n*}}}\} +A^{n}\ve \psi^{n}\ve ^2={\curl H^{n}}\label{sh2b}.
\end{numcases}
\end{subequations}

{Note that the zero of \eqref{pot} is generally not 1, but becomes approximately 1 when $\beta$ is large (or the temperature is low). The following analysis focuses on the case $x_0=1$, though it can be extended to more general cases $x_0\not=1$.}
\begin{theorem}[Maximum modulus bound]
Assuming that $\ve f'\ve\le L $ for some positive constant $L$. If $S\ge2L$ and $\ve \psi_0\ve \leq1$, we have {$\ve \psi^{n}\ve \leq1$ for all $n$}.
\end{theorem}
\begin{proof} Multiplying \eqref{sh2a} by ${\psi^{n*}}$ and taking the real part of the result, we have

\begin{align}
&\frac{\ve\psi^n\ve^2}\tau+\text{Re}\{{\psi^{n*}}(\frac{i}\kappa\nabla+A^{n-1})^2\psi^{n}\} + f(|\psi^{n-1}|^2)\ve \psi^{n}\ve ^2+\frac{S\delta_t\ve\psi^n\ve^2}2\leq\frac{\ve\psi^{n-1}\ve^2}\tau\nonumber.
\end{align}
Notice that $f(1)=0$. Since $f(\ve\psi^{n-1}\ve^2)-f(1) =f'(\ve\xi\ve^2)(\ve\psi^{n-1}\ve^2-1) $ {for some $\xi$ satisfying that $|\xi|$ is between $|\psi^{n-1}|$ and $1$}, we have
\begin{align}
&\text{LHS}\nonumber\\
&=\frac{\ve\psi^n\ve^2}\tau+\text{Re}\{{\psi^{n*}}(\frac{i}{\kappa}\nabla+A^{n-1})^2\psi^{n}\} +f(|\psi^{n-1}|^2)\ve \psi^{n}\ve ^2+\frac{S(\ve\psi^n\ve^2-\ve\psi^{n-1}\ve^2)}2\nonumber\\
&=\frac{\ve\psi^n\ve^2}\tau+\text{Re}\{{\psi^{n*}}(\frac{i}{\kappa}\nabla+A^{n-1})^2\psi^{n}\} \nonumber\\
&\phantom{xx}+\left( f(\ve\xi\ve^2)\ve\psi^{n}\ve ^2+\frac S2\ve\psi^n\ve^2\right)(1-\ve\psi^{n-1}\ve^2)+\frac{S}2\left(\ve\psi^n\ve^2-1\right){\ve\psi^{n-1}\ve^2}\label{lsh-1}.
\end{align}
The proof is based on induction. The bound of the initial value $\ve \psi^0\ve \leq1$ is guaranteed by appropriately discretizing the initial condition. In addition, we assume that {$\ve \psi^j\ve \leq1$ for $j=1\cdots n-1$, and we show that $\ve \psi^n\ve \leq1$ holds} by contradiction. Assume that $\ve  \psi^{n}\ve $ approaches its maximum at $x_m$ and $\ve  \psi^{n}(x_m)\ve =\psi_m>1$. Thus, we have
\begin{align*}
\nabla\ve \psi^{n}(x_m)\ve ^2=0, -\nabla^2(\ve \psi^{n}(x_m)\ve ^2)=-\nabla\cdot\text{Re}\{\psi^{n*}(x_m)\nabla\psi^{n}(x_m)\}\geq0.\notag
\end{align*}
{Here and in the following we use the notation $\nabla h(x_m)$ to denote $(\nabla h(x))|_{x=x_m}$.}

Denote $\psi^{n}= \ve  \psi\ve  e^{i\theta}$, thus $\nabla\ve \psi^{n}(x_m)\ve = 0.5\psi_m^{-1}\nabla\ve \psi^{n}(x_m)\ve ^2=0$. Since
\begin{align}
&\text{Re}\{(\frac{i}{\kappa}\nabla+A^{n-1})^2\psi^{n}{\psi^{n*}}\}\nonumber\\
&=\text{Re}\{-\frac1{\kappa^2}\nabla^2\psi^{n}{\psi^{n*}}+\frac{i}{\kappa}\text{div}{A^{n-1}}\ve \psi^{n}\ve ^2+\frac{2i}{\kappa}{A^{n-1}}\nabla\psi^{n}{\psi^{n*}}+|{A^{n-1}}|^2\ve \psi^{n}\ve ^2\}\nonumber\\
&=\text{Re}\{\frac1{\kappa^2}\nabla\psi^{n}\nabla{\psi^{n*}}+\frac{2i}{\kappa}{A^{n-1}}\nabla\psi^{n}{\psi^{n*}}+|{A^{n-1}}|^2\ve \psi^{n}\ve ^2\}\\
&\phantom{xx}-\frac{1}{\kappa^2}\nabla\text{Re}\{({\psi^{n*}}\nabla\psi^{n})\},\notag
\end{align}
and
\begin{align}
&\text{Re}\{\nabla(\ve  \psi\ve  e^{i\theta})\nabla(\ve  \psi\ve  e^{-i\theta})\} =\ve \nabla\ve  \psi\ve \ve ^2+\ve \psi\ve ^2\ve \nabla\theta\ve ^2,\notag\\
&\text{Re}\{2iA\nabla(\ve  \psi\ve  e^{i\theta})(\ve  \psi\ve  e^{-i\theta})\}= - 2\ve \psi\ve ^2A\cdot\nabla\theta,\notag
\end{align}
we have
\begin{align}
&\text{Re}\{{\psi^{n*}(x_m)}(\frac{i}{\kappa}\nabla+A^{n-1})^2\psi^{n}(x_m)\}\\
&=\psi_m^2(\frac{1}{\kappa^2}\ve \nabla\theta\ve ^2 - \frac{2}{\kappa}A^{n-1}\cdot\nabla\theta+\ve  A^{n-1}\ve ^2)\geq 0.\notag
\end{align}
According to the assumption $\ve\psi^{n-1}\ve^2\le1$, we have $\text{LHS}\le\tau^{-1}$. However, from the last inequality {in \eqref{lsh-1}}, we have $\text{LHS}{(x_m)}\geq \tau^{-1}\ve \psi^{n}(x_m)\ve ^2>\tau^{-1}$, which raises a contradiction. Thus, {$\ve \psi^{n}\ve \leq1$ holds.}
\end{proof}

Similarly, the following discrete energy stability can be established.
\begin{theorem}[Energy inequality]\label{thm-ene}
Assuming that $|\psi^0|=|\psi_0|\le 1$ and $\ve f'\ve\le L $ for $\ve\psi\ve\leq1$. If $\phi = 0$ and $S\geq 2L$, we have
\begin{equation}\label{dis-eq}
{\delta_tE^n\leq (H^n-\curl A^{n},\delta_t H^{n})}.\notag
\end{equation}
where,\begin{equation}
E^{n} = \frac12\left(\vert (\frac i\kappa\nabla+A^n)\psi^{n}\vert ^2+ F(|\psi^{n}|^2)+\vert\curl A^n-H^n\vert^2,1\right).\notag
\end{equation}
\end{theorem}
\begin{proof}
Taking the complex inner product $\la\cdot,\cdot\ra$ of \eqref{sh2a} with  {$2\delta_t\psi^n$} and the standard $L_2$ inner product $(\cdot,\cdot)$ of  \eqref{sh2b} with $\delta_t A^{n}$ correspondingly, and taking the real part of the results, {the resulting equalities can be established. The corresponding terms of this resulting equation can be dealt with as follows.}

The first and the second terms resulting from the equation \eqref{sh2a} are $2\tau^{-1}\text{Re}\la \delta_t\psi^n,\delta_t\psi^n \ra=2\tau^{-1}\la\ve\delta_t \psi\ve^2,1\ra$ and
\begin{align}
&2\text{Re}\la \delta_t\psi^n, f(|\psi^{n-1}|^2){\psi}^{n}+S\delta_t\psi^n)\ra\nonumber\\
=&\la  f(|\psi^{n-1}|^2)(\delta_t(\ve{\psi}^{n}\ve^2)+\ve\delta_t\psi^n\ve^2)+2S\ve\delta_t\psi^n\ve^2,1\ra\nonumber\\
=&\la \delta_tF(|\psi^n|^2){-0.5f'(|\eta^{n}|^2)(\delta_t\ve{\psi}^{n}\ve)^2}(\ve\psi^n\ve+\ve\psi^{n-1}\ve)^2\\
&+(2S+f(|\psi^{n-1}|^2))\ve\delta_t\psi^n\ve^2,1\ra\nonumber\\
\ge&\la \delta_tF(|\psi^n|^2)+(2S+f(|\psi^{n-1}|^2)-2L)\ve\delta_t\psi^n\ve^2,1\ra\nonumber
\end{align}
for some $\eta^n$. If $S>3L$, using the maximum bound, we have $2S+f(|\psi^{n-1}|^2)-2L>0$.

A similar approach to that in \cite{Mu98} is adopted to handle the third term
\begin{align*}
&2\text{Re}\la(\frac {i}{\kappa}\nabla +A^{n-1})^2\psi^{n},\delta_t {\psi}^n\ra\notag\\
&=\left(\ve  (\frac i\kappa\nabla+A^{n})\psi^{n}\ve ^2 - \ve  (\frac i\kappa\nabla+A^{n-1})\psi^{n-1}\ve ^2+\ve  (\frac i\kappa\nabla+A^{n-1})\delta_t\psi^{n}\ve ^2,1\right)\\
&\phantom{xx}-{(\frac2\kappa\text{Im}\{\psi^{n}\nabla{\psi^{n*}}\}\delta_t A^n +\delta_t \vert A^{n}\vert^2\ve \psi^{n}\ve ^2,1)}\nonumber.
\end{align*}

Furthermore, taking the inner product between \eqref{sh2b} with $\delta_t A^{n}$, we have
\begin{align}
&\sigma(\ve \delta_tA\ve^2,1)+(\curl A^n-H^n,\curl\delta_t A^n) +{(\frac1{\kappa}\text{Im}\{\psi^{n}\nabla{\psi^{n*}}\} +A^{n}\ve \psi^{n}\ve ^2,\delta_t A^n)}\notag\\
&={(\curl A^{n}-H^n,\delta_t H^{n})+(H^n-\curl A^{n},\delta_t H^{n})}\label{en-A}.
\end{align}

By summing up all the equations preceeding \eqref{en-A} and 2 $\times$ \eqref{en-A}, we obtain
\begin{equation}
    2\delta_tE^n+(2S-3L)\la \ve\delta_t\psi^n\ve^2,1\ra\le{2(H^n-\curl A^{n},\delta_t H^{n})}.\notag
\end{equation}
Thus, the desired discrete energy inequality \eqref{dis-eq} follows directly from this inequality under the condition $2S>3L$.
\end{proof}

\begin{remark} 
In this section, the stability results are established to support the feasibility of numerical simulations. The IMEX scheme maintains stability while avoiding the solvability issues of fully nonlinear methods. Topics on theory analysis such as the rigorous asymptotic limit $T\to T_c$, other effective schemes, and numerical analysis remain for future investigations.
\end{remark}

\section{Numerical results}
\label{sec-Num}
In this section, the parameters are chosen artificially to test the numerical performance and simulate the vortex dynamics. The constant $\sigma = 1$. The Debye frequency is set to be approximately $\tilde\omega = 29.3$.
The dimensionless temperature constant is taken as $\beta = 8.82$, corresponding to $T=0.1T_c$. For simplicity, the zero electric gauge $\phi=0$ and the homogenous Neumann boundary condition are imposed in this section. $\nu_0$ is approximated using $\nu_0 = 2(\text{arcsinh}(\tilde\omega))^{-1}$, which is calculated at $T = 0$ in \eqref{gapeq}. The initial values are chosen to be a superconducting states $\psi = 0.8+i0.6$ and $A = (1e-6,1e-6)$. For the spatial discretization, \(P_1\) and \(N_1\text{curl}\) finite element spaces are employed in FEniCS~\cite{Fenics12}. In addition, the linearity of the scheme is utilized; the nonlinear term is obtained from a pre-calculated table, which increases the efficiency of the numerical simulation.

\setlength{\tabcolsep}{3pt}  %
\begin{table}[!htp]
    \centering
    \begin{tabular}{ccccccc}
        \hline
        $\tau\,(\times10^{-3})$ & $\|\psi - \psi_{\rm ref}\|_{L^2}$ & rate & $\|\psi - \psi_{\rm ref}\|_{H^1}$ & rate & $\|A - A_{\rm ref}\|_{H(\curl)}$ & rate \\
        \hline
        32.0 & 5.691e-04 &   & 2.875e-03 &   & 3.831e-03 &   \\
        16.0 & 3.553e-04 & 0.76  & 1.775e-03 & 0.78  & 1.543e-03 & 1.31  \\
         8.0 & 2.102e-04 & 0.80  & 1.042e-03 & 0.81  & 6.812e-04 & 1.13  \\
         4.0 & 1.204e-04 & 0.89  & 5.961e-04 & 0.89  & 3.124e-04 & 1.30  \\
         2.0 & 6.478e-05 & 1.22  & 3.211e-04 & 1.22  & 1.269e-04 & 2.03  \\
        \hline
    \end{tabular}

    \vspace{1em}

    \begin{tabular}{ccccccc}
        \hline
        $\tau\,(\times10^{-3})$ & $\|\psi - \psi_{\rm ref}\|_{L^2}$ & rate & $\|\psi - \psi_{\rm ref}\|_{H^1}$ & rate & $\|A - A_{\rm ref}\|_{H(\curl)}$ & rate \\
        \hline
        32.0 & 5.174e-04 &   & 2.596e-03 &   & 2.149e-03 &   \\
        16.0 & 3.059e-04 & 0.76  & 1.524e-03 & 0.77  & 9.490e-04 & 1.13  \\
         8.0 & 1.753e-04 & 0.80  & 8.716e-04 & 0.81  & 4.351e-04 & 1.30  \\
         4.0 & 9.432e-05 & 0.89  & 4.696e-04 & 0.89  & 1.766e-04 & 2.04  \\
         2.0 & 4.038e-05 & 1.22  & 2.016e-04 & 1.22  & 4.306e-05 & 2.04  \\
        \hline
    \end{tabular}
    \caption{Temporal convergence rates for $\Omega_1 = (-\pi,\pi)^2$ and $\Omega_2 = (-2\pi,2\pi)^2$ and $H = 0.5+\exp(-t)$.}
    \label{table.1}
\end{table}

In the first Table~\ref{table.1}, we test the temporal convergence rate of the scheme. Specifically, two sets of data are presented for the convergence rate, where the reference solutions at $\Omega_k = (-k\pi,k\pi)^2$ for $k= 1,2$ are calculated using the time step sizes, $ \Delta t = 0.0005$. The error values are calculated at $t=0.064$, $\kappa=2$, and $H = 0.5+\exp(-t)$. The error convergence rate is presented to verify the scheme's accuracy. By comparing numerical solutions obtained with different time step sizes, the rate of convergence is calculated to validate the expected behavior of the method. Exact error values and their corresponding convergence rates are summarized in Table~\ref{table.1}, which presents the convergence behavior of two different error terms, denoted by $L_2$ and $H_1$ norms, respectively, supporting the temporal accuracy of the scheme.

\begin{figure}[htbp]
   \centering
   \begin{minipage}{0.85\textwidth}
       \centering
       \textbf{H = 0.15} \\[2pt]
       \begin{minipage}{0.19\textwidth}
           \centering
           \includegraphics[width=\linewidth]{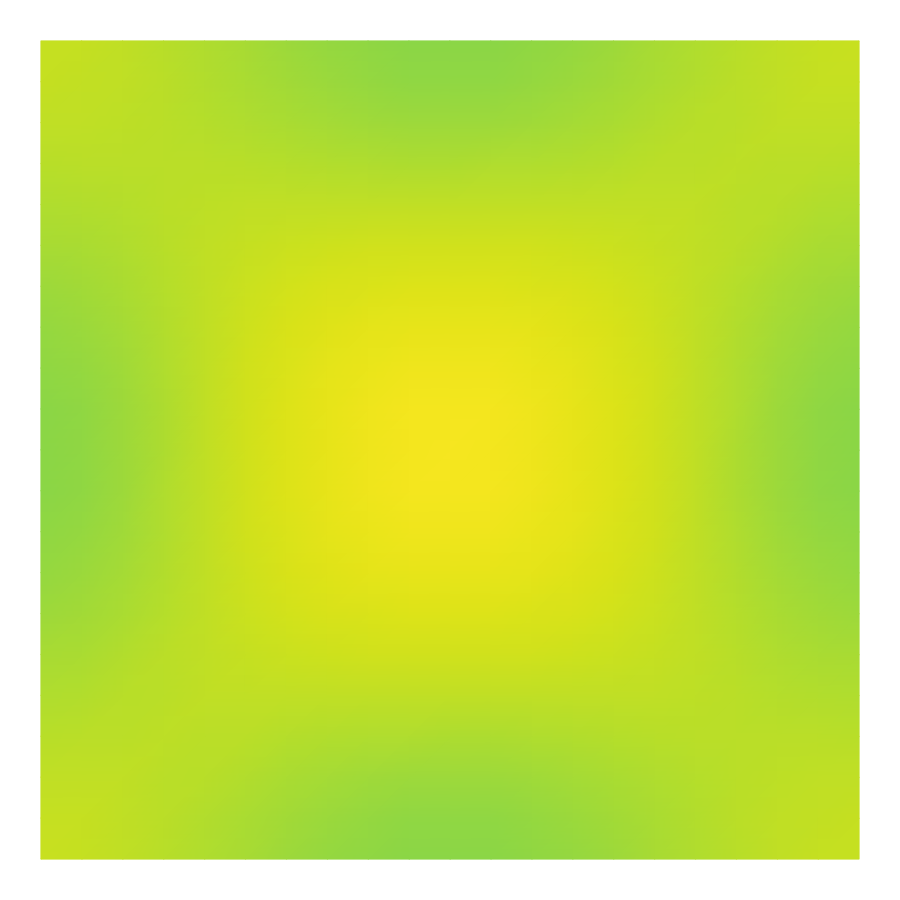}
       \end{minipage}
       \begin{minipage}{0.19\textwidth}
           \centering
           \includegraphics[width=\linewidth]{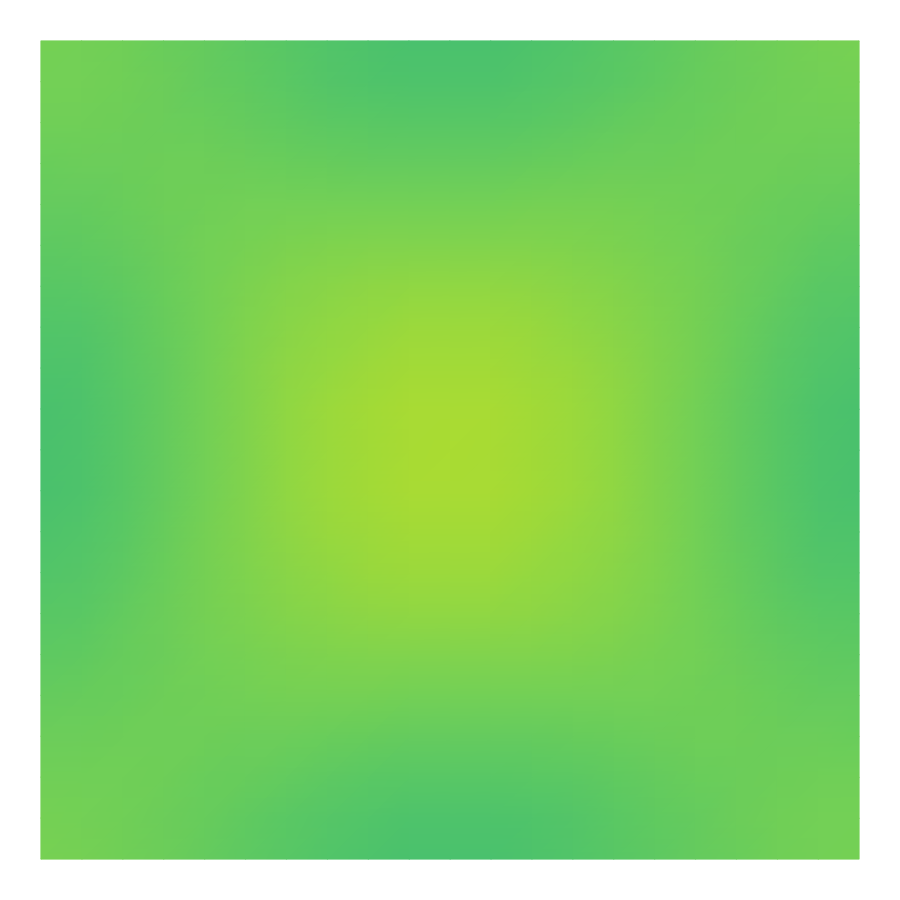}
       \end{minipage}
       \begin{minipage}{0.19\textwidth}
           \centering
           \includegraphics[width=\linewidth]{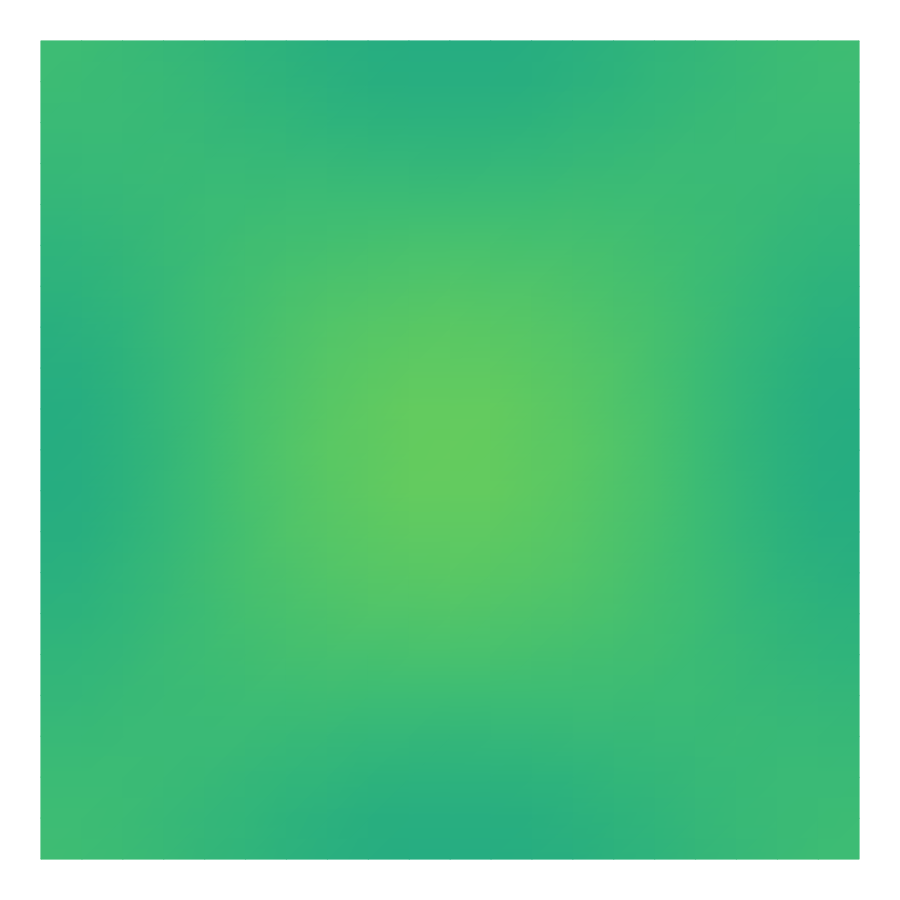}
       \end{minipage}
       \begin{minipage}{0.19\textwidth}
           \centering
           \includegraphics[width=\linewidth]{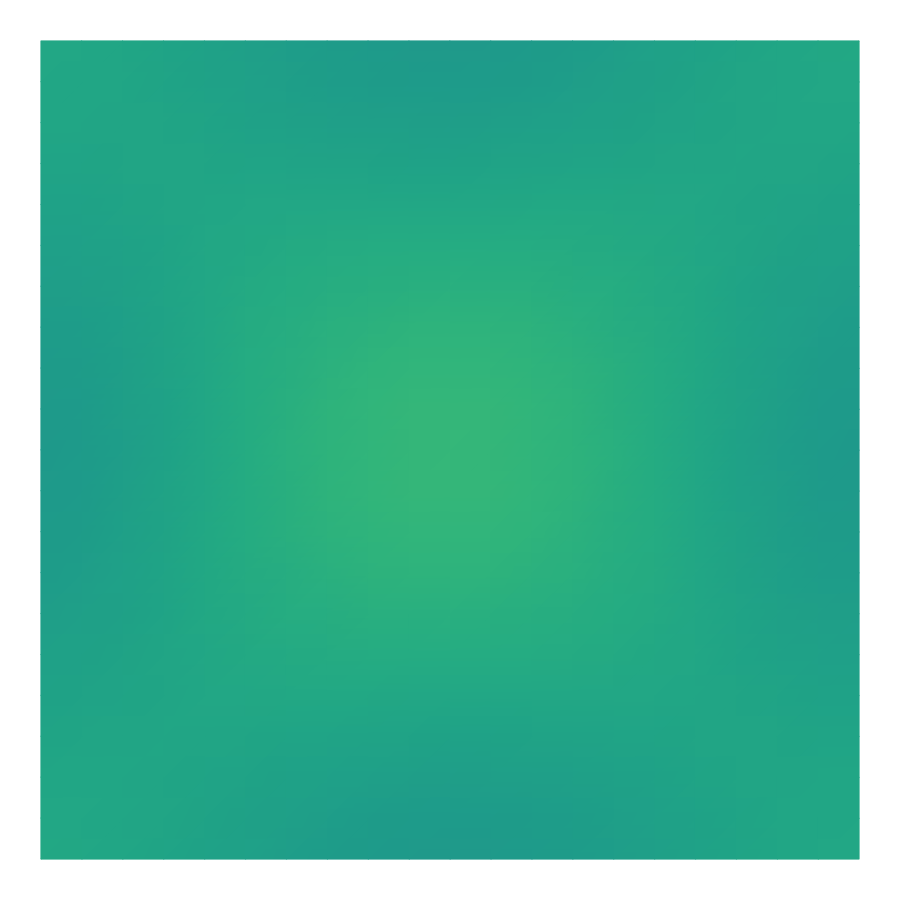}
       \end{minipage}
       \begin{minipage}{0.19\textwidth}
           \centering
           \includegraphics[width=\linewidth]{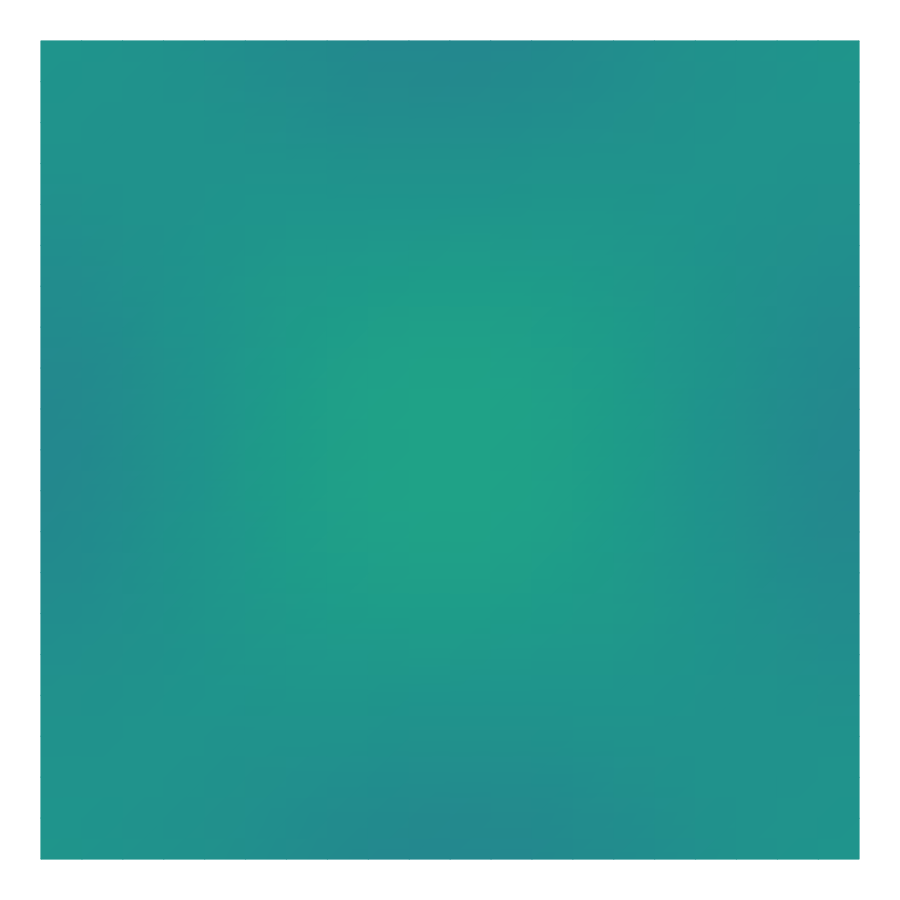}
       \end{minipage}

       \vspace{5pt}
       \textbf{H = 0.3} \\[2pt]
       \begin{minipage}{0.19\textwidth}
           \centering
           \includegraphics[width=\linewidth]{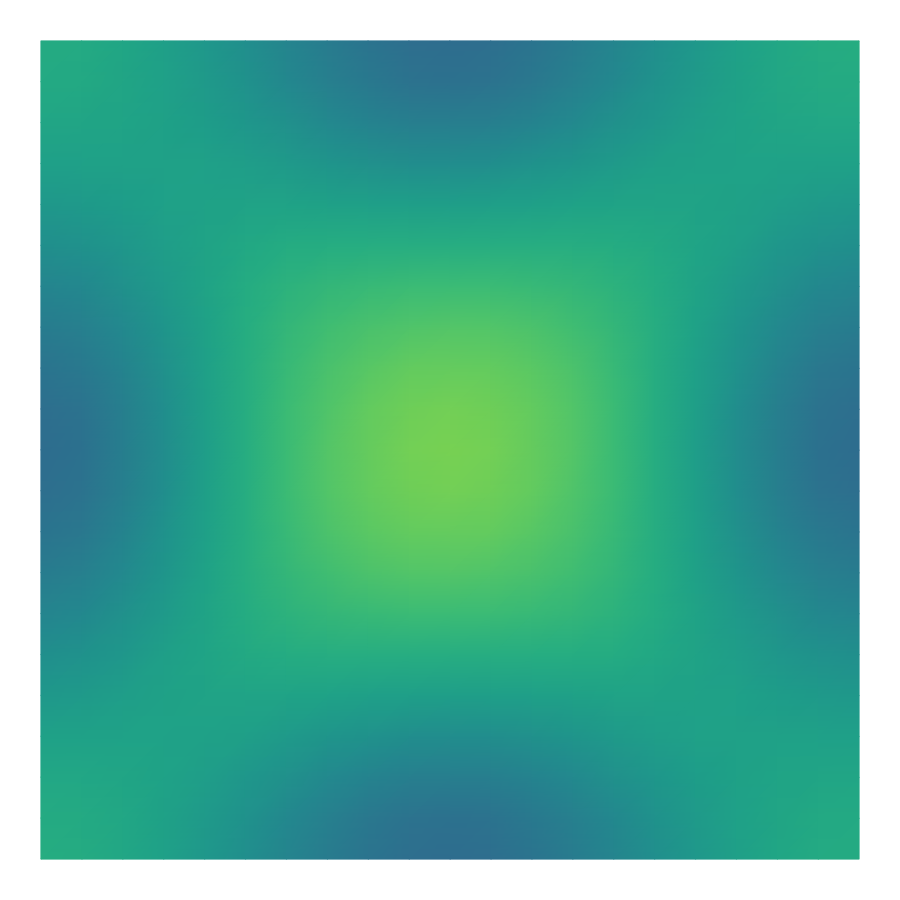}
       \end{minipage}
       \begin{minipage}{0.19\textwidth}
           \centering
           \includegraphics[width=\linewidth]{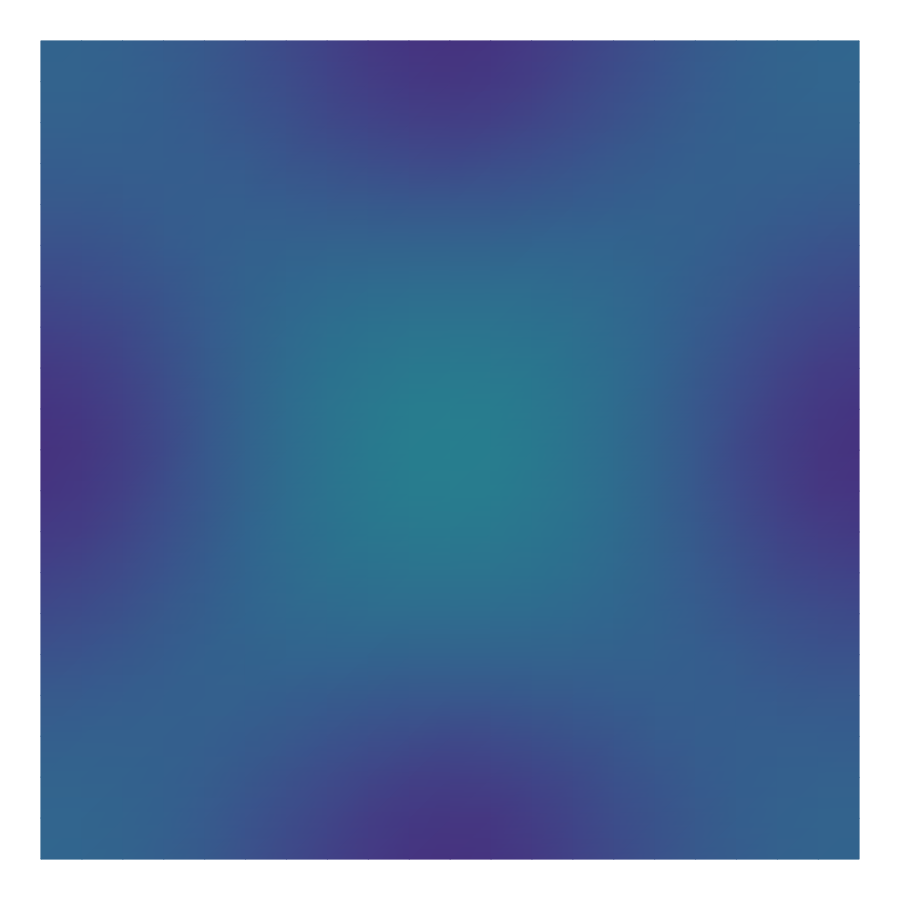}
       \end{minipage}
       \begin{minipage}{0.19\textwidth}
           \centering
           \includegraphics[width=\linewidth]{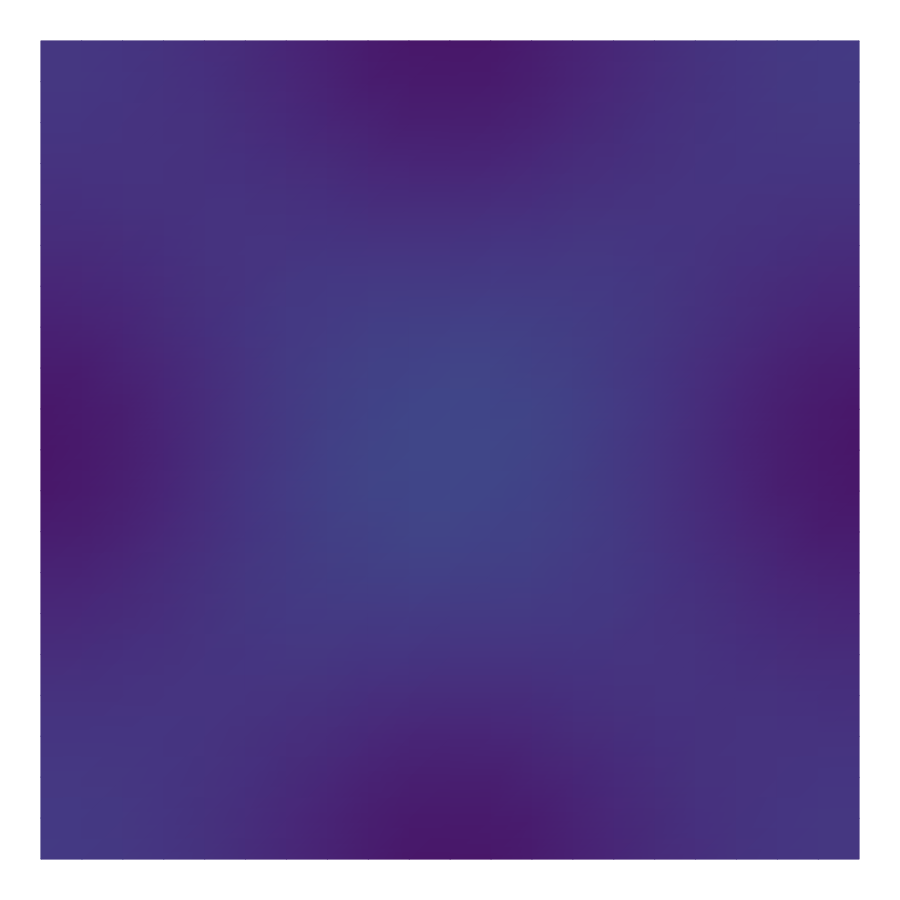}
       \end{minipage}
       \begin{minipage}{0.19\textwidth}
           \centering
           \includegraphics[width=\linewidth]{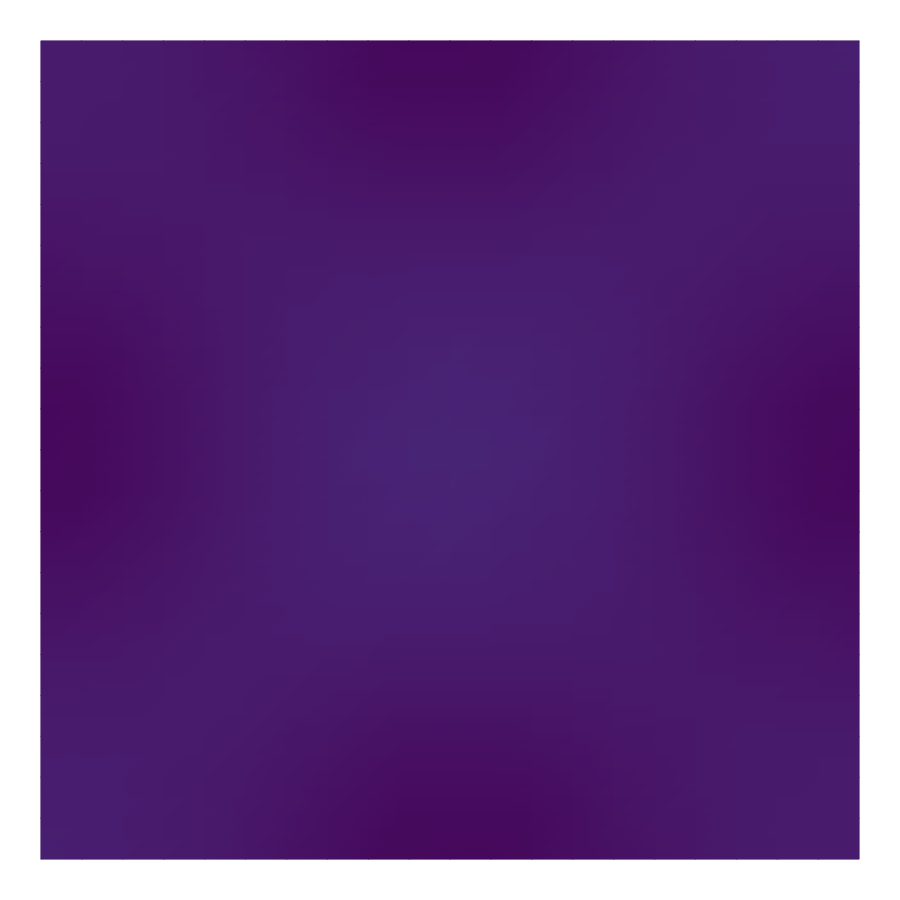}
       \end{minipage}
       \begin{minipage}{0.19\textwidth}
           \centering
           \includegraphics[width=\linewidth]{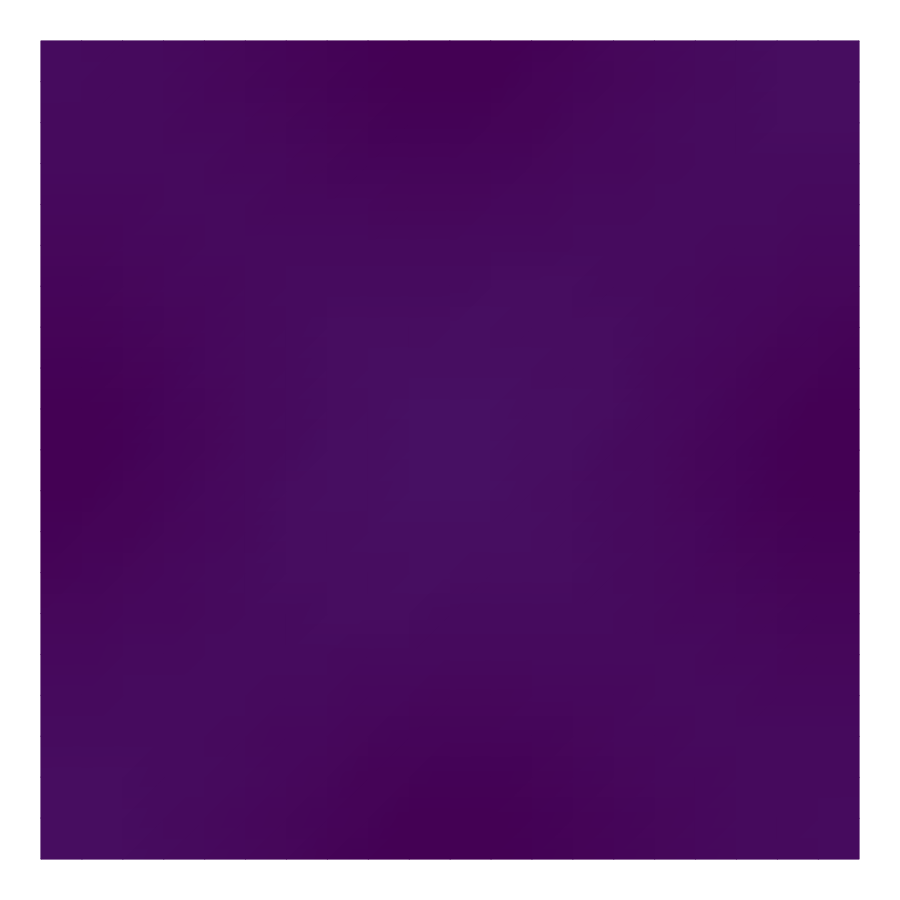}
       \end{minipage}
   \end{minipage}%
   \hspace{2pt}
   \begin{minipage}{0.035\textwidth}
       \centering
       \includegraphics[height=0.24\textheight]{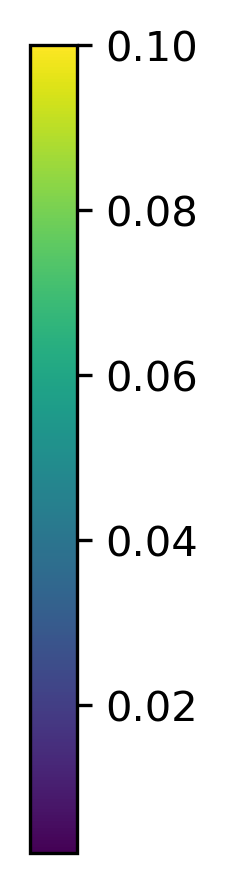}
   \end{minipage}

       \begin{minipage}{0.85\textwidth}
       \centering
       \textbf{H = 0.15} \\[2pt]
       \begin{minipage}{0.19\textwidth}
           \centering
           \includegraphics[width=\linewidth]{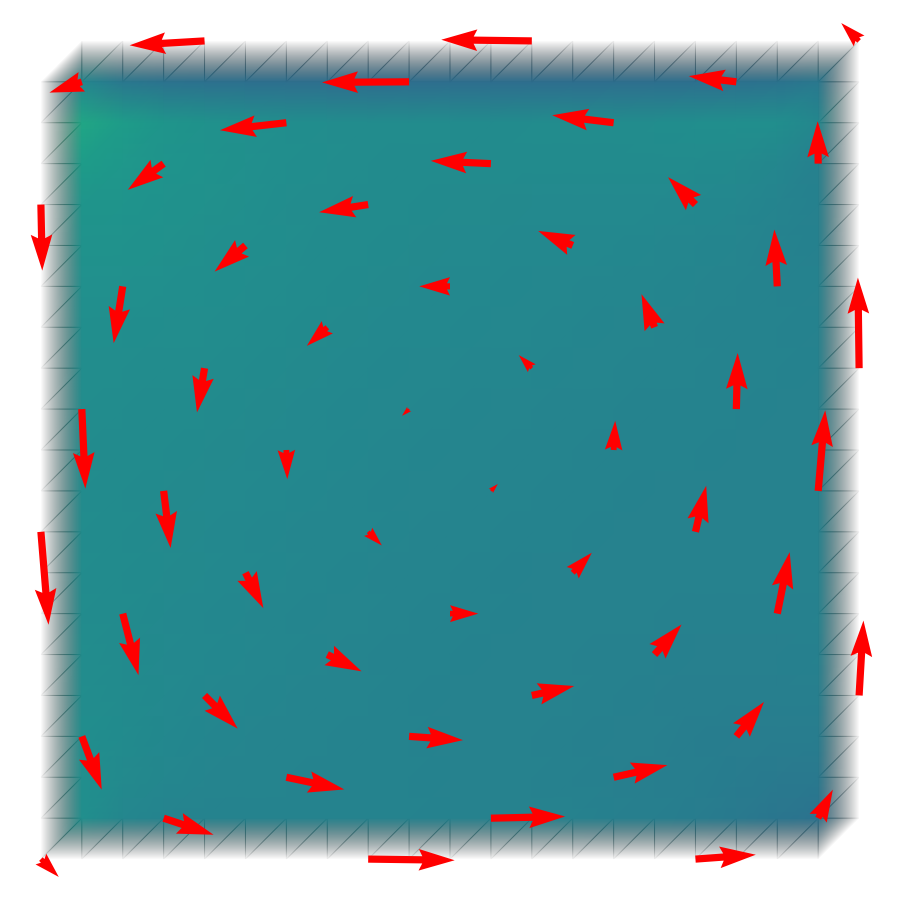}
       \end{minipage}
       \begin{minipage}{0.19\textwidth}
           \centering
           \includegraphics[width=\linewidth]{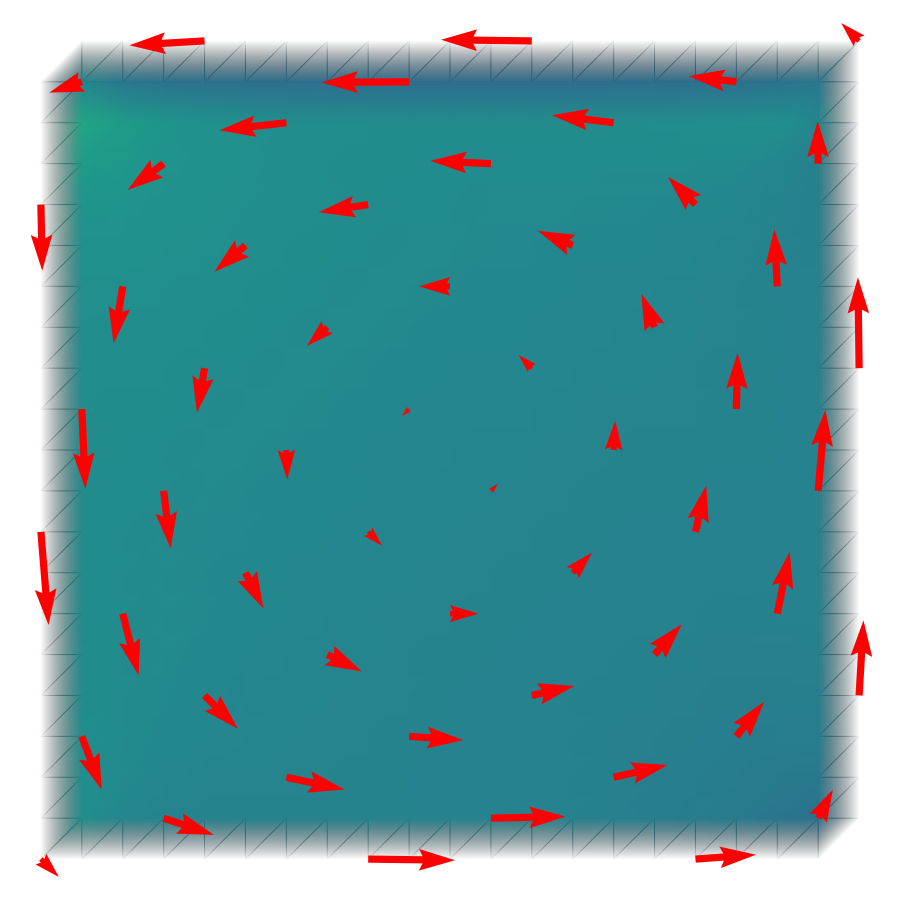}
       \end{minipage}
       \begin{minipage}{0.19\textwidth}
           \centering
           \includegraphics[width=\linewidth]{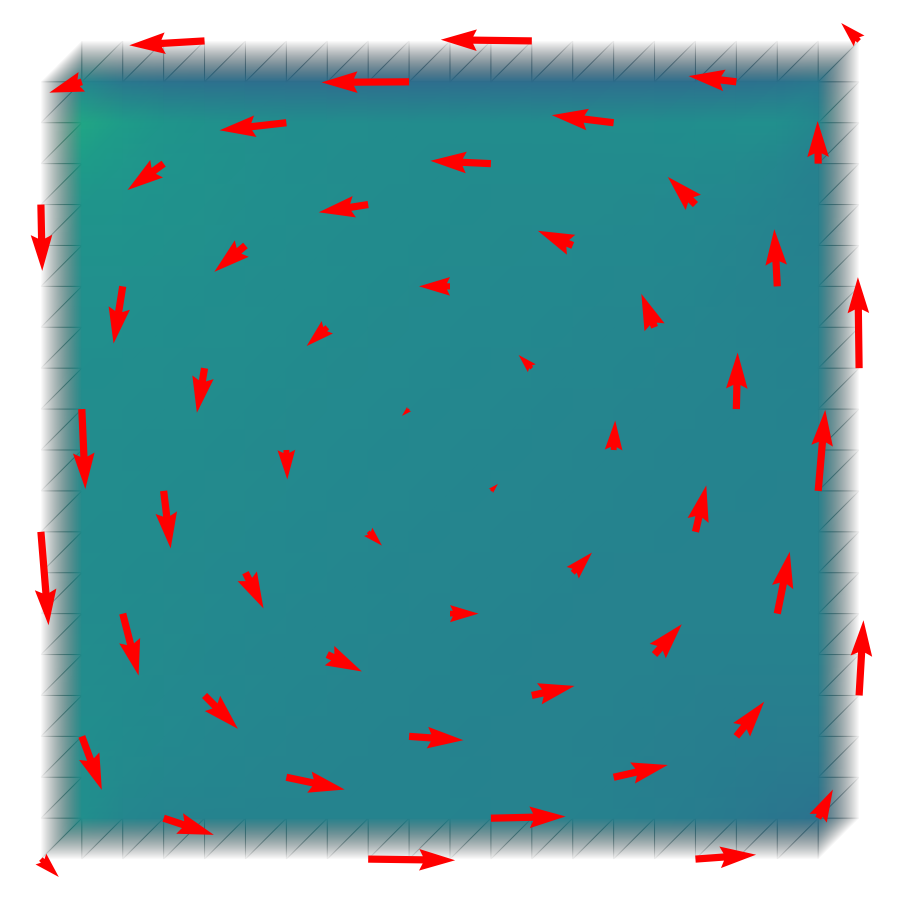}
       \end{minipage}
       \begin{minipage}{0.19\textwidth}
           \centering
           \includegraphics[width=\linewidth]{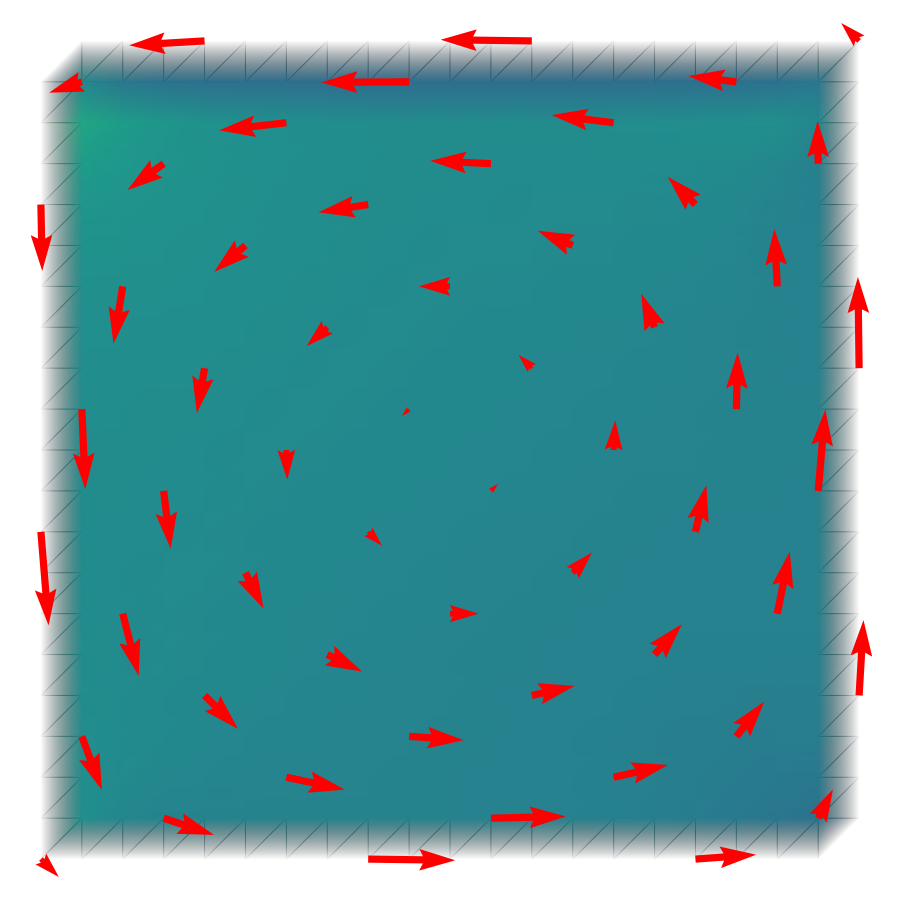}
       \end{minipage}
       \begin{minipage}{0.19\textwidth}
           \centering
           \includegraphics[width=\linewidth]{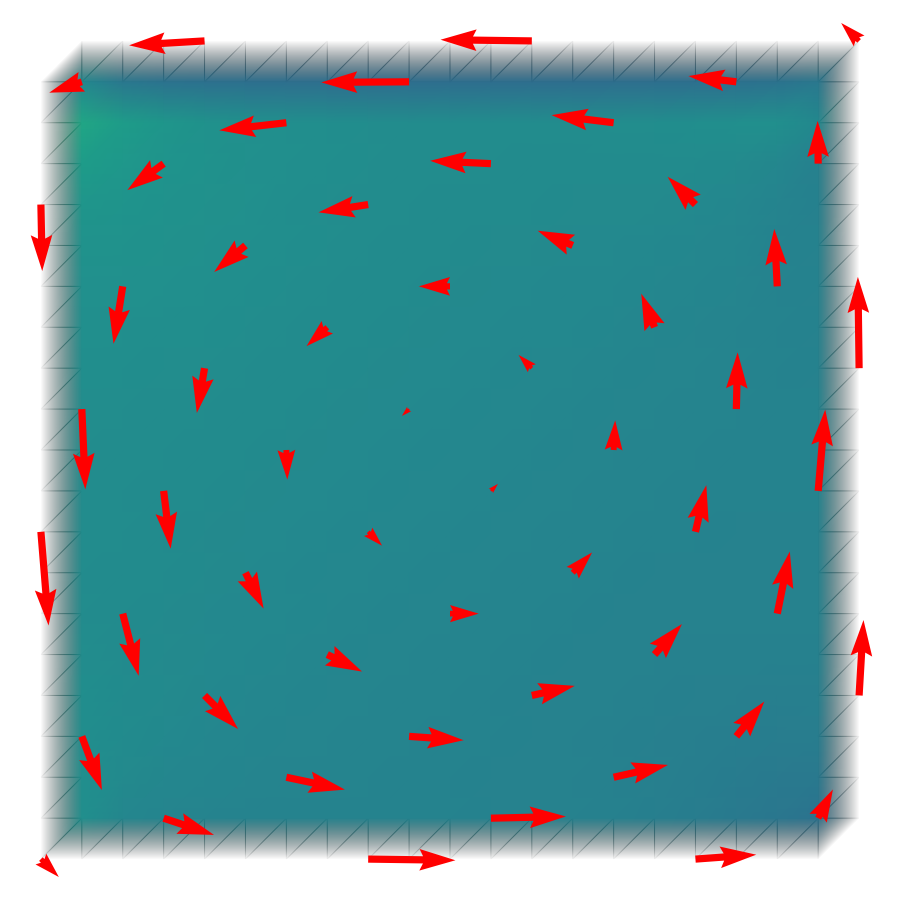}
       \end{minipage}

       \vspace{5pt}
       \textbf{H = 0.3} \\[2pt]
       \begin{minipage}{0.19\textwidth}
           \centering
           \includegraphics[width=\linewidth]{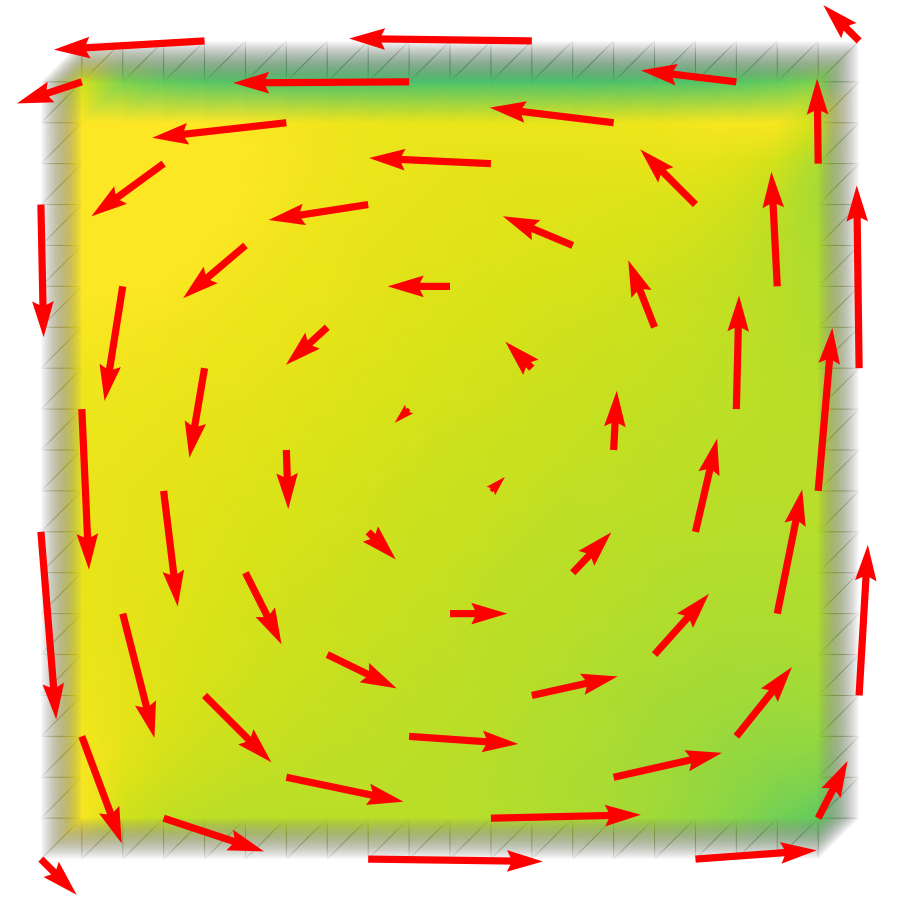}
       \end{minipage}
       \begin{minipage}{0.19\textwidth}
           \centering
           \includegraphics[width=\linewidth]{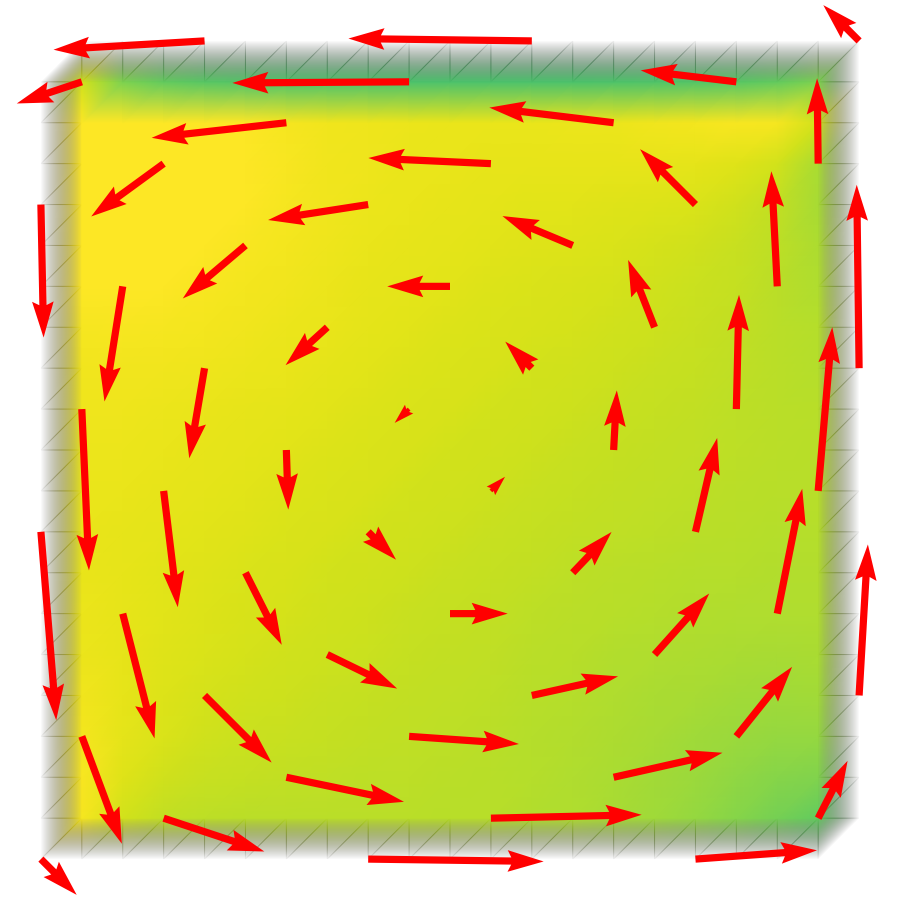}
       \end{minipage}
       \begin{minipage}{0.19\textwidth}
           \centering
           \includegraphics[width=\linewidth]{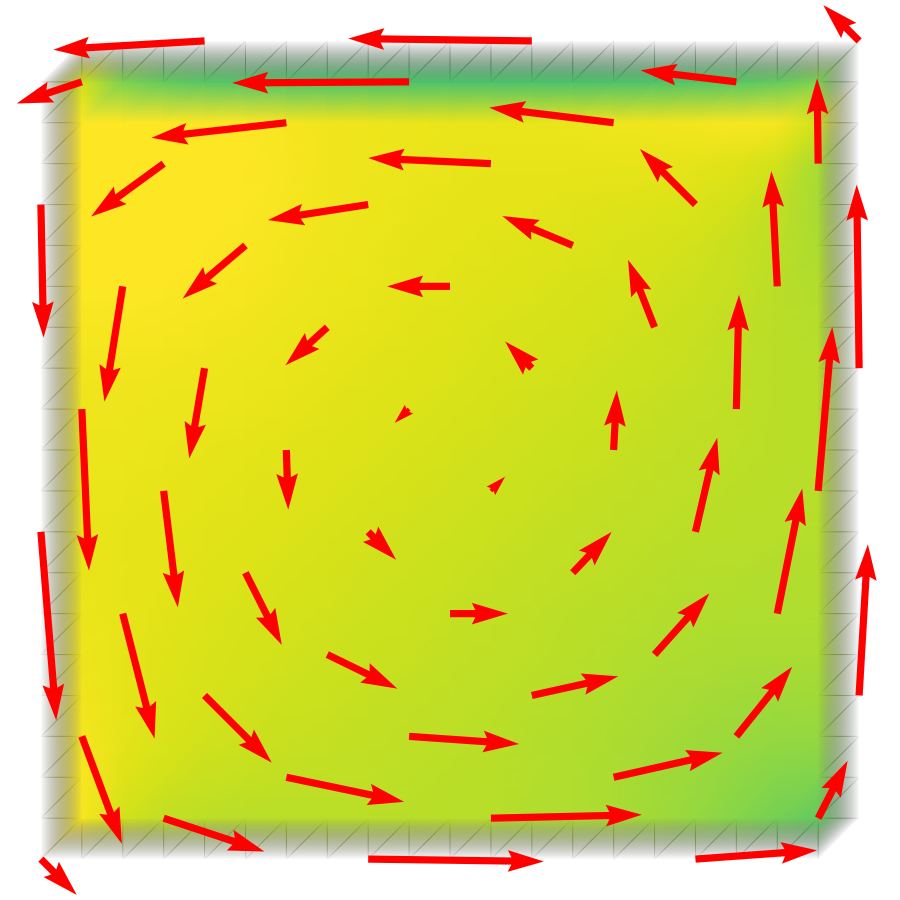}
       \end{minipage}
       \begin{minipage}{0.19\textwidth}
           \centering
           \includegraphics[width=\linewidth]{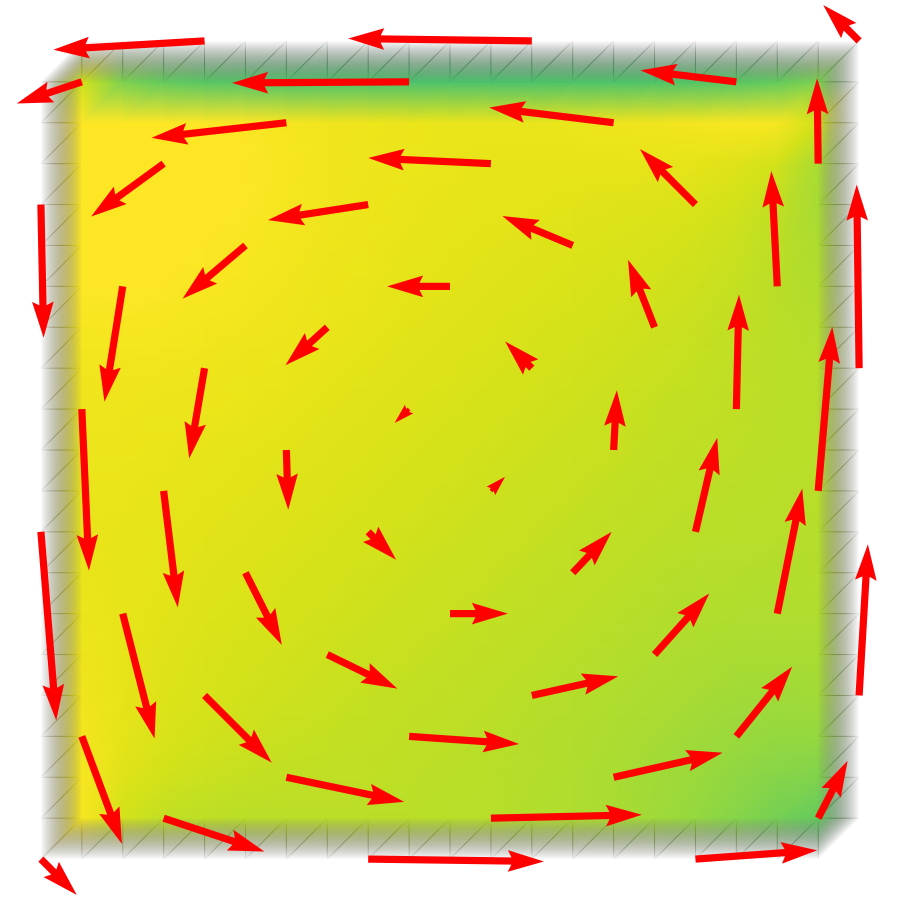}
       \end{minipage}
       \begin{minipage}{0.19\textwidth}
           \centering
           \includegraphics[width=\linewidth]{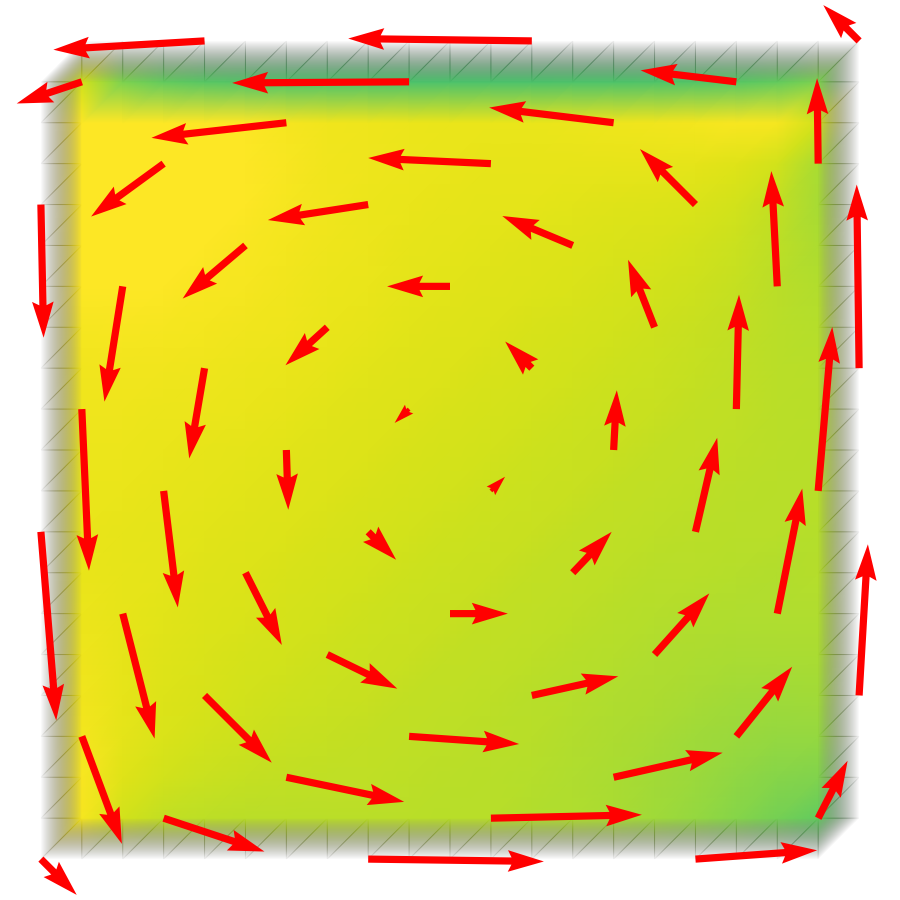}
       \end{minipage}
   \end{minipage}%
   \hspace{2pt}
   \begin{minipage}{0.035\textwidth}
       \centering
       \includegraphics[height=0.24\textheight]{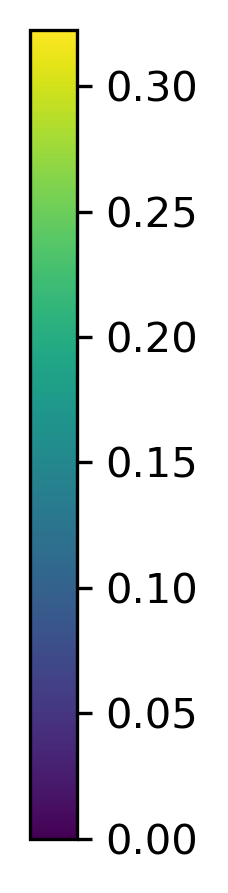}
   \end{minipage}
\caption{
Snapshots at $0.96\,T_c$ ($T_c$: critical temperature), 
initialized with a small amplitude ($|\psi(0)| \approx 0.1$). 
Rows 1 and 2 show the superconducting order parameter magnitude $|\psi|$ 
for $H=0.15$ (row 1) and $H=0.3$ (row 2), 
at times $t=20, 40, 60, 80, 100$. 
Rows 3 and 4 show the corresponding magnetic field distribution $\mathrm{curl}\,A$ 
for $H=0.15$ (row 3) and $H=0.3$ (row 4). 
Comparison of rows 1–2 indicates that the decay of $|\psi|$ 
is faster and reaches smaller values at $H=0.3$. 
Arrows indicate the direction of the vector potential $A$. 
The two blocks (upper and lower panels) correspond to two independent datasets (fig1 and figA1)
}
   \label{fig.1}
\end{figure}

The second numerical experiment (Figure~\ref{fig.1}) is conducted near the critical temperature $T_c$. 
The effect of the magnetic field on superconductivity at $T=0.96T_c$ is simulated in the domain $\Omega = (-\pi,\pi)$ 
with fixed time step $\tau = 0.5$. The initial states are $\psi = 0.08+i0.06$ and $A\approx(0,0)$. 
Two cases with $H = 0.15$ and $H=0.3$ are compared. 
By comparing the first and second rows in Figure~\ref{fig.1}, one observes that at the same times 
a larger $H$ yields a smaller average order parameter, which is further illustrated in Figure~\ref{fig.sup}. 
The results indicate that the decay of the order parameter is noticeably slower for smaller $H$, 
suggesting that weaker magnetic fields delay the breakdown of superconductivity.

\begin{figure}[htbp!]
    \centering
    \includegraphics[width=0.45\linewidth]{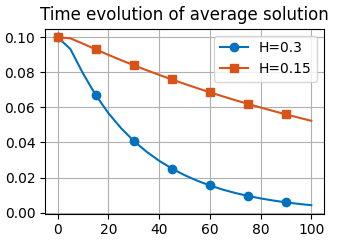}
\caption{
Time evolution of the spatially averaged order parameter corresponding to Figures~\ref{fig.3} and \ref{fig.4}. 
In this example, superconductivity is suppressed more strongly by the larger magnetic field, 
leading to smaller values of the order parameter modulus.
}
\label{fig.sup}
\end{figure}
\begin{figure}[!htp]
   \centering
    \begin{minipage}{0.85\textwidth}
        \centering
        \begin{minipage}{0.19\textwidth}
            \centering
            \includegraphics[width=\linewidth]{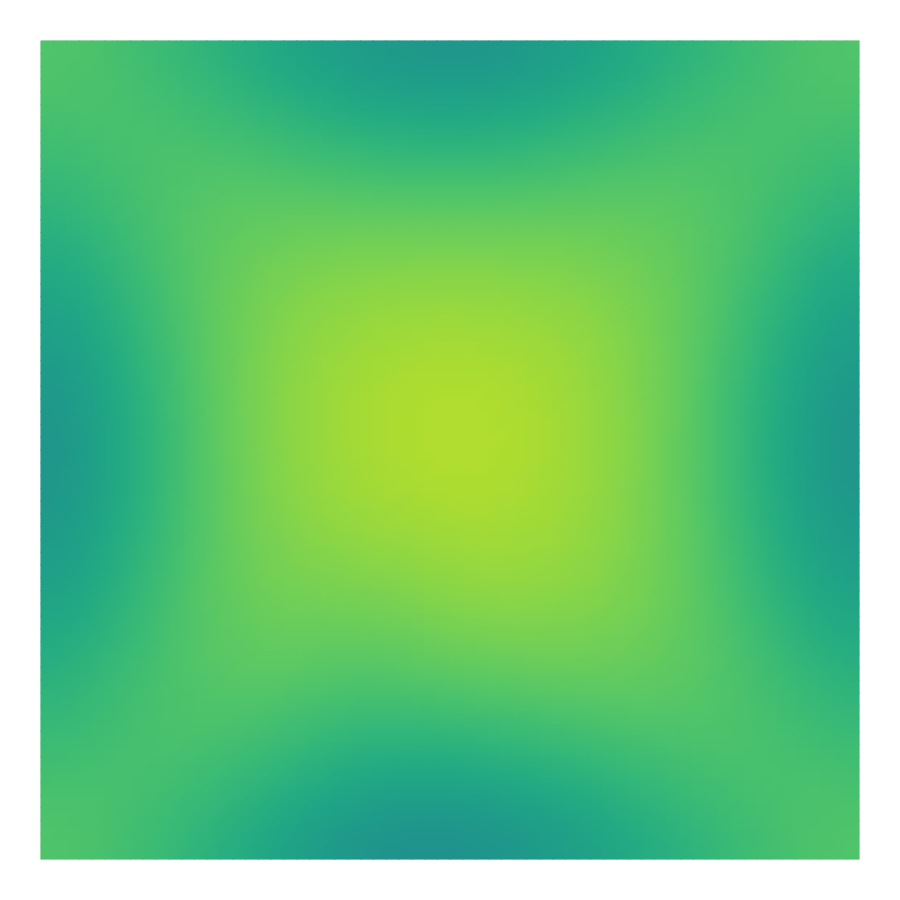}
        \end{minipage}
        \begin{minipage}{0.19\textwidth}
            \centering
            \includegraphics[width=\linewidth]{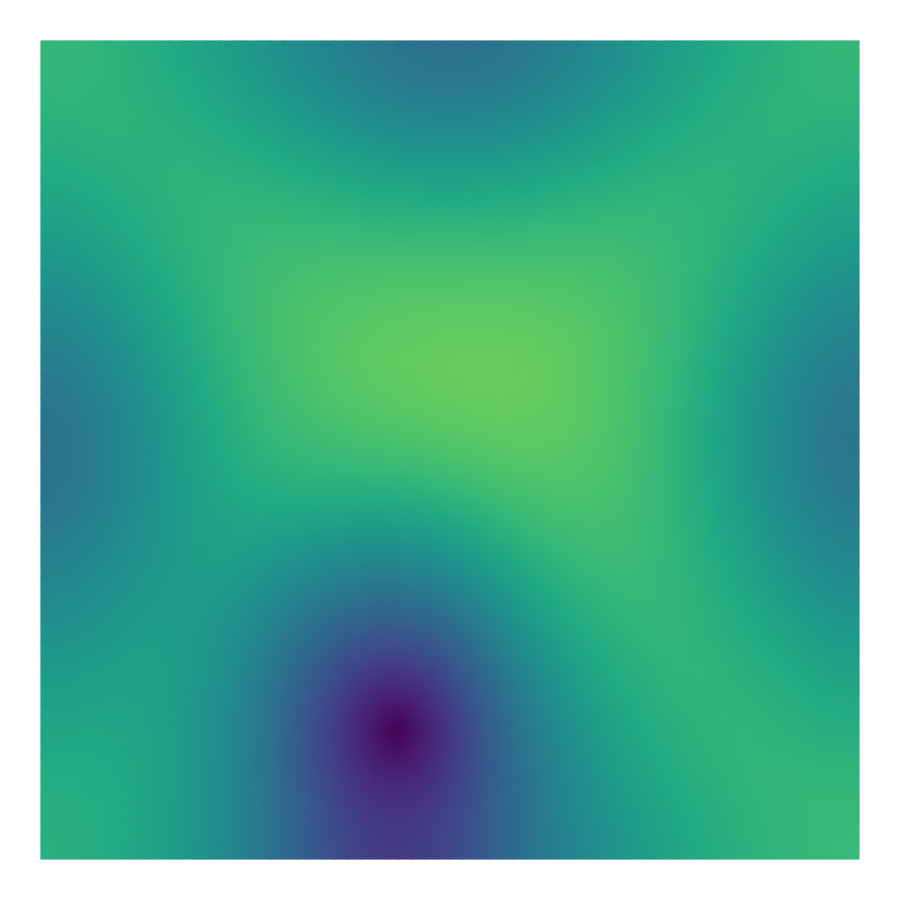}
        \end{minipage}
        \begin{minipage}{0.19\textwidth}
            \centering
            \includegraphics[width=\linewidth]{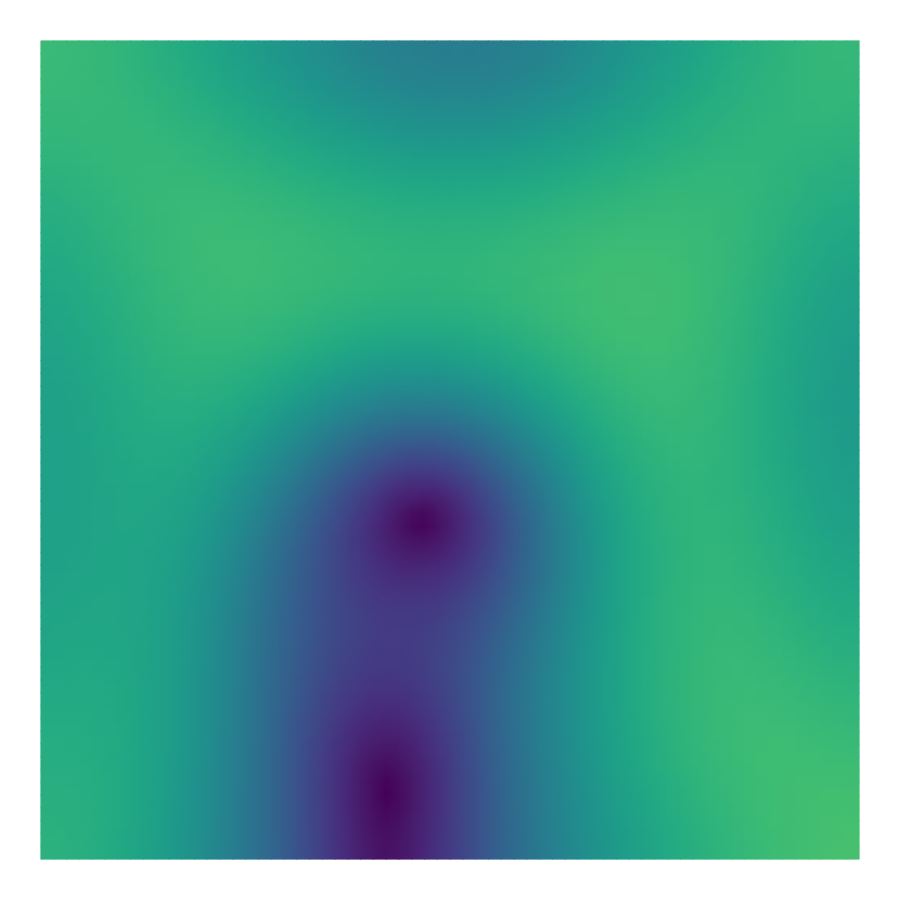}
        \end{minipage}
        \begin{minipage}{0.19\textwidth}
            \centering
            \includegraphics[width=\linewidth]{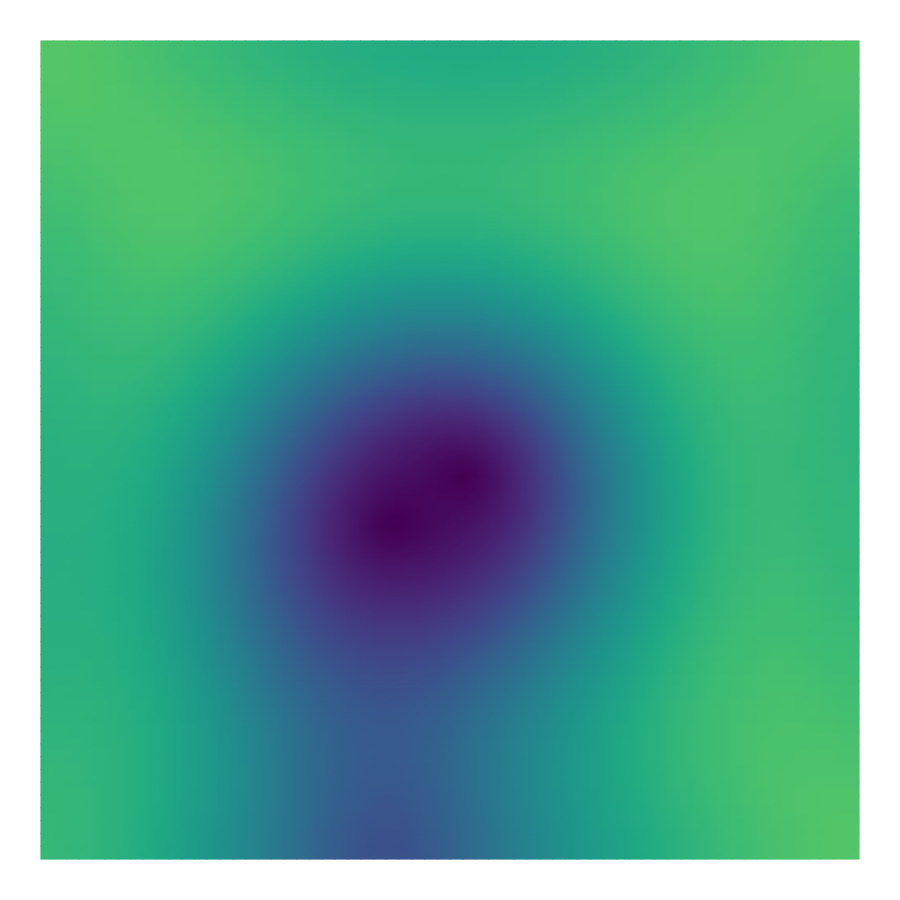}
        \end{minipage}
        \begin{minipage}{0.19\textwidth}
            \centering
            \includegraphics[width=\linewidth]{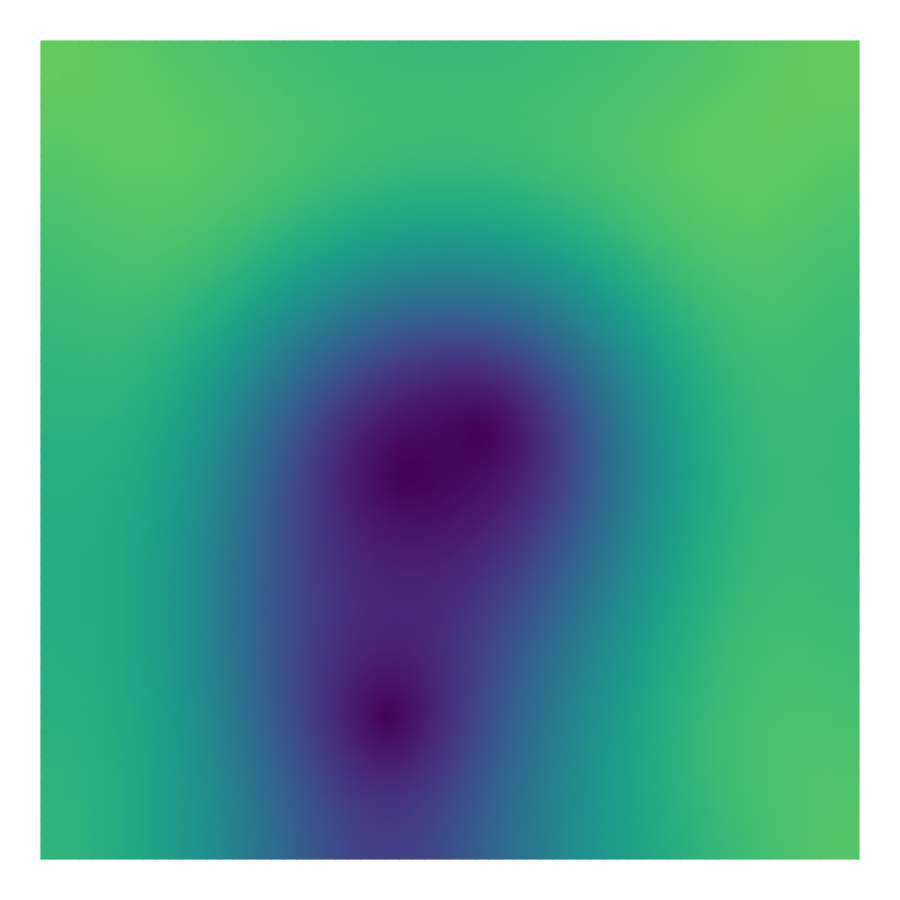}
        \end{minipage}

        \vspace{5pt}
        \begin{minipage}{0.19\textwidth}
            \centering
            \includegraphics[width=\linewidth]{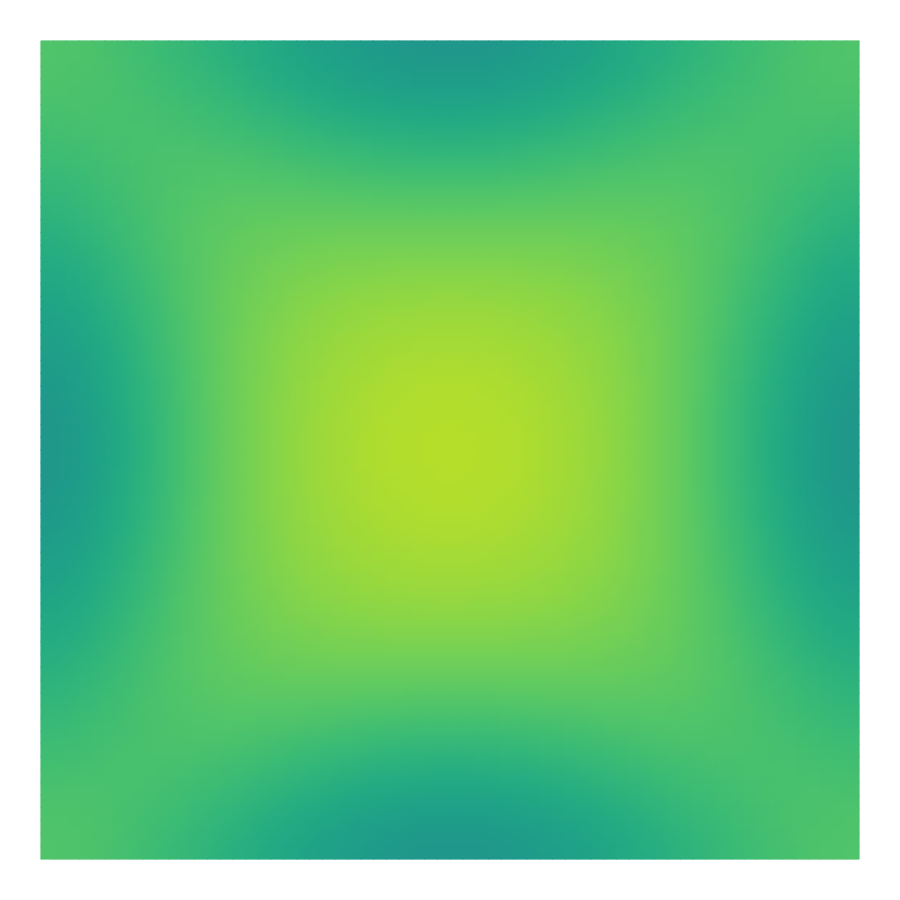}
        \end{minipage}
        \begin{minipage}{0.19\textwidth}
            \centering
            \includegraphics[width=\linewidth]{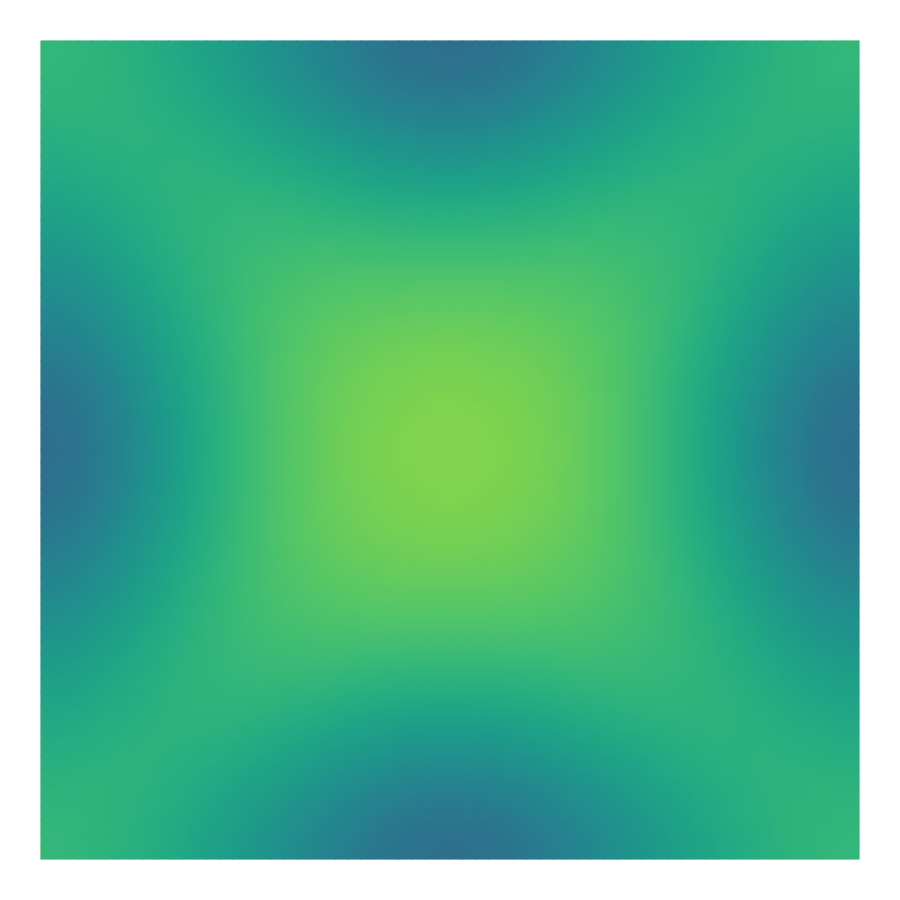}
        \end{minipage}
        \begin{minipage}{0.19\textwidth}
            \centering
            \includegraphics[width=\linewidth]{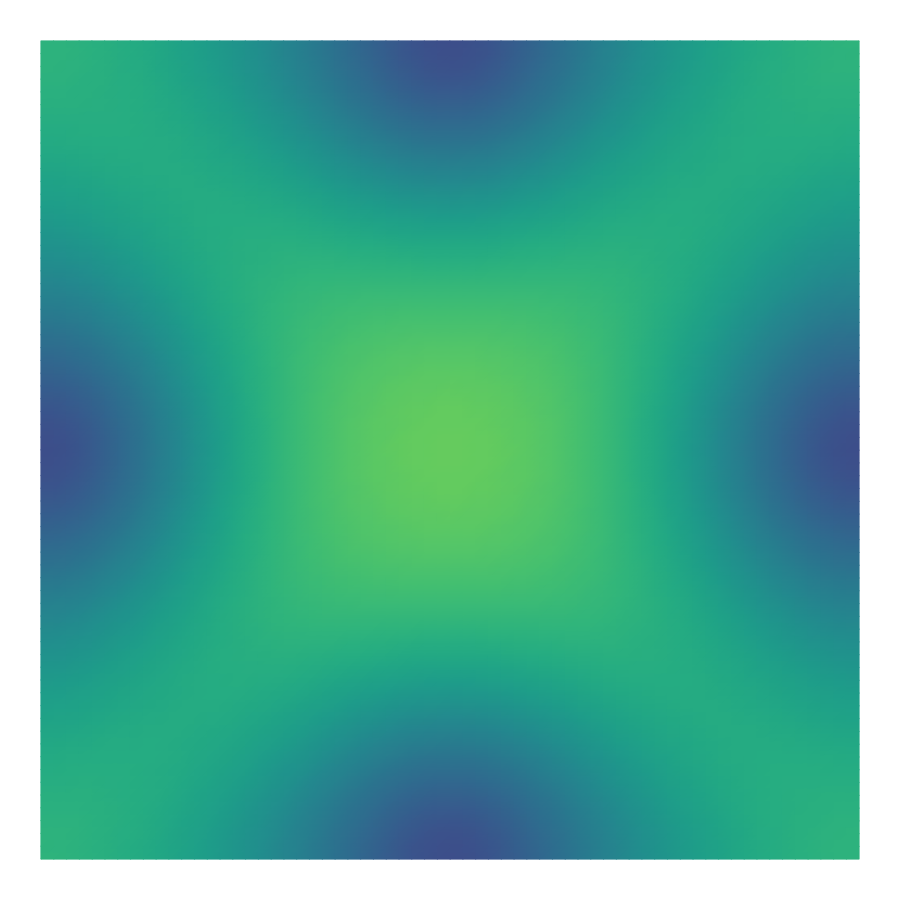}
        \end{minipage}
        \begin{minipage}{0.19\textwidth}
            \centering
            \includegraphics[width=\linewidth]{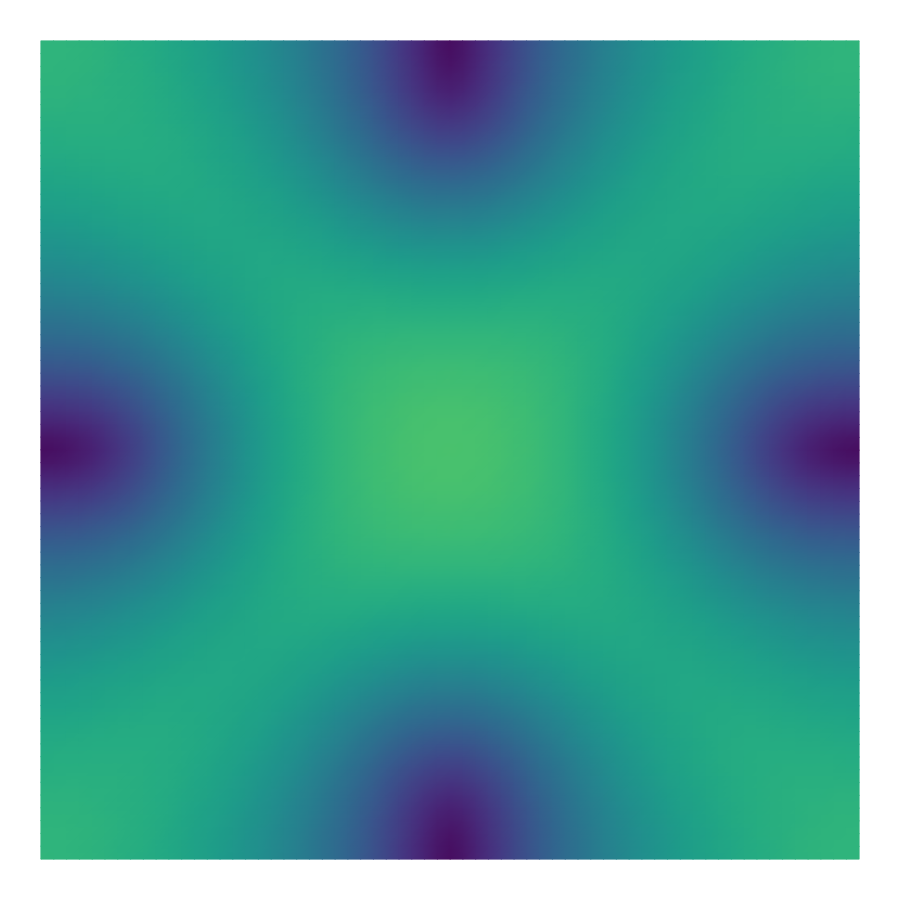}
        \end{minipage}
        \begin{minipage}{0.19\textwidth}
            \centering
            \includegraphics[width=\linewidth]{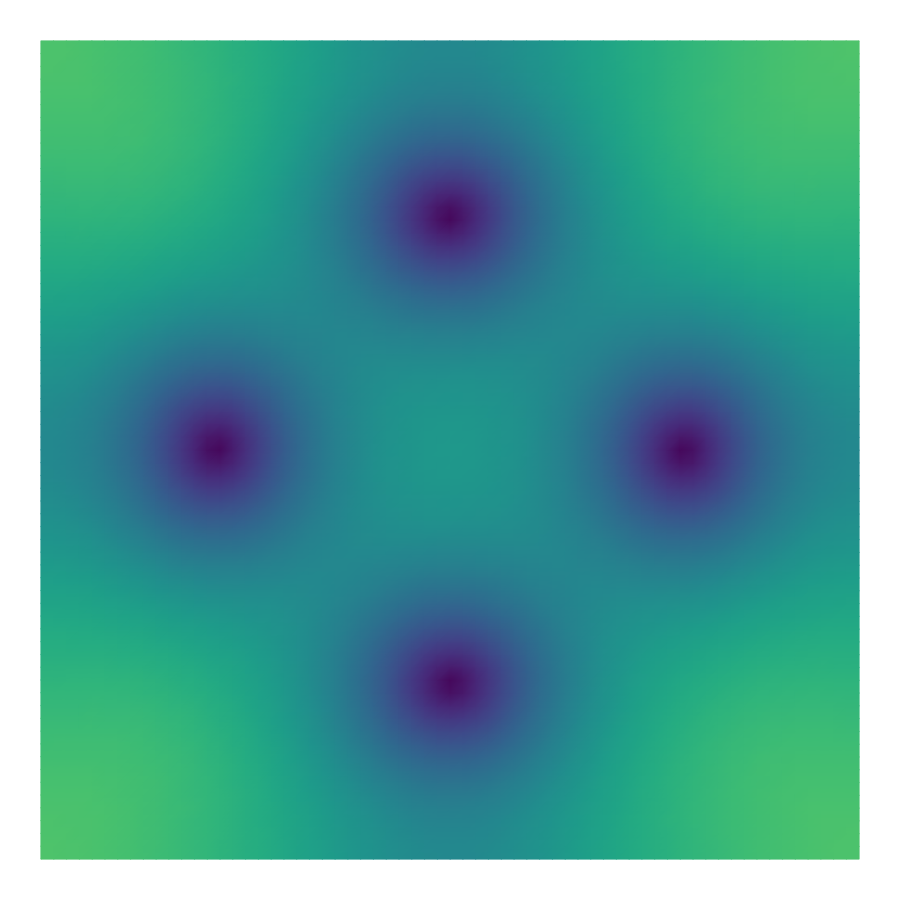}
        \end{minipage}
    \end{minipage}%
    \hspace{2pt}
    \begin{minipage}{0.035\textwidth}
        \centering
        \includegraphics[height=0.28\textheight]{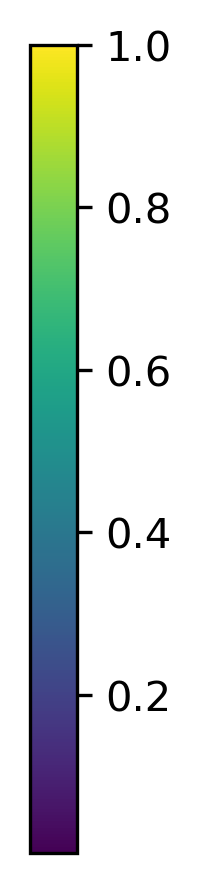}
    \end{minipage}

            \begin{minipage}{0.85\textwidth}
        \centering
        \begin{minipage}{0.19\textwidth}
            \centering
            \includegraphics[width=\linewidth]{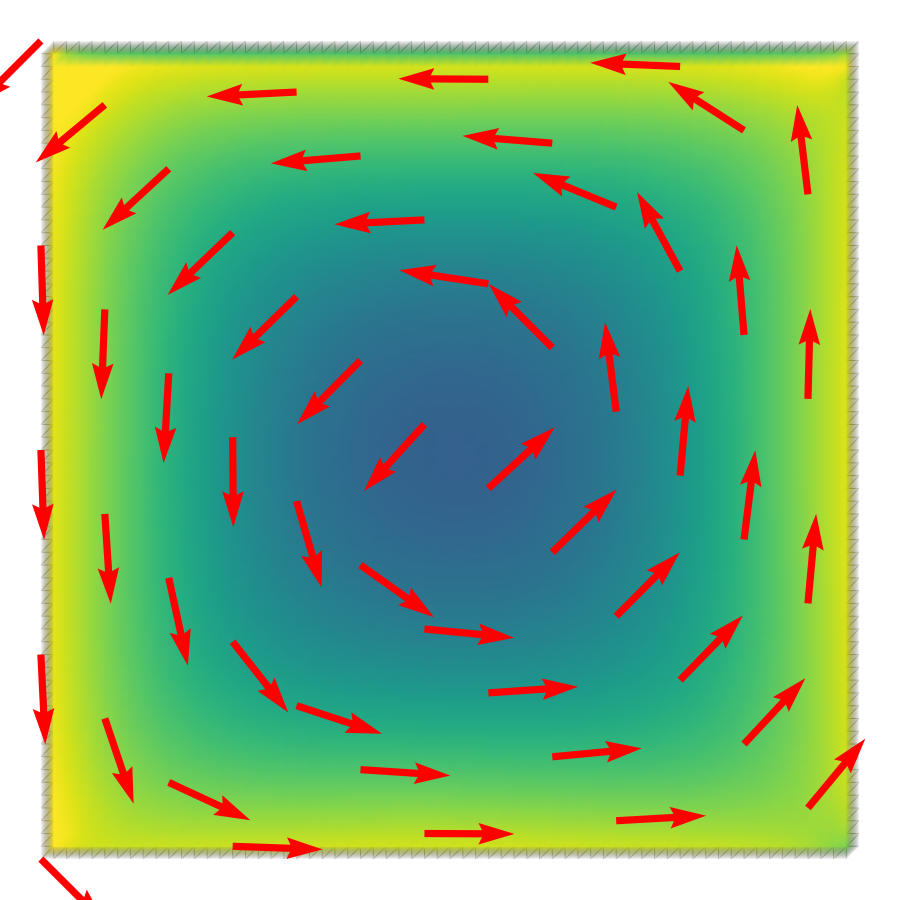}
        \end{minipage}
        \begin{minipage}{0.19\textwidth}
            \centering
            \includegraphics[width=\linewidth]{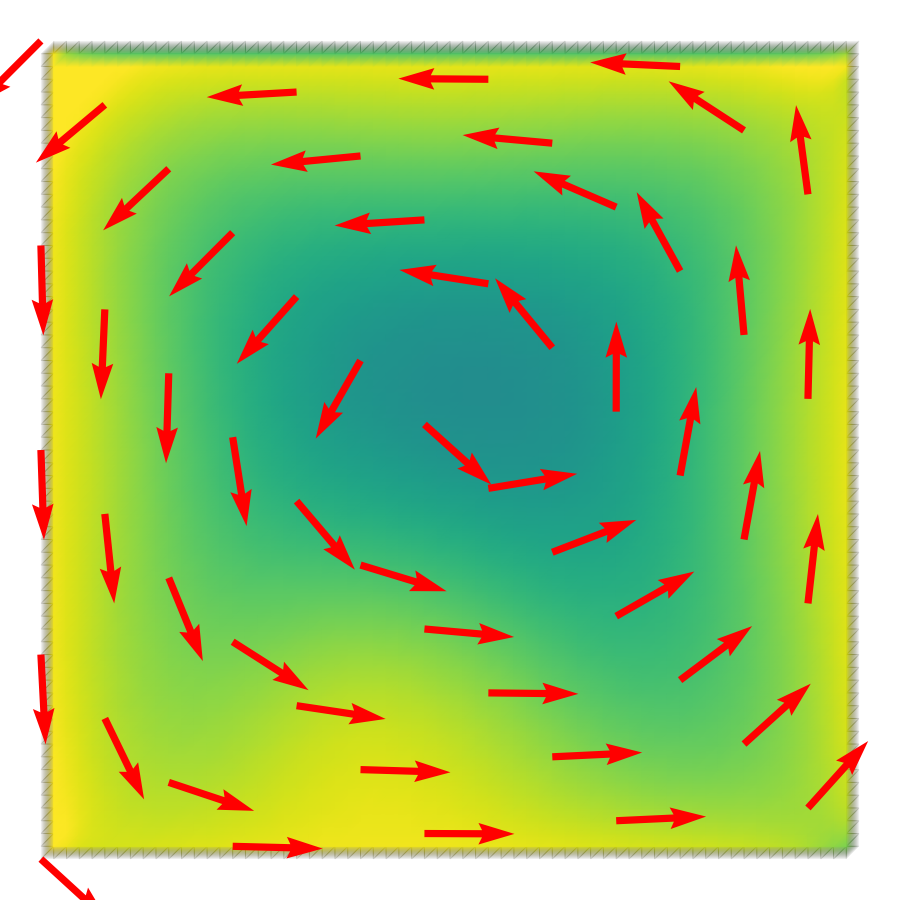}
        \end{minipage}
        \begin{minipage}{0.19\textwidth}
            \centering
            \includegraphics[width=\linewidth]{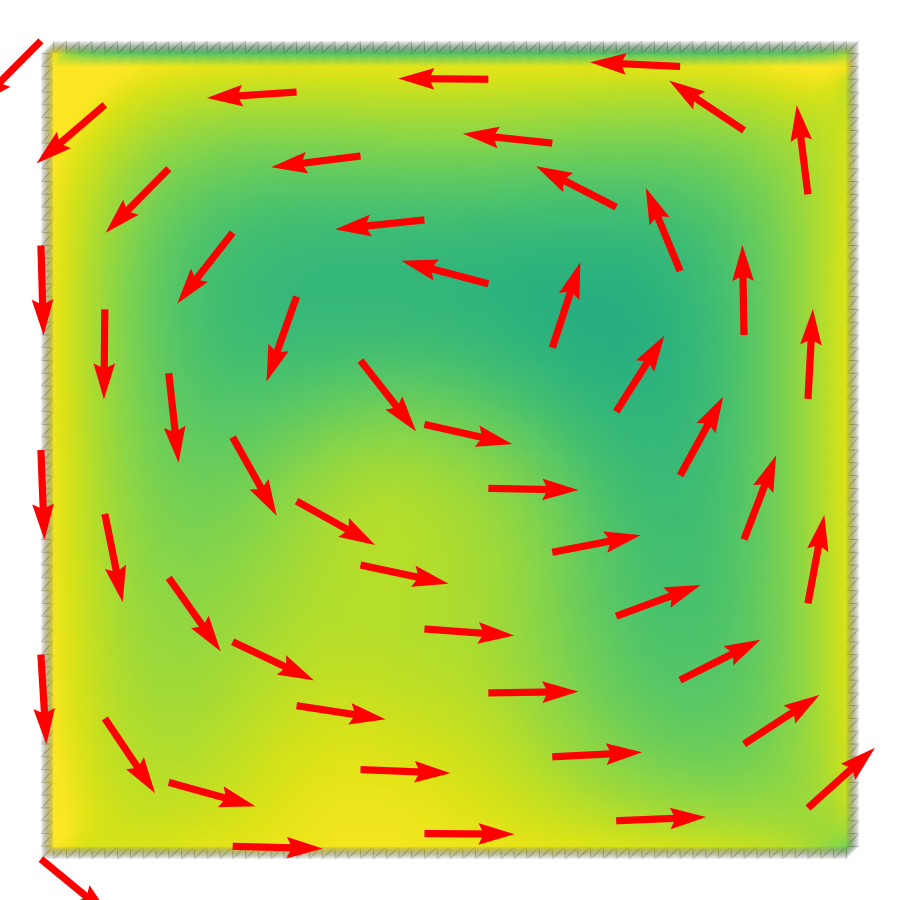}
        \end{minipage}
        \begin{minipage}{0.19\textwidth}
            \centering
            \includegraphics[width=\linewidth]{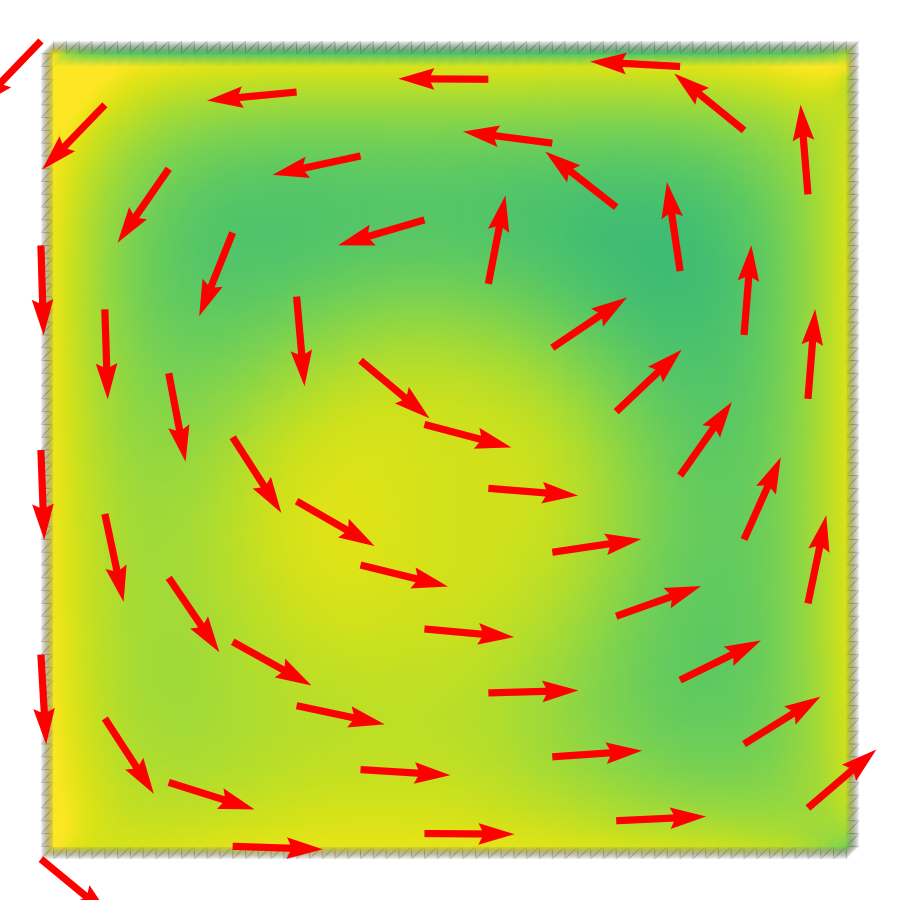}
        \end{minipage}
        \begin{minipage}{0.19\textwidth}
            \centering
            \includegraphics[width=\linewidth]{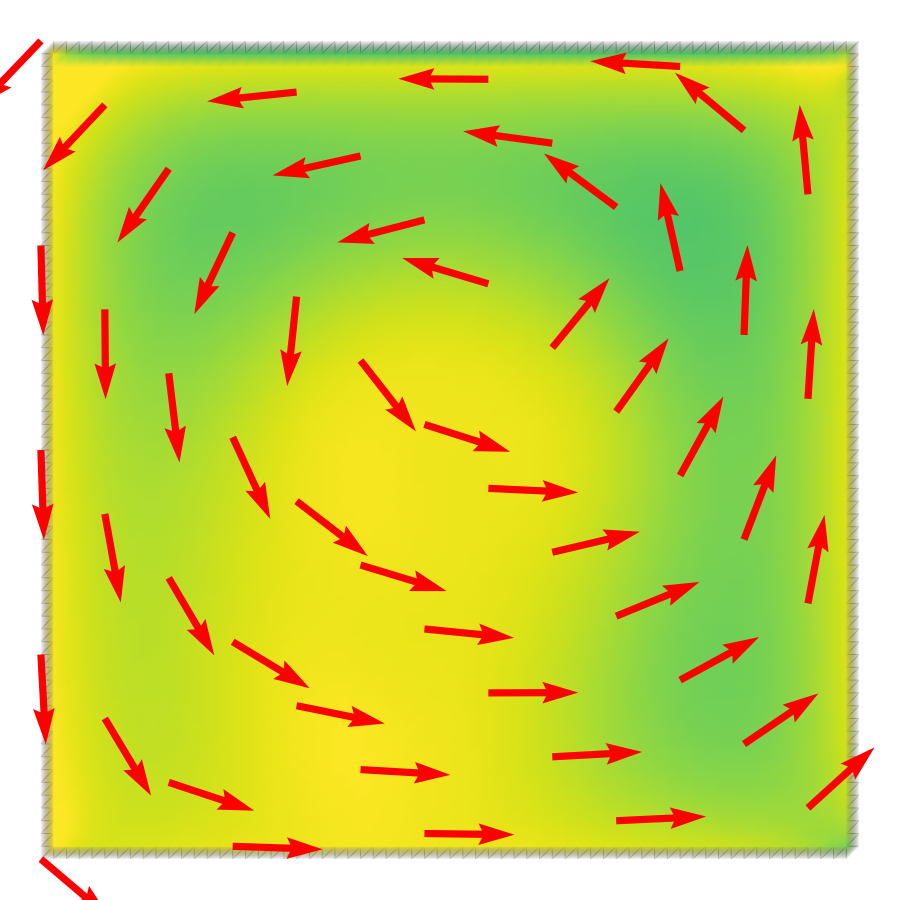}
        \end{minipage}

        \vspace{5pt}
        \begin{minipage}{0.19\textwidth}
            \centering
            \includegraphics[width=\linewidth]{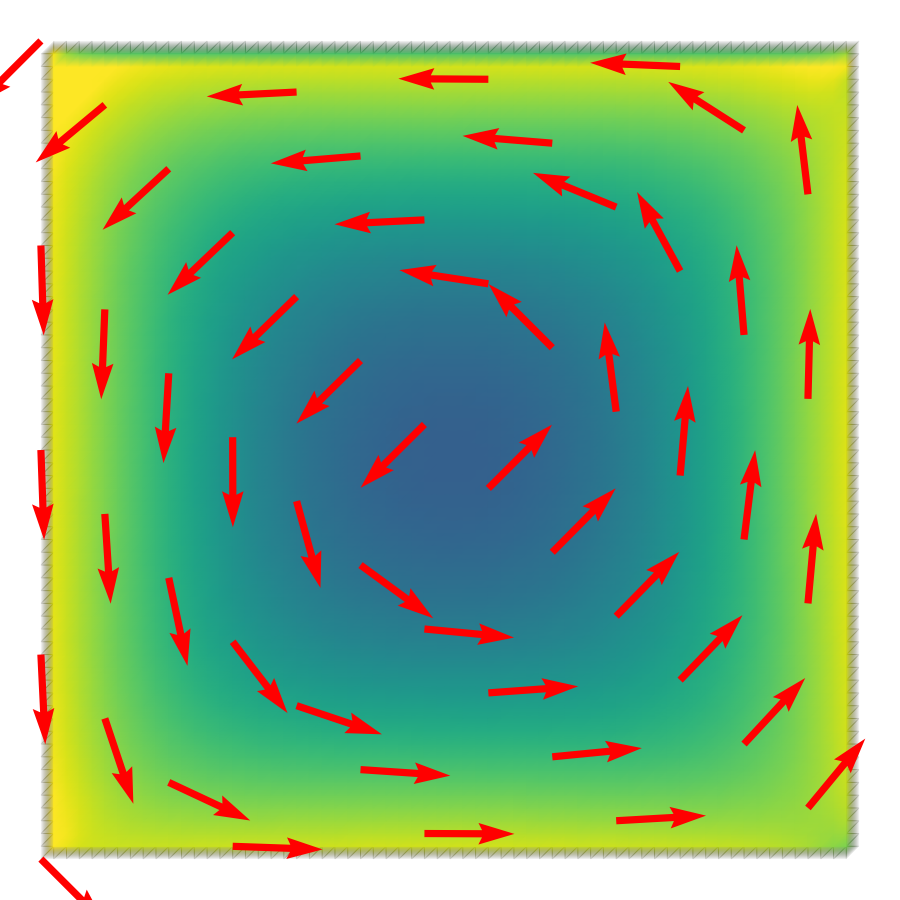}
        \end{minipage}
        \begin{minipage}{0.19\textwidth}
            \centering
            \includegraphics[width=\linewidth]{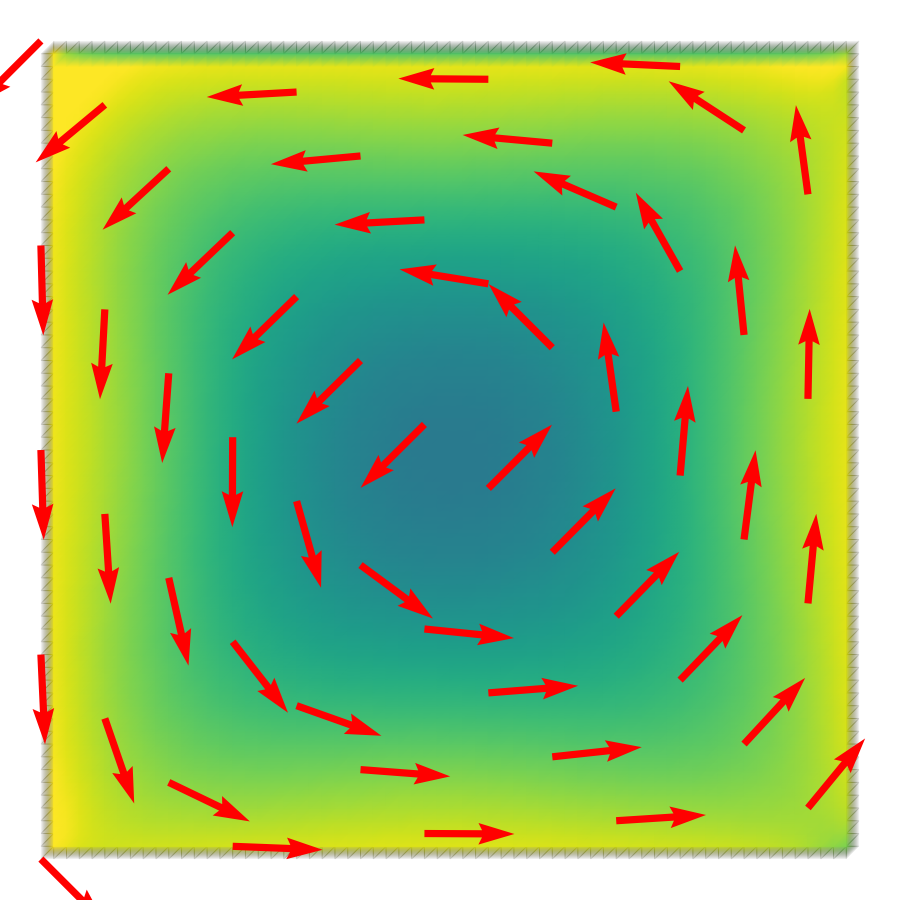}
        \end{minipage}
        \begin{minipage}{0.19\textwidth}
            \centering
            \includegraphics[width=\linewidth]{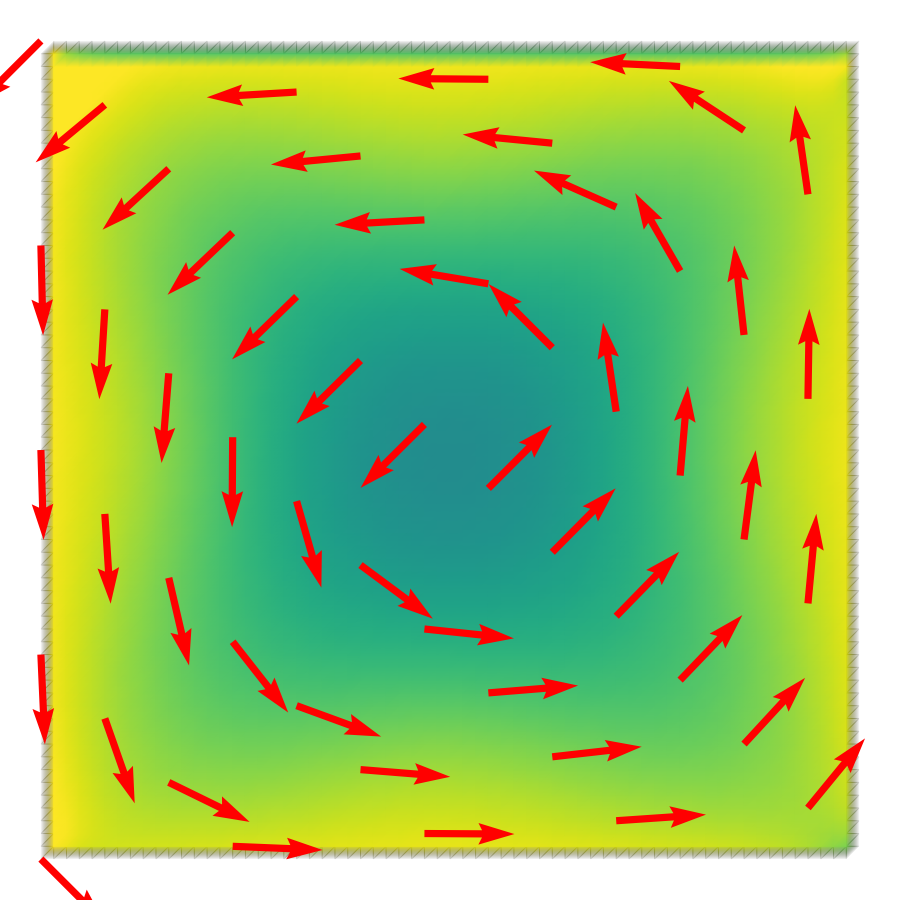}
        \end{minipage}
        \begin{minipage}{0.19\textwidth}
            \centering
            \includegraphics[width=\linewidth]{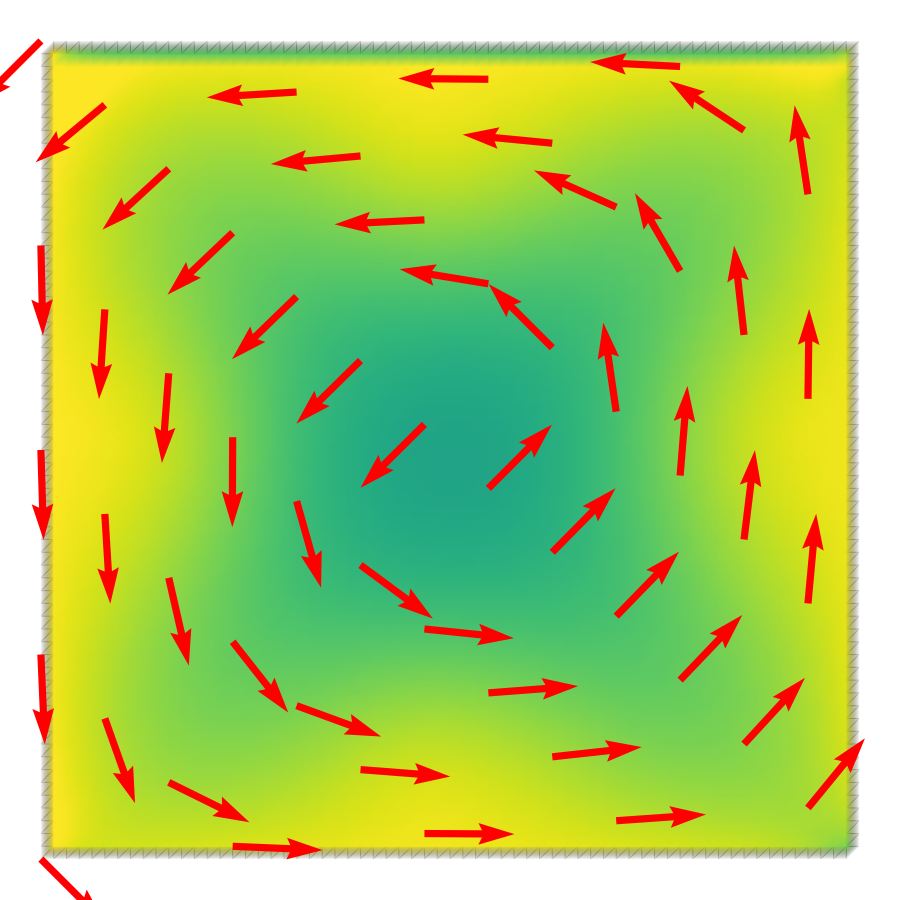}
        \end{minipage}
        \begin{minipage}{0.19\textwidth}
            \centering
            \includegraphics[width=\linewidth]{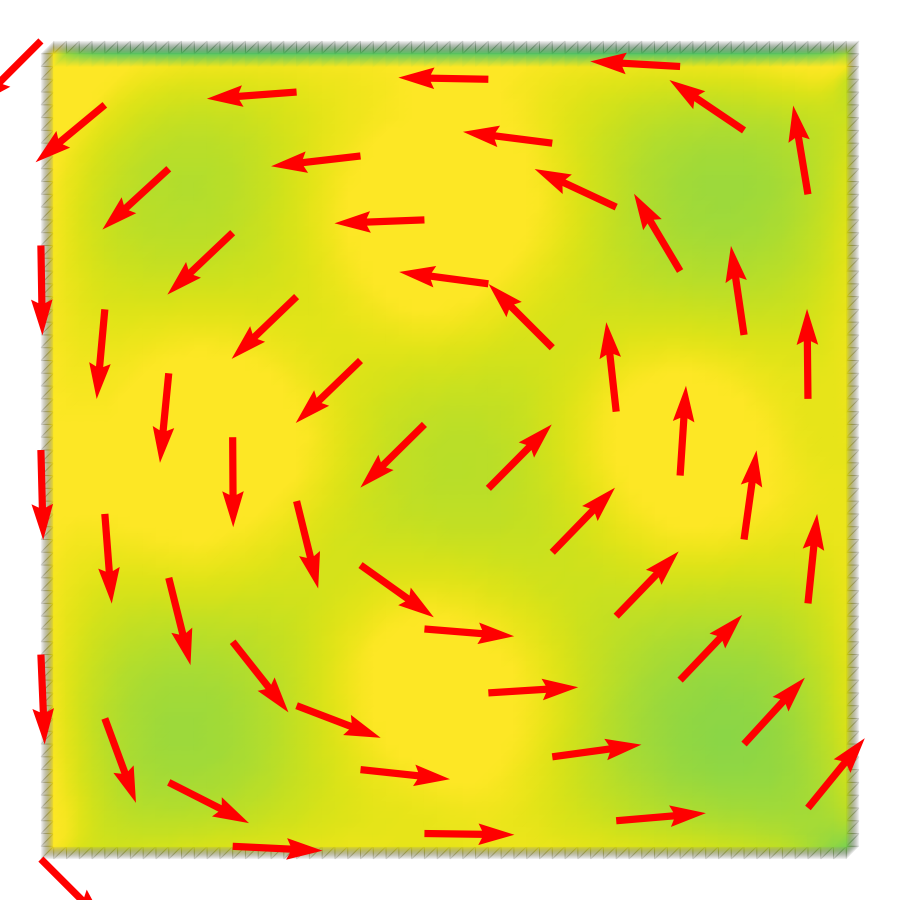}
        \end{minipage}
    \end{minipage}%
    \hspace{2pt}
    \begin{minipage}{0.035\textwidth}
        \centering
        \includegraphics[height=0.24\textheight]{figures/figA1_outputs/figA1_colorbar.png}
    \end{minipage}
    \caption{Time evolution of vortex structures of $|\psi|$ and $curl A$ in $(-\pi,\pi)$ with GL constant $\kappa = 2$ under an external magnetic field along the (0,0,1) direction (t=20, 40, 60, 80, 100, and the arrows indicate the direction of the vector potential A). The stabilization parameter is $S = 4$. The singular and even rows correspond to the cases with and without the applied inhomogeneity potential, respectively}
    \label{fig.3}
\end{figure}

\begin{figure}[!htbp]
   \centering
    \begin{minipage}{0.85\textwidth}
        \centering
        \begin{minipage}{0.19\textwidth}
            \centering
            \includegraphics[width=\linewidth]{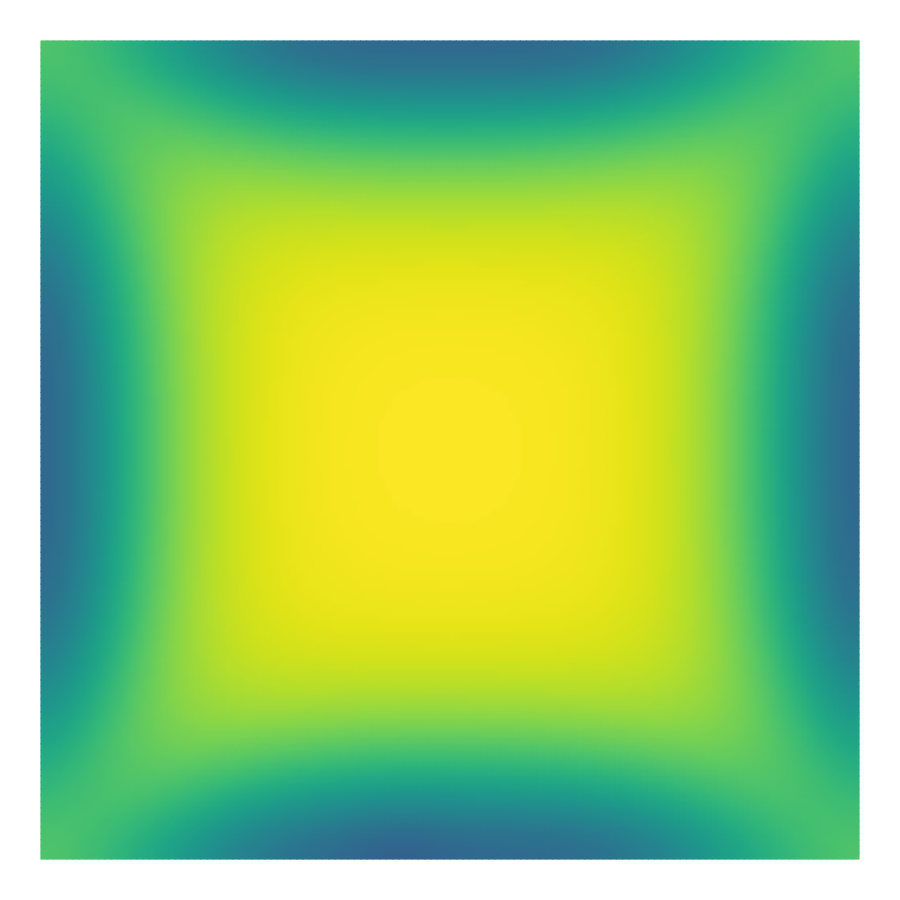}
        \end{minipage}
        \begin{minipage}{0.19\textwidth}
            \centering
            \includegraphics[width=\linewidth]{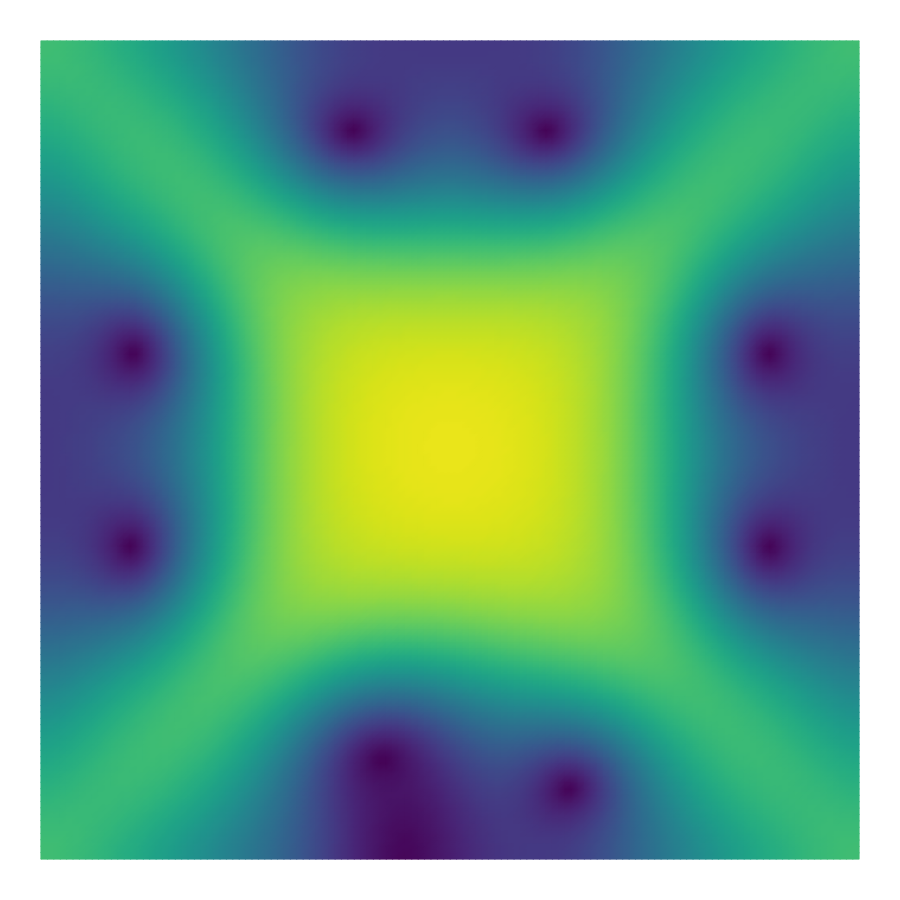}
        \end{minipage}
        \begin{minipage}{0.19\textwidth}
            \centering
            \includegraphics[width=\linewidth]{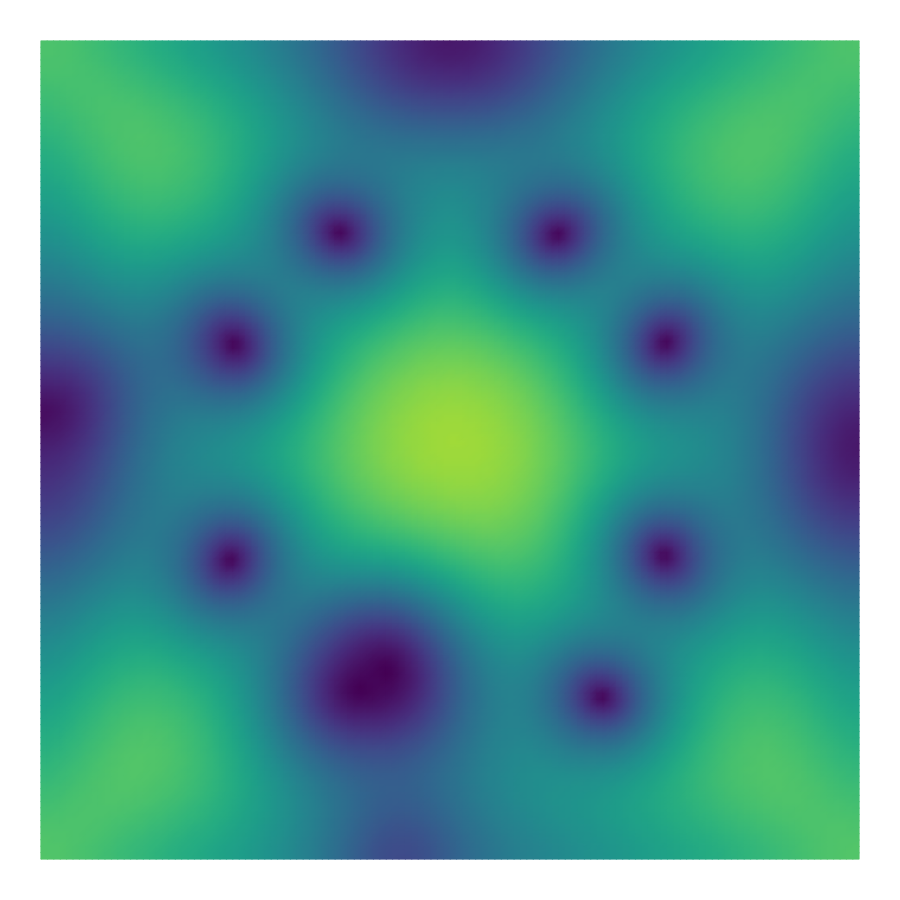}
        \end{minipage}
        \begin{minipage}{0.19\textwidth}
            \centering
            \includegraphics[width=\linewidth]{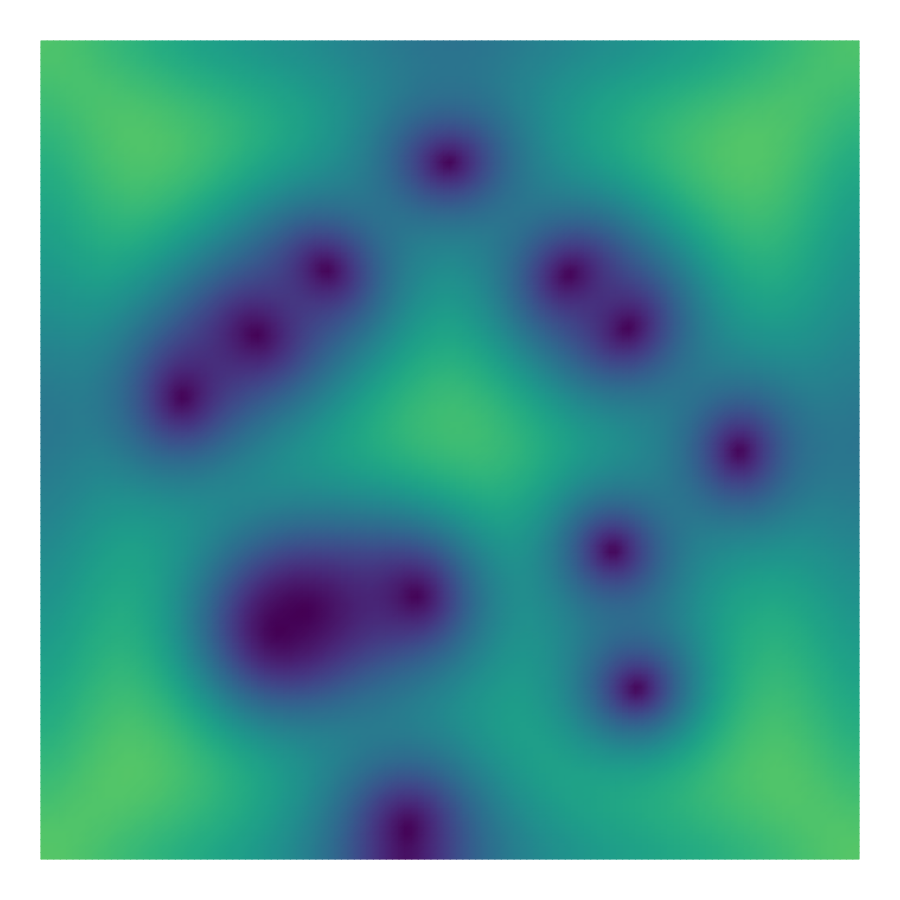}
        \end{minipage}
        \begin{minipage}{0.19\textwidth}
            \centering
            \includegraphics[width=\linewidth]{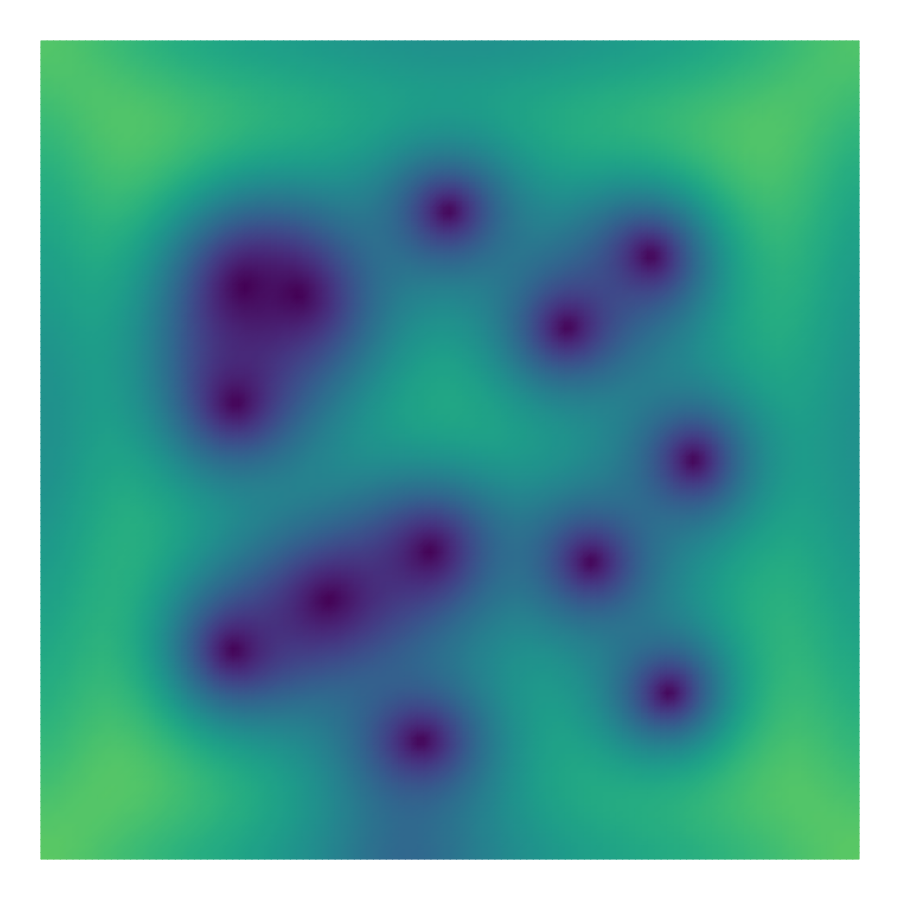}
        \end{minipage}

        \vspace{5pt}
        \begin{minipage}{0.19\textwidth}
            \centering
            \includegraphics[width=\linewidth]{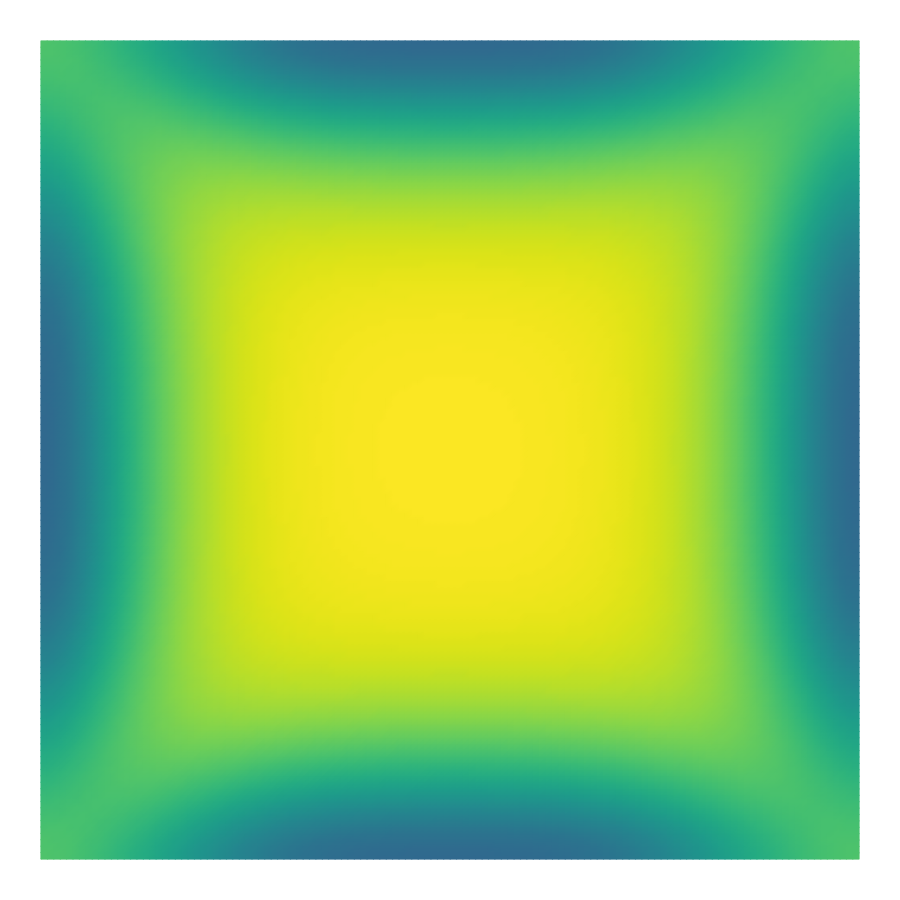}
        \end{minipage}
        \begin{minipage}{0.19\textwidth}
            \centering
            \includegraphics[width=\linewidth]{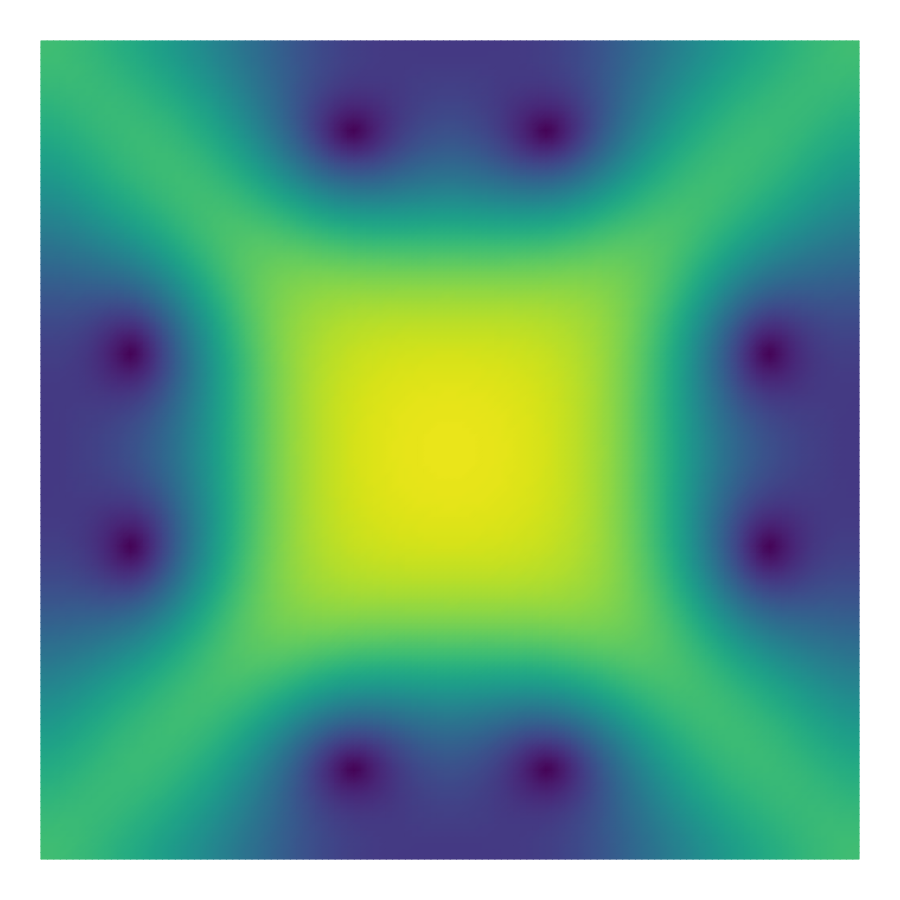}
        \end{minipage}
        \begin{minipage}{0.19\textwidth}
            \centering
            \includegraphics[width=\linewidth]{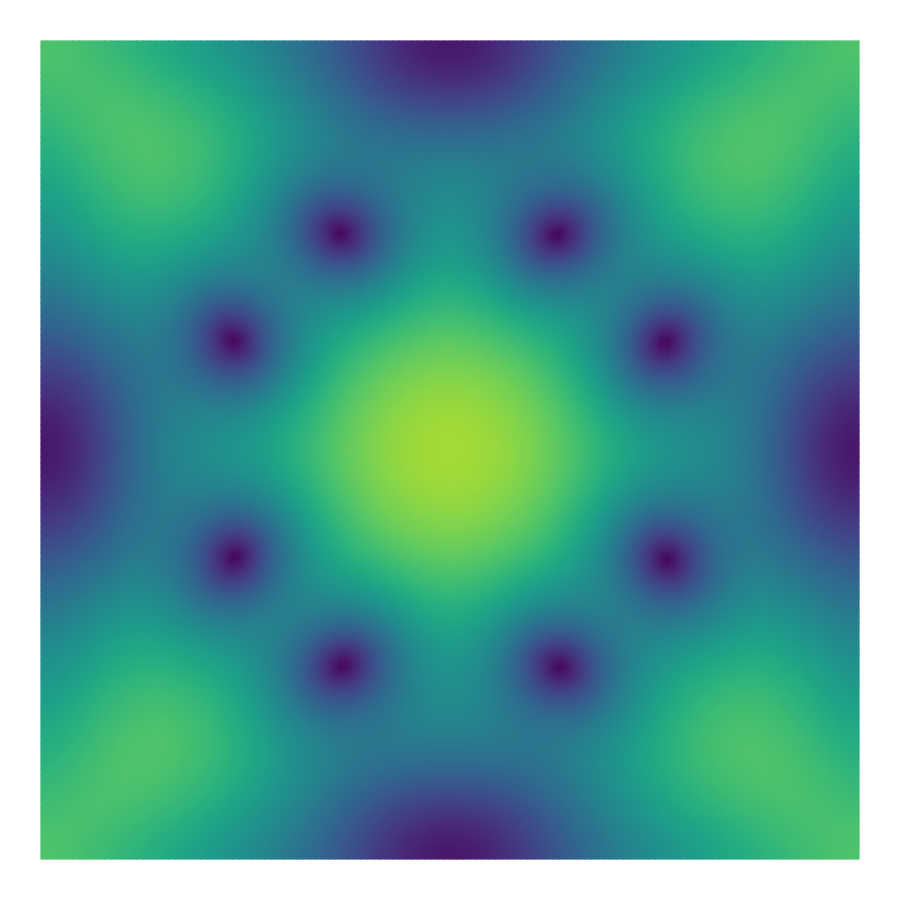}
        \end{minipage}
        \begin{minipage}{0.19\textwidth}
            \centering
            \includegraphics[width=\linewidth]{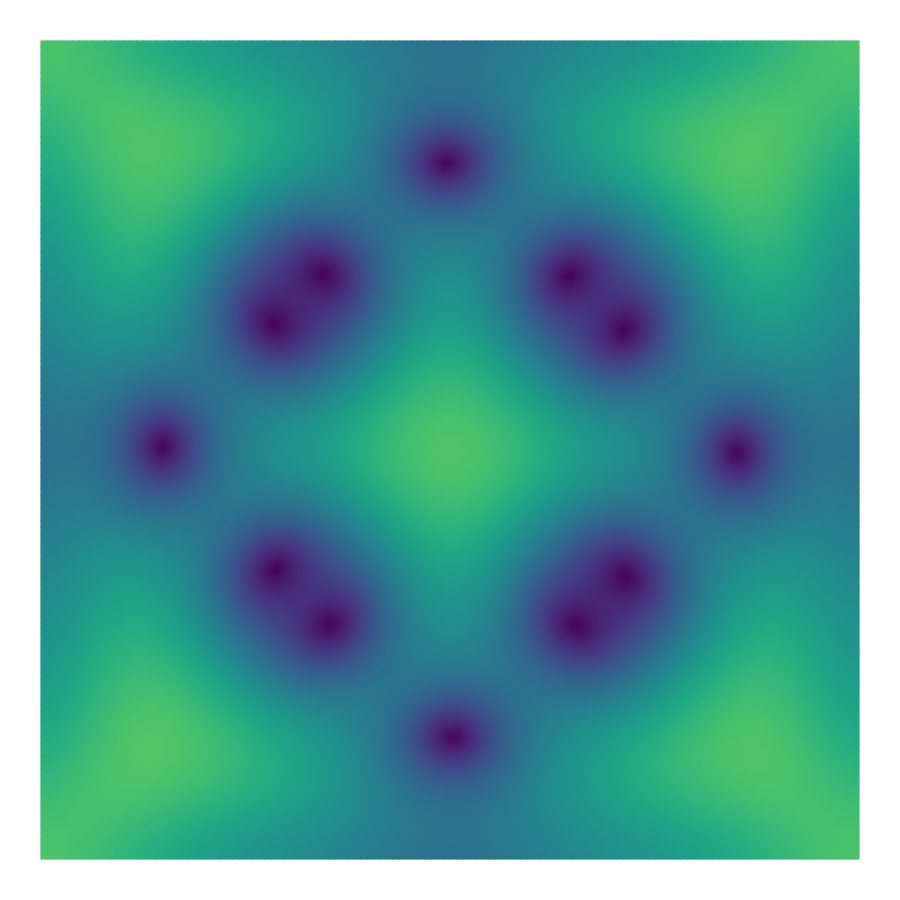}
        \end{minipage}
        \begin{minipage}{0.19\textwidth}
            \centering
            \includegraphics[width=\linewidth]{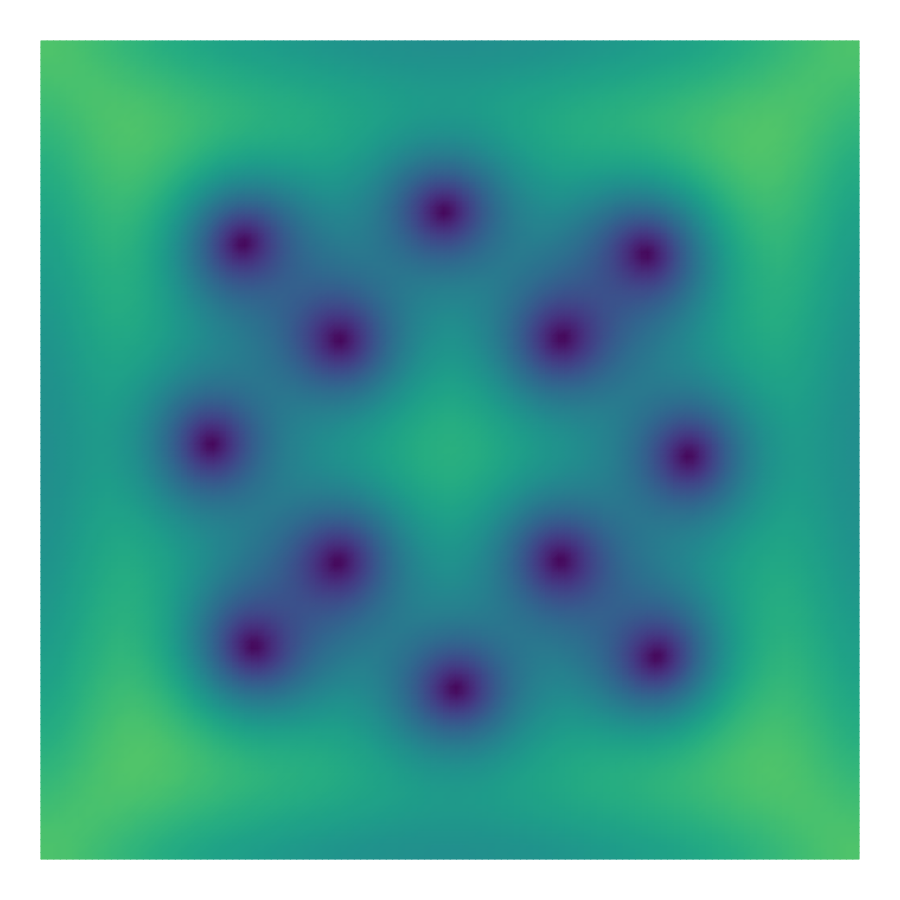}
        \end{minipage}
    \end{minipage}%
    \hspace{2pt}
    \begin{minipage}{0.035\textwidth}
        \centering
        \includegraphics[height=0.28\textheight]{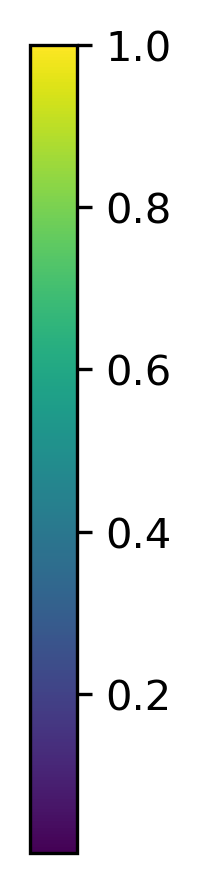}
    \end{minipage}

        \begin{minipage}{0.85\textwidth}
        \centering
        \begin{minipage}{0.19\textwidth}
            \centering
            \includegraphics[width=\linewidth]{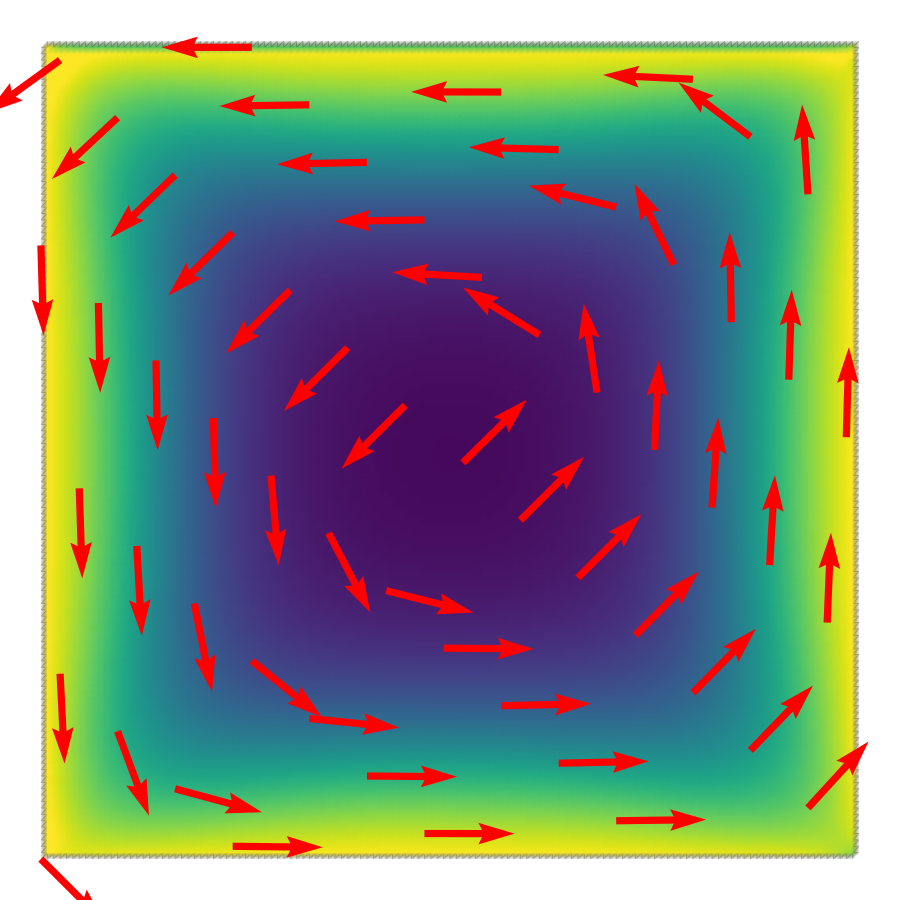}
        \end{minipage}
        \begin{minipage}{0.19\textwidth}
            \centering
            \includegraphics[width=\linewidth]{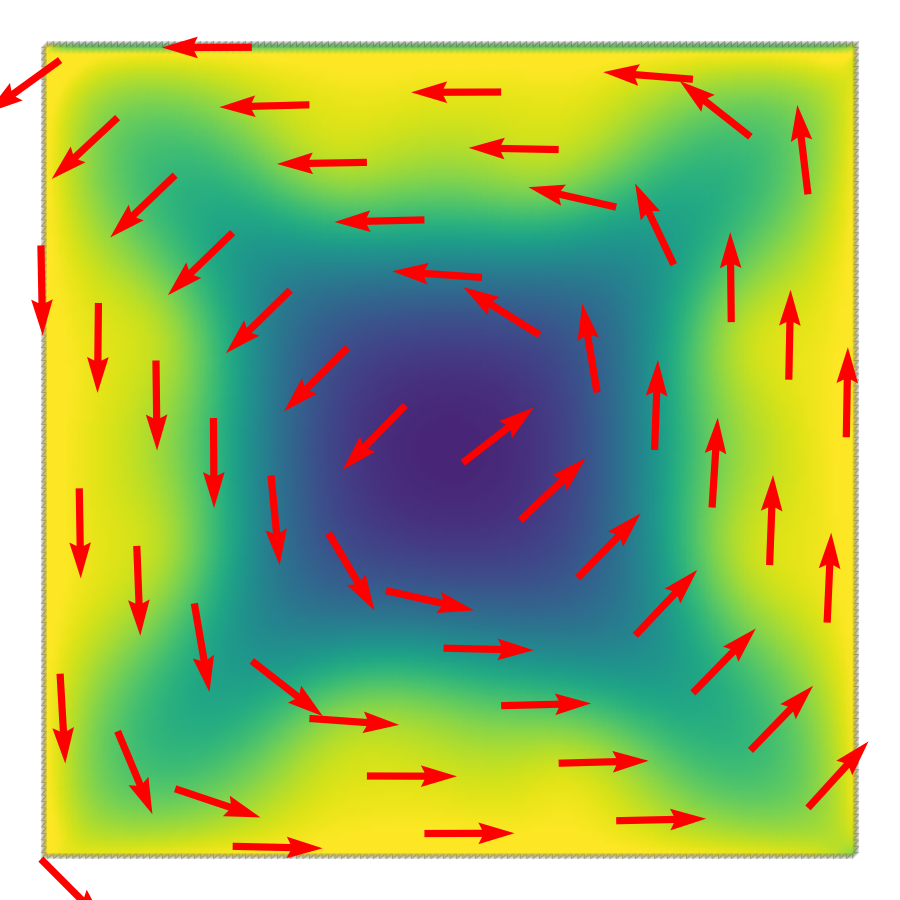}
        \end{minipage}
        \begin{minipage}{0.19\textwidth}
            \centering
            \includegraphics[width=\linewidth]{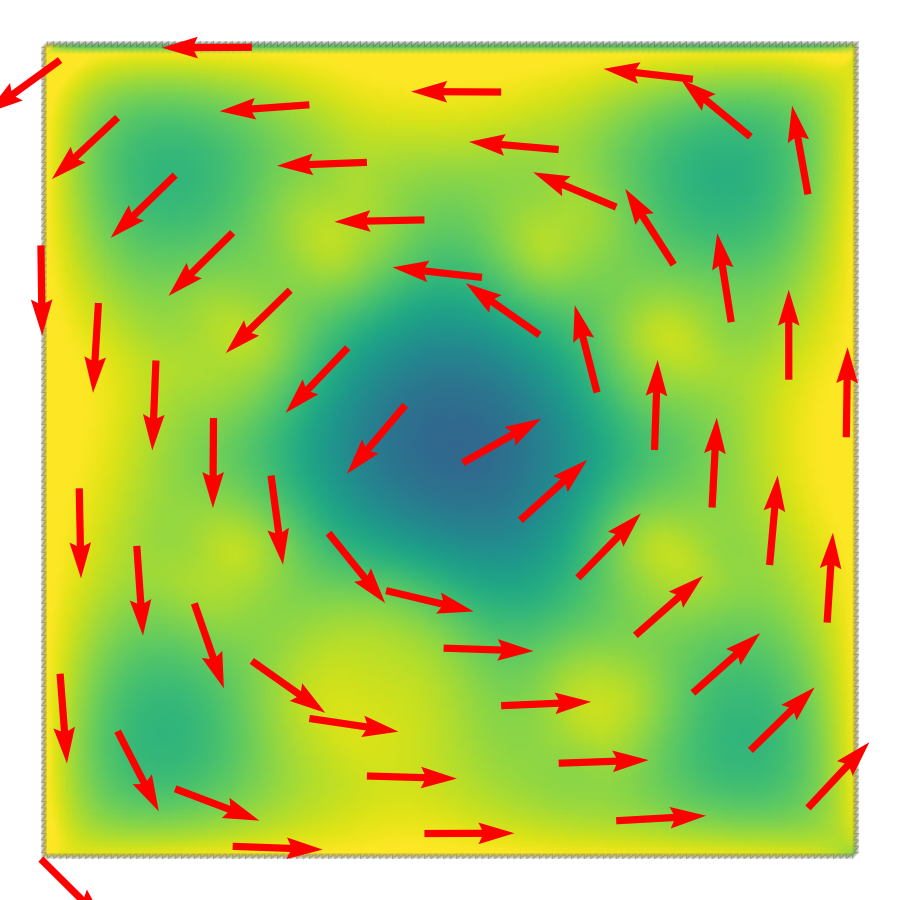}
        \end{minipage}
        \begin{minipage}{0.19\textwidth}
            \centering
            \includegraphics[width=\linewidth]{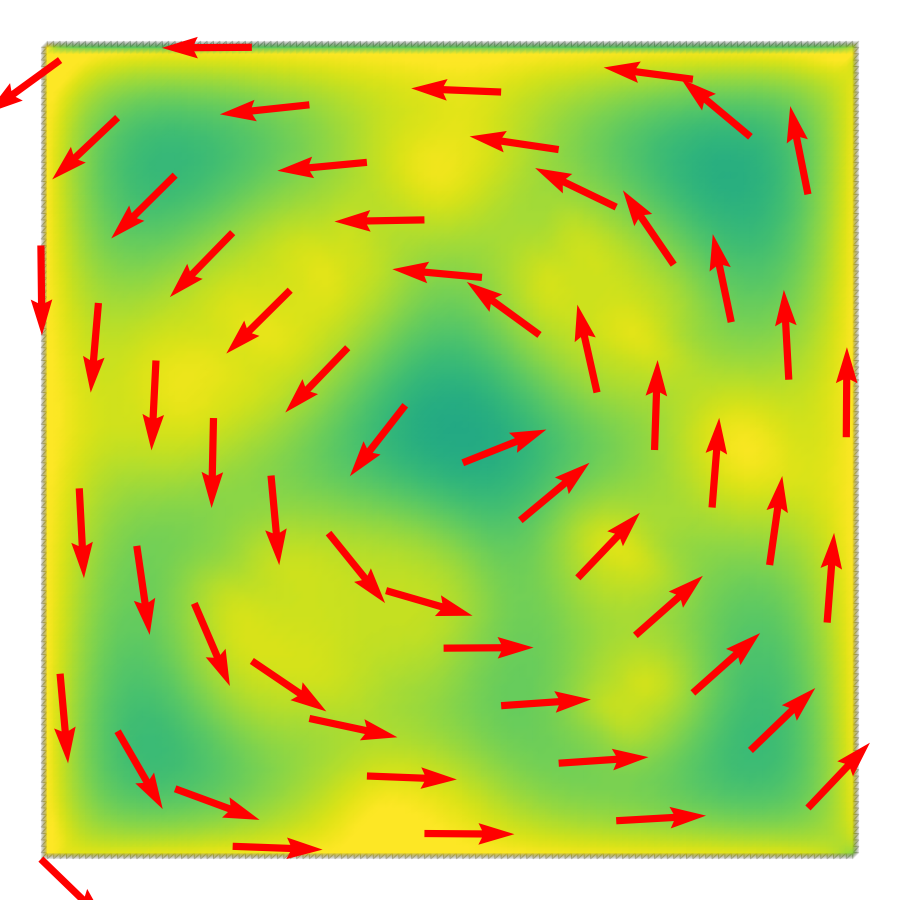}
        \end{minipage}
        \begin{minipage}{0.19\textwidth}
            \centering
            \includegraphics[width=\linewidth]{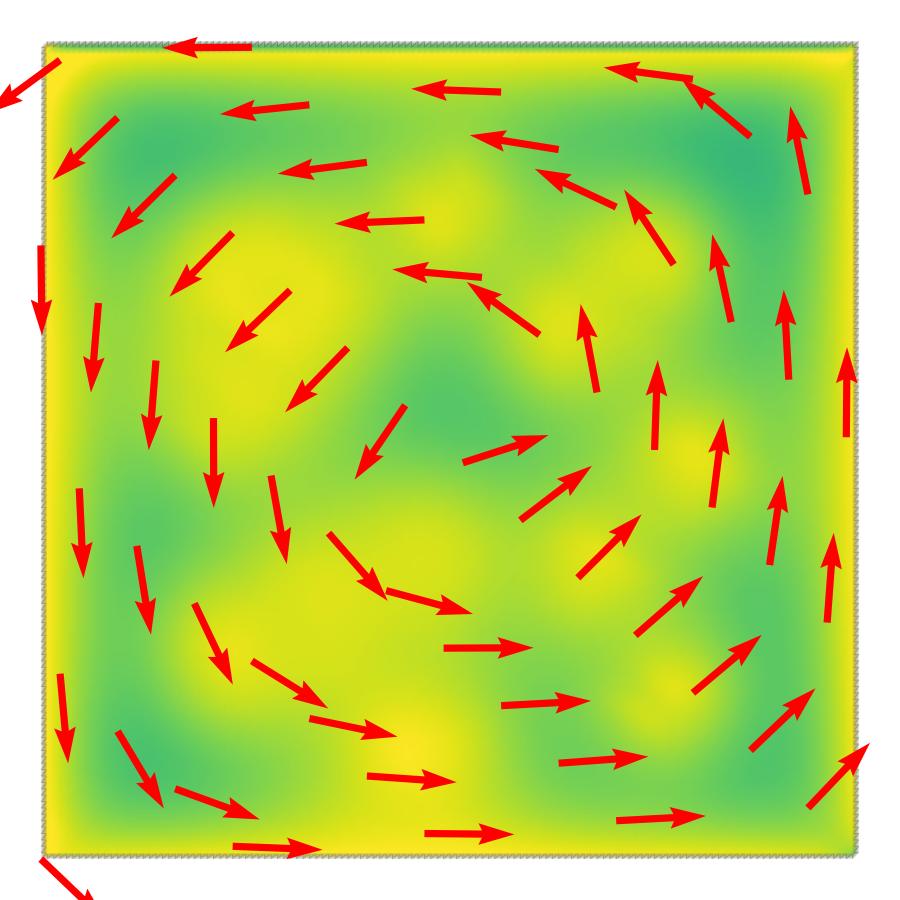}
        \end{minipage}

        \vspace{5pt}
        \begin{minipage}{0.19\textwidth}
            \centering
            \includegraphics[width=\linewidth]{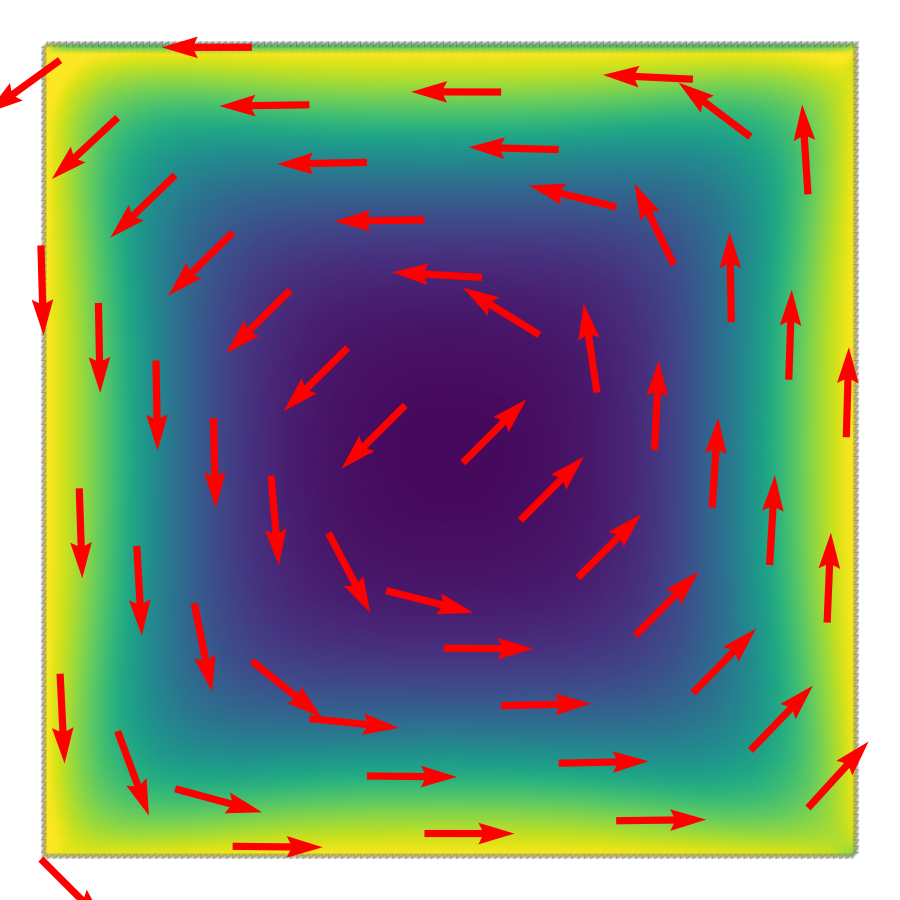}
        \end{minipage}
        \begin{minipage}{0.19\textwidth}
            \centering
            \includegraphics[width=\linewidth]{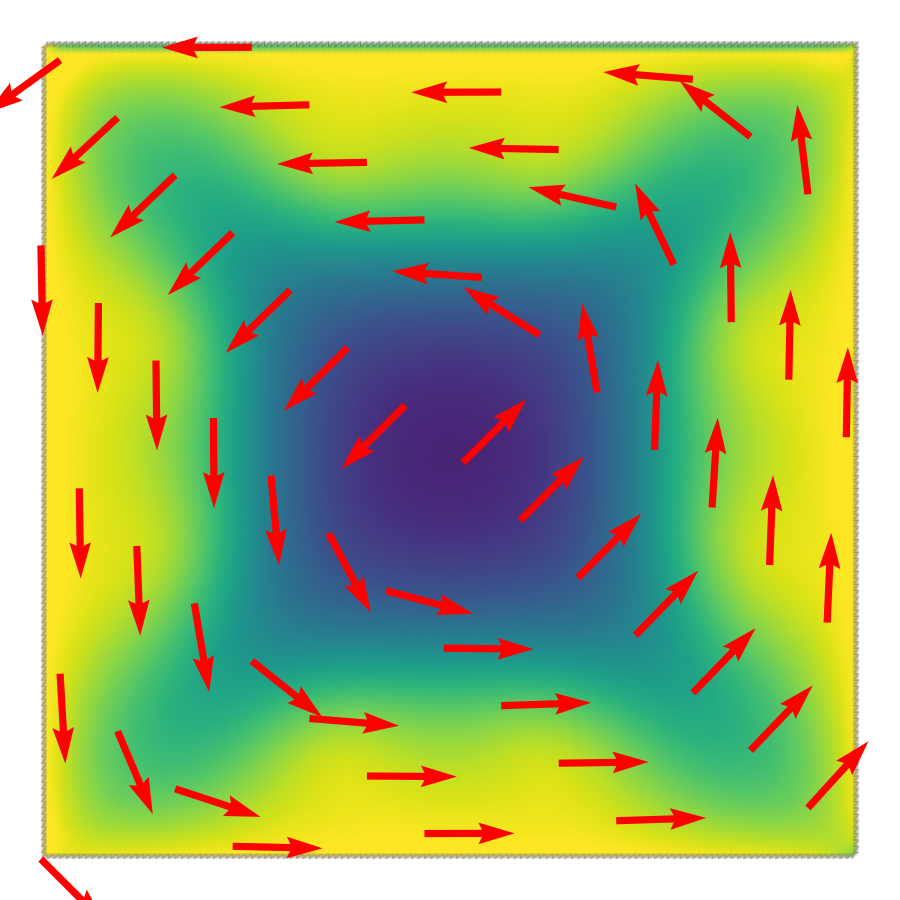}
        \end{minipage}
        \begin{minipage}{0.19\textwidth}
            \centering
            \includegraphics[width=\linewidth]{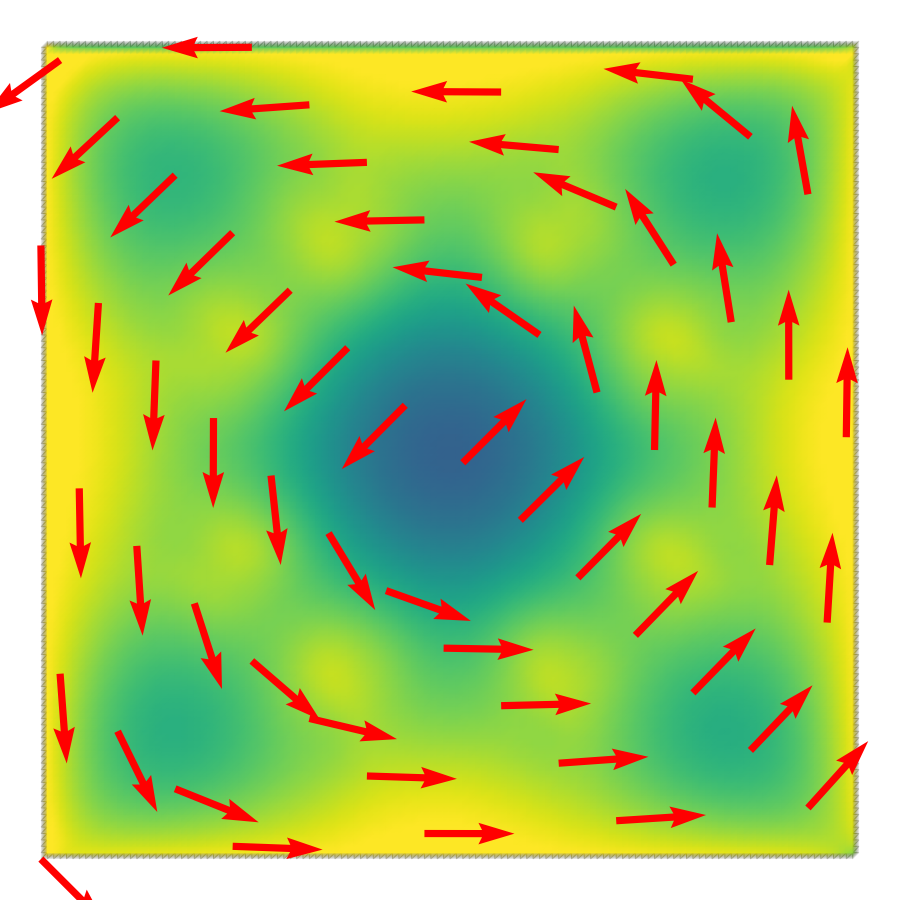}
        \end{minipage}
        \begin{minipage}{0.19\textwidth}
            \centering
            \includegraphics[width=\linewidth]{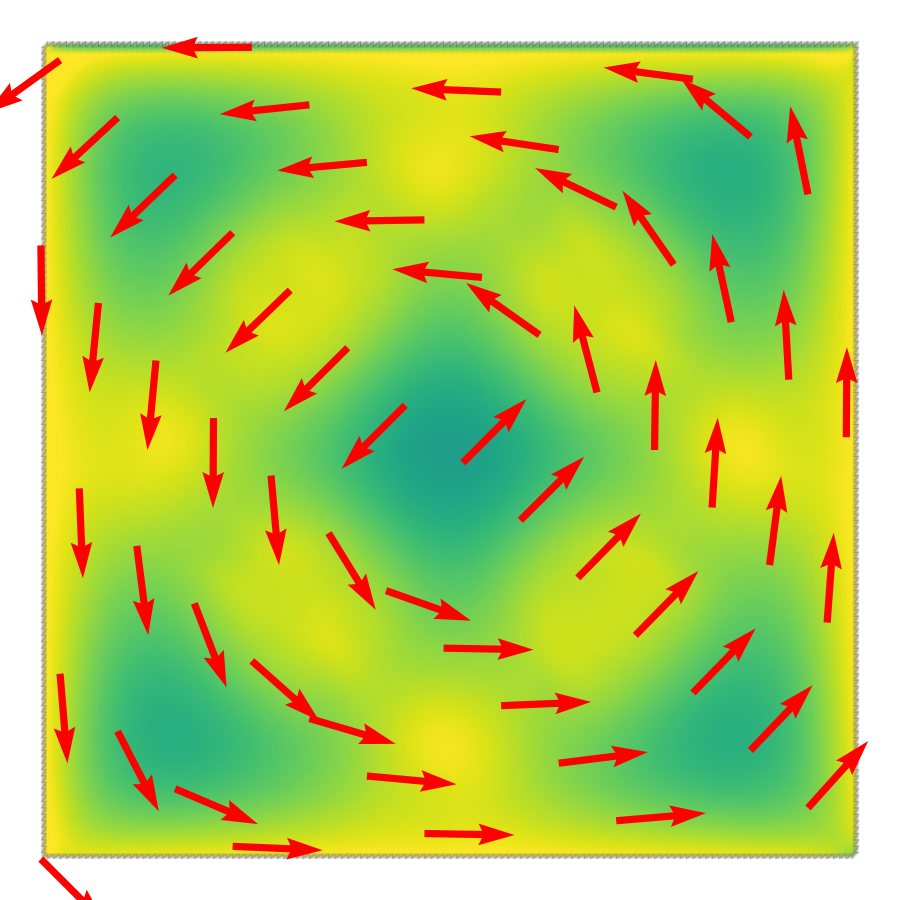}
        \end{minipage}
        \begin{minipage}{0.19\textwidth}
            \centering
            \includegraphics[width=\linewidth]{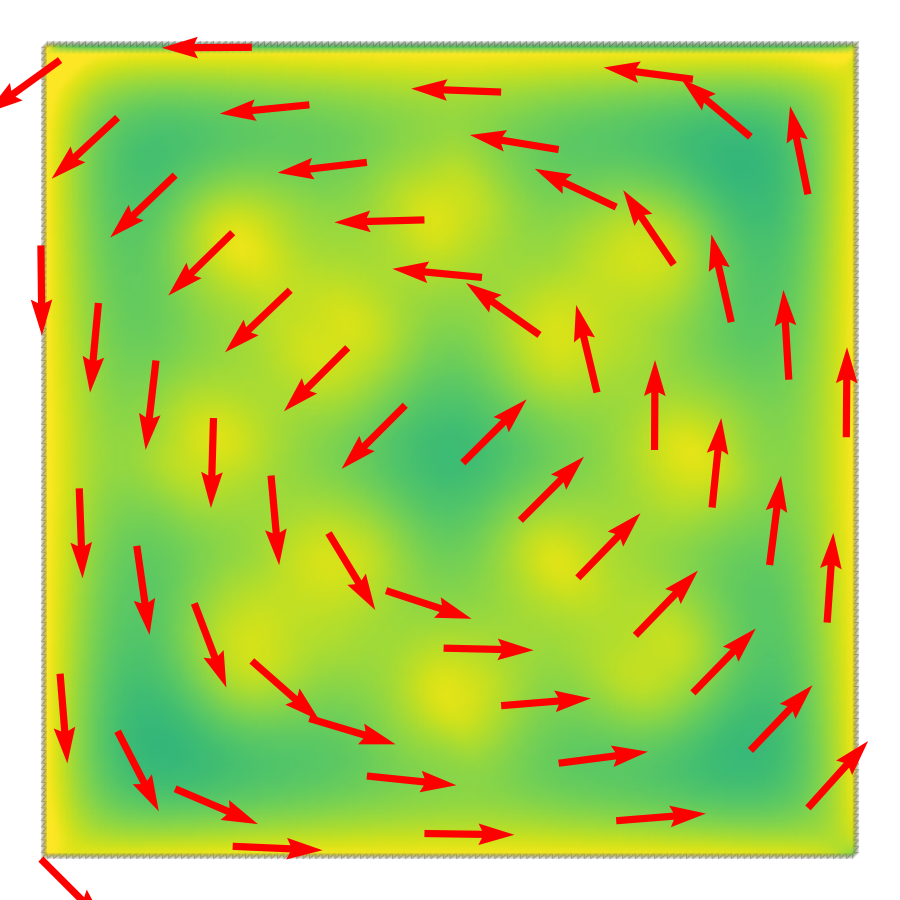}
        \end{minipage}
    \end{minipage}%
    \hspace{2pt}
    \begin{minipage}{0.035\textwidth}
        \centering
        \includegraphics[height=0.28\textheight]{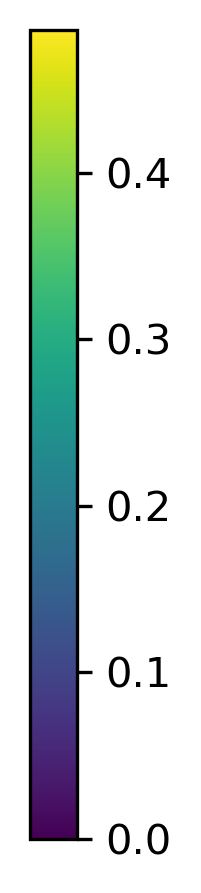}
    \end{minipage}
    
    \caption{Time evolution of vortex structures of $|\psi|$ and $curl A$ in $(-2\pi,2\pi)$ with GL constant $\kappa = 2$ under an external magnetic field along the (0,0,1) direction (t=20, 40, 60, 80, 100, and the arrows indicate the direction of the vector potential A). he stabilization parameter is $S = 2$. The singular and even rows correspond to the cases with and without the applied inhomogeneity potential, respectively.}
    \label{fig.4}
\end{figure}

The third numerical experiment (Figures~\ref{fig.3}--\ref{fig.4}) investigates the morphological evolution of magnetic vortices, comparing the cases with and without inhomogeneity. The corresponding discrete energy curves are shown in Figure~\ref{fig:ene-pair}, exhibiting monotone decay in accordance with the stability result of Theorem~\ref{thm-ene}. The simulations are performed with a time step $\tau=0.25$ and Ginzburg--Landau parameter $\kappa=2$ on domains $\Omega_k = (-k\pi,k\pi)^2$ for $k=1,2$. Both homogeneous and inhomogeneous cases are considered, with inhomogeneity $\delta$ generated randomly to highlight its influence on the structure and dynamics of magnetic flux vortices. A depinning transition can be clearly observed: for instance, from the second to the third snapshot in the first row of Figure~\ref{fig.3}, a vortex escapes from a local pinning site and begins to migrate. This transition is also reflected in the energy dissipation at around $t = 50$ shown in the left panel of Figure~\ref{fig:ene-pair}. Overall, the results demonstrate that inhomogeneity significantly alters vortex formation and distribution, leading to noticeable differences in simulations.

\begin{figure}[!htp]
    \centering
\centering

    \begin{minipage}{0.85\textwidth}
        \centering
        \begin{minipage}{0.27\textwidth}
            \includegraphics[width=\linewidth]{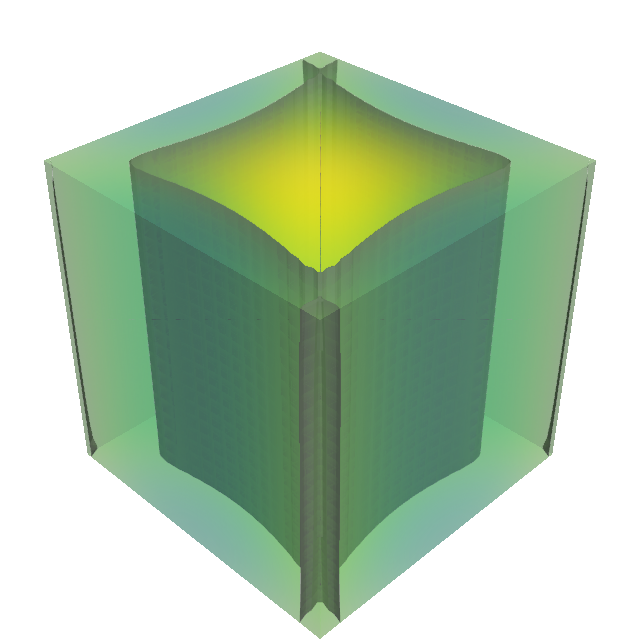}
        \end{minipage}
        \begin{minipage}{0.27\textwidth}
            \includegraphics[width=\linewidth]{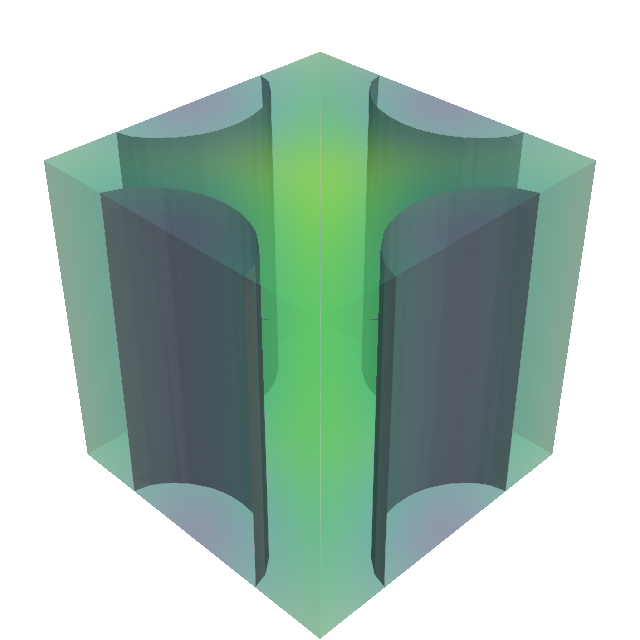}
        \end{minipage}
        \begin{minipage}{0.27\textwidth}
            \includegraphics[width=\linewidth]{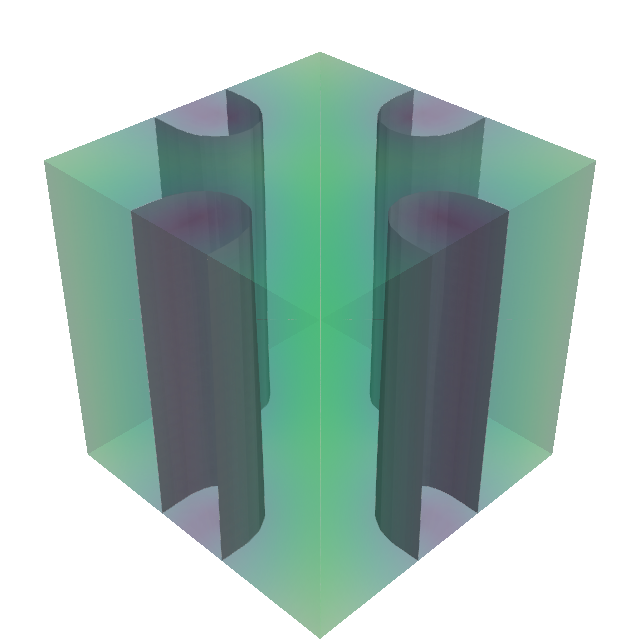}
        \end{minipage}

        \vspace{2pt}
        \begin{minipage}{0.27\textwidth}
            \includegraphics[width=\linewidth]{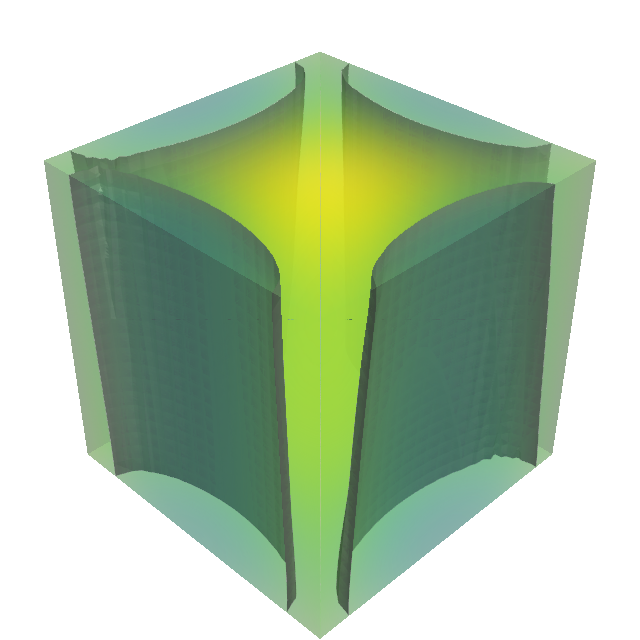}
        \end{minipage}
        \begin{minipage}{0.27\textwidth}
            \includegraphics[width=\linewidth]{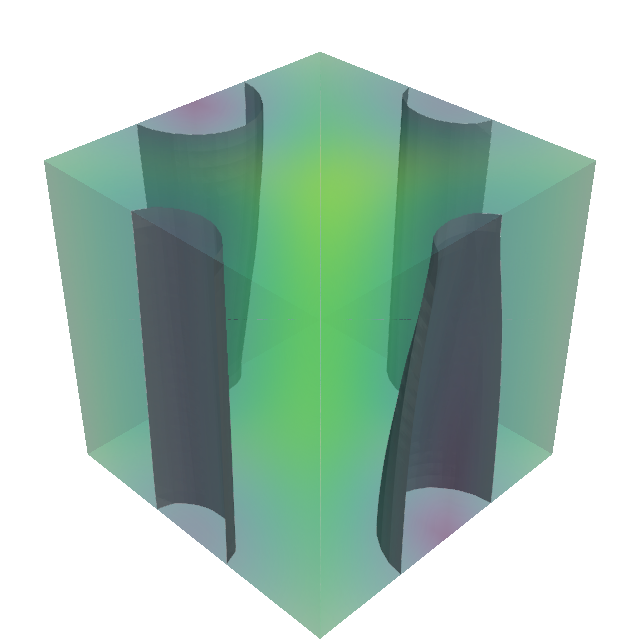}
        \end{minipage}
        \begin{minipage}{0.27\textwidth}
            \includegraphics[width=\linewidth]{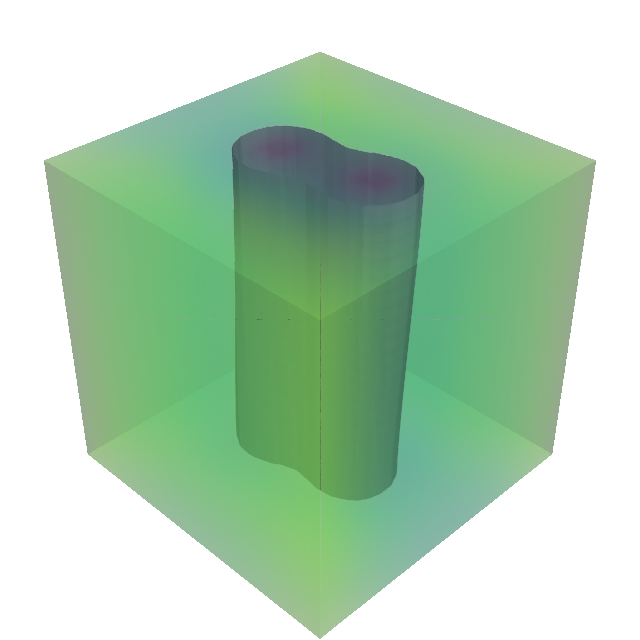}
        \end{minipage}

        \vspace{2pt}
        \begin{minipage}{0.27\textwidth}
            \includegraphics[width=\linewidth]{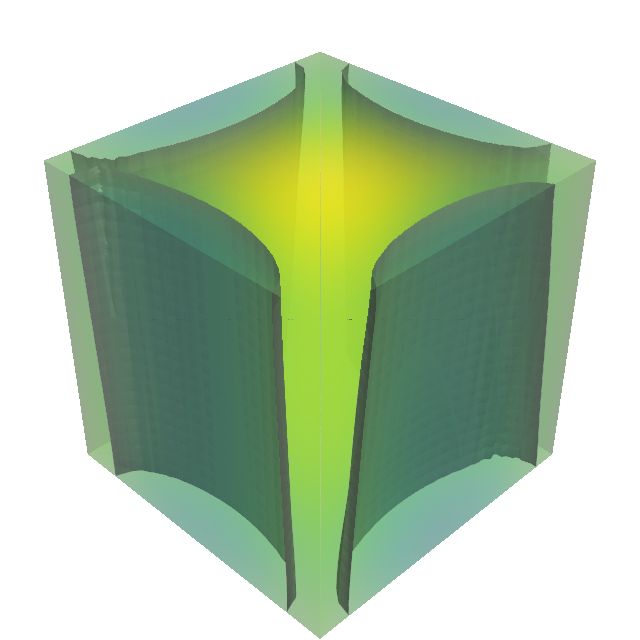}
        \end{minipage}
        \begin{minipage}{0.27\textwidth}
            \includegraphics[width=\linewidth]{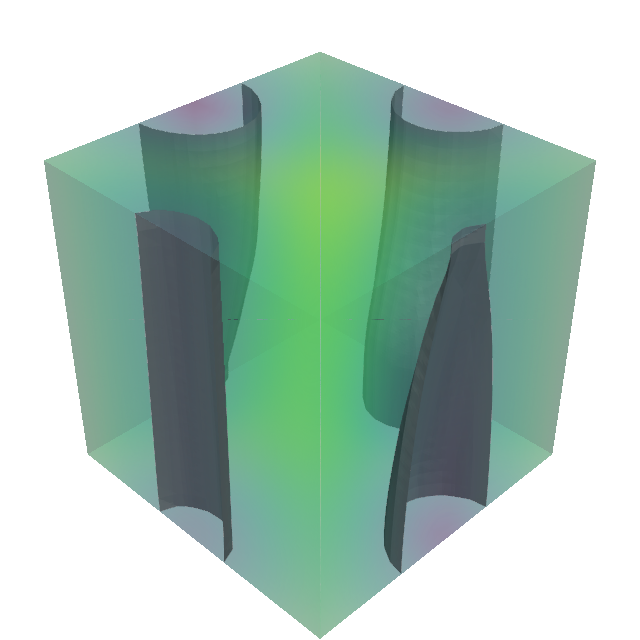}
        \end{minipage}
        \begin{minipage}{0.27\textwidth}
            \includegraphics[width=\linewidth]{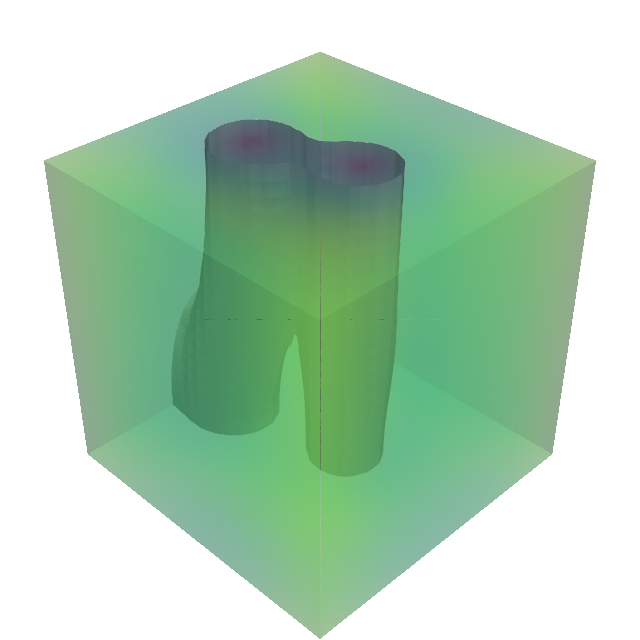}
        \end{minipage}
    \end{minipage}%
    \begin{minipage}{0.05\textwidth}
        \begin{minipage}{\linewidth}
            \centering
            \includegraphics[height=0.3\textheight]{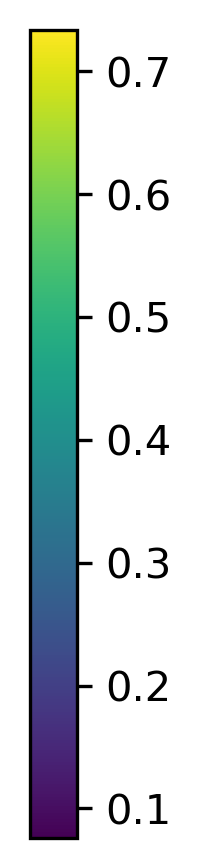}

            \includegraphics[height=0.095\textheight]{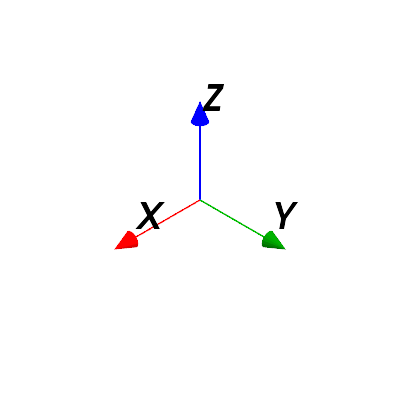}
        \end{minipage}
    \end{minipage}

\caption{
Time evolution of vortex structures in a 3D cubic domain under different applied magnetic fields.  
Top row: $\delta(x)=0$, $H=(0,\,0,\,0.5)$;  
middle row: $\delta(x)=0$, $H=(0,\,0.5\sin(\pi/36),\,0.5\cos(\pi/36))$;  
bottom row: $\delta(x)\neq0$, $H=(0,\,0.5\sin(\pi/36),\,0.5\cos(\pi/36))$.  
Contours are taken at the average of the maximum and minimum values. 
A comparison between rows 2 and 3 shows that inhomogeneity causes vortices to deviate 
from the otherwise aligned structures observed in the homogeneous case.
}
    \label{fig.5}
\end{figure}

\begin{figure}[htbp]
    \centering
    \setlength{\fboxsep}{0pt}  
    \begin{minipage}[t]{0.4\textwidth}
        \centering
        \includegraphics[width=\linewidth]{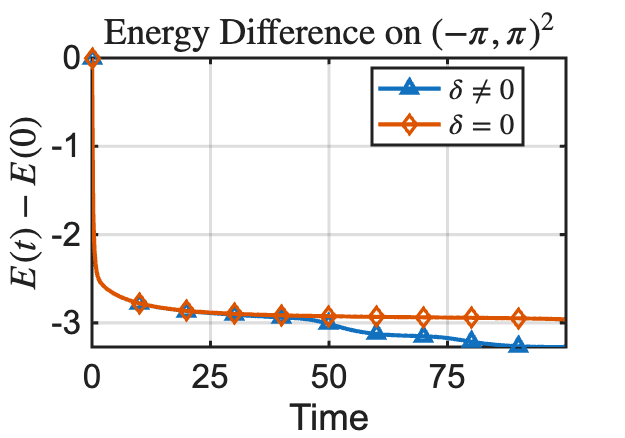}
    \end{minipage}
    \hspace{1mm} 
    \begin{minipage}[t]{0.4\textwidth}
        \centering
        \includegraphics[width=\linewidth]{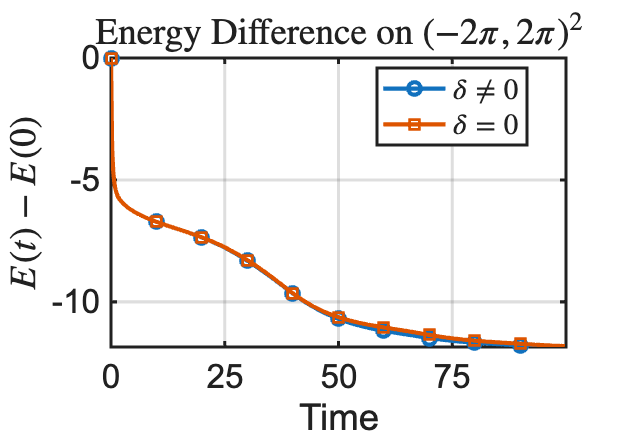}
    \end{minipage}
\caption{
Energy evolution for the simulations shown in Figures~\ref{fig.3} and \ref{fig.4}. 
The curves display the monotone decay of the discrete energy, confirming the stability result of Theorem~\ref{thm-ene}. 
They also demonstrate that depinning induced by inhomogeneity leads to a lower final energy level due to vortex dynamics.
}\label{fig:ene-pair}
    
\end{figure}

The last example shows the numerical simulation in the cubic domain $(-\pi,\pi)^3$. The first and third rows are the case without inhomogeneity, while the second and the fourth rows represent the case with a inhomogeneity, modeled by some randomly generated overlapping spheres. The simulations were performed on a 20x20x20 grid under a constant magnetic field along the $(0,0.5\sin(\pi/36),0.5\cos(\pi/36))$. The presence of inhomogeneity significantly affects the vortex dynamic, resulting in more complicated structures. Without inhomogeneity, the vortice lines tend to align approximately with the external magnetic field, forming well-structured patterns. However, when a inhomogeneity is introduced, the vortices experience some local pinning effects and deviate from the expected alignment. See Figure \ref{fig.5}.

\begin{remark}
{The numerical experiments presented here focus on validating the maximum modulus bound and energy dissipation properties of the scheme using simple test configurations. The implementation is based on FEniCS with default solvers and runs in serial; parallel execution is not yet enabled. More detailed numerical analysis and physically realistic simulations will be considered and compared in future work.}
\end{remark}

\section{Discussion and Conclusion}
\label{sec-con}
We have investigated a hybrid TDGL model for superconductors that extends the classical TDGL formulation by incorporating a nonlinear BCS term. To simulate this system reliably, we developed a structure-preserving IMEX scheme that maintains both a discrete maximum bound and energy stability. These properties ensure long-time robustness and physical consistency, which are essential for capturing vortex dynamics in extended simulations. An asymptotic analysis of the hybrid model is also included.

The scheme was tested in two- and three-dimensional numerical experiments, where it successfully reproduced vortex formation, alignment, and suppression under applied fields. The results show the importance of structure-preserving properties for reliable simulations. Future work will focus on developing asymptotically preserving numerical schemes. Such advances would further improve the accuracy and applicability of structure-preserving methods in complex superconducting systems.

\paragraph*{CReditT Author Contributions}
B. Wang: Conceptualization, Methodology, Formal analysis, Software Writing - original draft.
S. Shenoy:  Writing - review \& editing.
D. Fortino: Software
W. Hao: Methodology, Formal analysis, Writing - review \& editing.  
L.-Q. Chen: Conceptualization, Supervision.

\paragraph*{Acknowledgments} The authors would like to thank Dr. Fei Yang and Professor Yin Shi for their insightful discussions and constructive feedback that contributed to the development of this work.

\paragraph*{Data availability}
No experimental or observational data were used. The simulation code will be available upon request.

\end{document}